\documentclass[12pt]{article}                
\usepackage{graphicx, xcolor}
\usepackage{psfrag}
\usepackage{amsmath, amssymb, amsthm}
\usepackage{mathptmx}
\usepackage{esvect, bm, bbm}
\usepackage{hyperref}
\usepackage{cleveref}
\usepackage{enumitem}

\hypersetup{
	colorlinks=true,
	linkcolor={blue!80!black},
	citecolor={green!50!black},
	urlcolor={red!50!black}
}

\newcommand{\Q}{\mathbb{Q}}
\newcommand{\R}{\mathbb{R}}
\newcommand{\Rp}{\mathbb{R}_{\geq 0}}
\newcommand{\Rpp}{\mathbb{R}_{>0}}
\newcommand{\1}{\mathbbm{1}}

\newcommand{\A}{\mathcal{A}}
\newcommand{\D}{\mathcal{D}}
\newcommand{\C}{\mathsf{C}}
\newcommand{\CC}{\mathcal{C}}
\newcommand{\EE}{\mathcal{E}}

\renewcommand{\H}{\mathsf{H}}
\newcommand{\HH}{\mathcal{H}}

\newcommand{\I}{\mathcal{I}}
\newcommand{\K}{\mathcal{K}}
\renewcommand{\L}{\mathcal{L}}

\newcommand{\M}{\mathcal{M}}

\renewcommand{\O}{\Omega}
\newcommand{\p}{\varphi}

\renewcommand{\P}{\mathbb{P}}

\renewcommand{\S}{\mathcal{S}}
\newcommand{\TV}{\mathrm{TV}}
\newcommand{\td}{\widetilde}
\newcommand{\fHL}{f_{\mathrm{HL}}}

\newcommand{\abs}[1]{\lvert#1\rvert}
\newcommand{\norm}[1]{\lVert#1\rVert}
\newcommand{\Norm}[1]{{\vert\kern-0.1ex\vert\kern-0.1ex\vert#1\vert\kern-0.1ex\vert\kern-0.1ex\vert}}
\newcommand{\dNorm}[2]{{\vert\kern-0.1ex\vert\kern-0.1ex\vert#2\vert\kern-0.1ex\vert\kern-0.1ex\vert}_{#1,*}}

\newcommand{\dec}{\searrow}

\newcommand{\ud}{\, \mathrm{d}}
\DeclareMathOperator{\udo}{\mathrm{d}}

\DeclareMathOperator{\supp}{\operatorname{supp}}

\DeclareMathOperator{\sgn}{\operatorname{sgn}}
\DeclareMathOperator{\BigO}{\operatorname{\mathcal{O}}}

\newcommand{\step}[1]{{\underline{\textit {Step #1}}}}
\newcommand{\case}[1]{{\underline{\textit {Case #1}}}}

\newtheorem{theorem}{Theorem}[section]
\newtheorem{lemma}[theorem]{Lemma}
\newtheorem{proposition}[theorem]{Proposition}
\newtheorem{corollary}[theorem]{Corollary}
\newtheorem{definition}[theorem]{Definition}
\newtheorem*{acknowledgements}{Acknowledgements}

\theoremstyle{remark}

\numberwithin{equation}{section}

\newcommand{\Ll}{\left}
\newcommand{\Rr}{\right}

\begin{document}

\title{\vspace{-1cm} Infinite-dimensional Hamilton-Jacobi equations for statistical inference on sparse graphs}
\author{Tomas Dominguez\thanks{\textsc{\tiny Department of Mathematics, University of Toronto, tomas.dominguezchiozza@mail.utoronto.ca}} \and  Jean-Christophe Mourrat\thanks{\textsc{\tiny Department of Mathematics, ENS Lyon and CNRS, jean-christophe.mourrat@ens-lyon.fr}}}

\date{}
\maketitle
\vspace*{-0.7cm}

\begin{abstract}
We study the well-posedness of an infinite-dimensional Hamilton-Jacobi equation posed on the set of non-negative measures and with a monotonic non-linearity. Our results will be used in a companion work to propose a conjecture and prove partial results concerning the asymptotic mutual information in the assortative stochastic block model in the sparse regime. The equation we consider is naturally stated in terms of the Gateaux derivative of the solution, unlike previous works in which the derivative is usually of transport type. We introduce an approximating family of finite-dimensional Hamilton-Jacobi equations, and use the monotonicity of the non-linearity to show that no boundary condition needs to be prescribed to establish well-posedness. The solution to the infinite-dimensional Hamilton-Jacobi equation is then defined as the limit of these approximating solutions. In the special setting of a convex non-linearity, we also provide a Hopf-Lax variational representation of the solution.
\end{abstract}

\section{Introduction and main results}
A recent approach to describe the asymptotic free energy of a mean-field disordered system is as the solution to a Hamilton-Jacobi equation. 
Spin-glass models have typically led to infinite-dimensional Hamilton-Jacobi equations of transport type \cite{JC_upper,JC_NC,JC_parisi} while statistical inference problems have given rise to finite-dimensional Hamilton-Jacobi equations defined on closed convex cones \cite{HB_nonsymmetric, JC_HB_FinRank, HB_tensor, HB_multilayer, JC_matrix, JC_HJ}. A general well-posedness theory for the former was established in~\cite{HB_HJ} while one for the latter was developed in \cite{HB_cone}. In \cite{TD_JC} we will propose to describe the asymptotic mutual information in the sparse stochastic block model in terms of a Hamilton-Jacobi equation posed over a space of probability measures, but featuring derivatives  of ``affine'' rather than transport type. The purpose of this paper is to develop a well-posedness theory for such an infinite-dimensional Hamilton-Jacobi equation; we expect that this type of equation will appear in other mean-field problems with sparse interactions. While our setting is different, the techniques we use here draw heavily upon the arguments introduced in \cite{HB_HJ} and \cite{HB_cone}.

Let us describe the class of infinite-dimensional Hamilton-Jacobi equations that we consider. We denote by $\M_s$ the space of signed measures on $[-1,1]$,
\begin{equation}
\M_s=\big\{\mu \mid \mu \text{ is a signed measure on } [-1,1]\big\},
\end{equation}
and by $\M_+$ the cone of non-negative measures on this interval,
\begin{equation}\label{eqn: SBME M+}
\M_+=\big\{\mu \in \M_s\mid \mu \text{ is a non-negative measure}\big\}.
\end{equation}
We follow the convention that a signed measure can only take finite values, and in particular, every $\mu \in \M_+$ must have finite total mass. 
We fix a continuously differentiable function $g:[-1,1]\to \R$, and for each measure $\mu \in \M_+$, define the function $G_\mu : [-1,1]\to \R$ by
\begin{equation}
G_\mu(x)=\int_{-1}^1 g(xy)\ud \mu (y).
\end{equation}
We introduce the cone of functions
\begin{equation}\label{eqn: SBME cone of functions}
\CC_\infty=\big\{G_\mu \mid  \mu\in \M_+\big\}
\end{equation}
as well as the non-linearity $\C_\infty: \CC_\infty\to \R$ defined on this cone by
\begin{equation}\label{eqn: SBME infinite non-linearity}
\C_\infty(G_\mu)=\frac{1}{2}\int_{-1}^1 G_\mu(x)\ud \mu(x)=\frac{1}{2}\int_{-1}^1 \int_{-1}^1 g(xy)\ud \mu(y)\ud \mu(x).
\end{equation}
This non-linearity is well-defined by the Fubini-Tonelli theorem. Indeed, if $G_\mu=G_\nu$ for some measures $\mu,\nu\in \M_+$, then
\begin{equation}
\int_{-1}^1 G_\mu(x)\ud \mu(x)=\int_{-1}^1 G_\nu(x)\ud \mu(x)=\int_{-1}^1\int_{-1}^1 g(xy)\ud \mu(x)\ud \nu(x)   ,
\end{equation}
while
\begin{equation}
\int_{-1}^1 G_\nu(x)\ud \nu(x)=\int_{-1}^1 G_\mu(x)\ud \nu(x)=\int_{-1}^1 \int_{-1}^1 g(xy)\ud \mu(y)\ud \nu(x),
\end{equation}
and the symmetry of the map $(x,y)\mapsto g(xy)$ implies that these two expressions coincide.
Given a function $f:[0,\infty)\times \M_+\to \R$ and measures $\mu,\nu\in \M_+$, we denote by $D_\mu f(t,\mu;\nu)$ the Gateaux derivative of the function $f(t,\cdot)$ at the measure $\mu$ in the direction $\nu$,
\begin{equation}\label{eqn: SBME Gateaux derivative}
D_\mu f(t,\mu;\nu)=\lim_{\epsilon \to 0}\frac{f(t,\mu+\epsilon \nu)-f(t,\mu)}{\epsilon}.
\end{equation}
We will say that the Gateaux derivative of $f(t,\cdot)$ at the measure $\mu \in \M_+$ admits a density if there exists a bounded measurable function $x\mapsto D_\mu f(t,\mu,x)$ defined on the interval $[-1,1]$ with
\begin{equation}\label{eqn: SBME Gateaux derivative density}
D_\mu f(t,\mu;\nu)=\int_{-1}^1 D_\mu f(t,\mu,x)\ud \nu(x)
\end{equation}
for every measure $\nu \in \M_+$. We will often abuse notation and identify the density $D_\mu f(t,\mu,\cdot)$ with the Gateaux derivative $D_\mu f(t,\mu)$. The purpose of this paper is to establish the well-posedness of the infinite-dimensional Hamilton-Jacobi equation
\begin{equation}\label{eqn: SBME infinite HJ eqn}
\left\{
\begin{aligned}
\partial_t f(t,\mu)&=\C_\infty\big(D_\mu f(t,\mu)\big) & \text{on }& \Rpp\times \M_+,\\
f(0,\mu)&=\psi(\mu)& \text{on }& \M_+.
\end{aligned}
\right.
\end{equation}
under appropriate assumptions on the kernel $g$ and the initial condition $\psi:\M_+\to \R$. In particular, these assumptions will imply that, in a suitably weak sense, the Gateaux derivative of the solution $D_\mu f(t,\mu)$ belongs to the cone $\CC_\infty$ for all $t \ge 0$ and $\mu \in \M_+$. Before stating these assumptions and the results they lead to precisely, let us describe the general strategy we will follow.

To obtain the well-posedness of the Hamilton-Jacobi equation \eqref{eqn: SBME infinite HJ eqn}, we will project it from the infinite-dimensional space of measures $\smash{\M_+}$ to a family of finite-dimensional spaces of measures~$\smash{\M_+^{(K)}}$ with dimension monotone in some integer parameter $K\geq 1$. The well-posedness of each of these projected equations will be obtained using techniques similar to those in \cite{HB_cone}, and the limit  as $K$ tends to infinity of these projected solutions will be shown to exist using techniques similar to those in \cite{HB_HJ, JC_upper, JC_NC}; we take the limit thus obtained as the definition of the solution to the infinite-dimensional Hamilton-Jacobi equation \eqref{eqn: SBME infinite HJ eqn}. 

In previous works, derivatives of transport type were the primary focus of investigation, and it was thus natural to discretize the space of measures by restricting to measures of the form $\smash{K^{-1} \sum_{k = 1}^K \delta_{x_k}}$, only allowing the $x_k$'s to vary but keeping the weight of each atom fixed. Due to the nature of the derivatives appearing in \eqref{eqn: SBME infinite HJ eqn}, we choose instead here to define our finite-dimensional approximating space as the cone of non-negative measures supported on dyadic rationals in the interval $[-1,1]$. That is, we allow the weights to vary, provided that they remain non-negative, but keep the positions of the atoms fixed.
Given an integer $K\geq 1$, we write
\begin{equation}\label{eqn: SBME dyadic rationals}
\D_K=\Big\{k=\frac{i}{2^K} \mid -2^K\leq i<  2^K\Big\}
\end{equation}
for the set of dyadic rationals on $[-1,1]$ at scale $K$. It will be convenient to index vectors using the set of dyadic rationals, writing
$\smash{x=(x_k)_{k\in \D_K}\in \R^{\D_K}}$.
We denote the set of discrete measures supported on the dyadic rationals at scale $K$ in the interval $[-1,1]$ by
\begin{equation}
\M^{(K)}_+=\Big\{\mu\in \M_+\mid \mu=\frac{1}{\abs{\D_K}}\sum_{k\in \D_K}x_k\delta_k\text{ for some } x=(x_k)_{k\in \D_K}\in \Rp^{\D_K}\Big\}.
\end{equation}
A natural way to project a general measure $\mu \in \M_+$ onto $\M_+^{(K)}$ is via the mapping
\begin{equation}\label{eqn: SBME weights of dyadic measure}
x^{(K)}(\mu)=\big(\abs{\D_K}\mu\big[k,k+2^{-K}\big)\big)_{k\in \D_K}
\in \Rp^{\D_K}.
\end{equation}
For $\smash{\mu \in \M_+^{(K)}}$, the image of $\mu$ is simply the sequence of weights of the measure $\mu$ at each point in $\D_K$, up to multiplication by $\abs{\D_K}$. The inverse of this mapping assigns to each $\smash{x \in \Rp^{\D_K}}$ the measure
\begin{equation}
\mu^{(K)}_x=\frac{1}{\abs{\D_K}}\sum_{k\in \D_K}x_{k}\delta_k \in \M_+^{(K)}.
\end{equation}
We can use these projections to devise finite-dimensional approximations to the Hamilton-Jacobi equation \eqref{eqn: SBME infinite HJ eqn}. These will be posed on the cone $\smash{\Rp\times \Rp^{\D_K}}$. Indeed, any real-valued function $\smash{f:[0,\infty)\times \M_+^{(K)}\to \R}$ may be identified with the function 
\begin{equation}
f^{(K)}(t,x)=f\big(t,\mu_x^{(K)}\big)
\end{equation}
defined on $\smash{\Rp\times \Rp^{\D_K}}$. Moreover, the Gateaux derivative at the measure $\smash{\mu\in \M_+^{(K)}}$ may be identified with the gradient $\smash{\abs{\D_K}\nabla f^{(K)}(t,x^{(K)}(\mu))}$ by duality. Indeed, for any direction $\smash{\nu\in \M_+^{(K)}}$,
\begin{equation}\label{eqn: SBME projecting the Gateaux derivative}
D_\mu f(t,\mu; \nu)=\frac{\mathrm{d}}{\mathrm{d} \epsilon}\Big|_{\epsilon=0}f^{(K)}(t,x^{(K)}(\mu)+\epsilon x^{(K)}(\nu))=\nabla f^{(K)}(t,x^{(K)}(\mu))\cdot x^{(K)}(\nu).
\end{equation}
The additional factor of $\abs{\D_K}$ appears because $\smash{x^{(K)}(\nu)}$ has $\ell^1$-norm $\abs{\D_K}$ whenever $\nu$ is a probability measure. The corresponding initial condition becomes the function $\smash{\psi^{(K)}:\Rp^{\D_K}\to \R}$ defined by
\begin{equation}\label{eqn: SBME projected initial condition}
\psi^{(K)}(x)=\psi\big(\mu_x^{(K)}\big).
\end{equation}
The cone \eqref{eqn: SBME cone of functions} and the non-linearity \eqref{eqn: SBME infinite non-linearity} may be projected in a similar manner. We introduce the symmetric matrix
\begin{equation}\label{eqn: SBME matrix GK}
G^{(K)}=\frac{1}{\abs{\D_K}^2}\big(g(kk')\big)_{k,k'\in \D_K}\in \R^{\D_K\times \D_K},
\end{equation}
and observe that for every $\mu\in \M_+^{(K)}$ and $k\in \D_K$,
\begin{equation}
G_\mu(k)=\sum_{k'\in \D_K}g(kk')\mu(k')=\frac{1}{\abs{\D_K}}\sum_{k'\in \D_K}g(kk')x^{(K)}(\mu)_{k'}=\abs{\D_K}\big(G^{(K)}x^{(K)}(\mu)\big)_{k}.
\end{equation}
This motivates the definition of the projected cone,
\begin{equation}\label{eqn: SBME projected cone}
\CC_{K}=\Big\{G^{(K)}x^{(K)}(\mu)\in \R^{\D_K}\mid \mu\in \M_+^{(K)}\Big\}=\Big\{G^{(K)}x\in \R^{\D_K}\mid x\in \Rp^{\D_K}\Big\},
\end{equation}
and the projected non-linearity $\C_K:\CC_K\to \R$ defined by
\begin{equation}\label{eqn: SBME projected non-linearity}
\C_K\big(G^{(K)}x\big)=\frac{1}{2}G^{(K)}x\cdot x=\frac{1}{2\abs{\D_K}^2}\sum_{k,k'\in \D_K}g(kk')x_{k}x_{k'}=\C_\infty\big( G_{\mu_x^{(K)}}\big).
\end{equation}
This projected non-linearity can be shown to be well-defined using the same argument that showed the non-linearity \eqref{eqn: SBME infinite non-linearity} was well-defined. With this notation, our finite-dimensional approximation of the Hamilton-Jacobi equation \eqref{eqn: SBME infinite HJ eqn} reads
\begin{equation}\label{eqn: SBME HJ eqn on cone}
\partial_t f^{(K)}(t,x)=\C_K\big(\nabla f^{(K)}(t,x)\big) \quad\text{on}\quad \Rpp\times \Rp^{\D_K}
\end{equation}
subject to the initial condition $f^{(K)}(0,x)=\psi^{(K)}(x)$ on $\Rp^{\D_K}$. To study this equation we will introduce an appropriate extension $\smash{\H_K:\R^{\D_K}\to \R}$ of the non-linearity $\smash{\C_K}$, and instead consider the Hamilton-Jacobi equation 
\begin{equation}\label{eqn: SBME HJ eqn on Rdd H_K}
\partial_t f^{(K)}(t,x)=\H_K\big(\nabla f^{(K)}(t,x)\big) \quad\text{on}\quad \Rpp\times \Rpp^{\D_K}
\end{equation}
subject to the initial condition $f^{(K)}(0,x)=\psi^{(K)}(x)$ on $\Rp^{\D_K}$. Notice that we have used the cone $\smash{\Rpp^{\D_K}}$ as opposed to the more intuitive cone $\smash{\Rp^{\D_K}}$. As detailed below, the monotonicity of the projected non-linearity \eqref{eqn: SBME projected non-linearity} and the ideas regarding boundary conditions of Hamilton-Jacobi equations with suitable non-linearities developed in \cite{HB_cone, Crandall, Souganidis} make these two choices equivalent. In particular, it will not be necessary to endow the projected Hamilton-Jacobi equation with a boundary condition. Remembering that this Hamilton-Jacobi equation appears in the context of statistical inference makes this insight rather reassuring. Indeed, the statistical inference model does not suggest an obvious choice of boundary condition---given that we are ultimately interested in the identification of the value of the solution at a point in $\Rp \times \{0\}$, we would at least not want to use a Dirichlet boundary condition there! In earlier works, the imposition of a Neumann-type boundary condition was observed to be a workable option  \cite{JC_upper, JC_HJ, JC_NC}. In \cite{HB_cone}, it was shown that this somewhat artificial choice is not necessary,  and no boundary condition needs to be specified, because the non-linearity ``points in the right direction''.

We now state the precise assumptions that will allow us to obtain the well-posedness of the projected Hamilton-Jacobi equations and establish the convergence of their solutions. In the same spirit as \cite{HB_HJ, HB_cone, JC_upper, JC_NC}, we will need the initial conditions $\smash{\psi^{(K)}}$ and $\psi$ to satisfy a certain number of Lipschitz continuity assumptions. Given an integer $d\geq 1$, we introduce the normalized-$\smash{\ell^1}$ and normalized-$\smash{\ell^{1,*}}$ norms, defined for every $x, y \in \R^d$ by
\begin{equation}\label{eqn: SBME normalized norms}
\Norm{x}_1=\frac{1}{d}\sum_{k=1}^d \abs{x_k} \quad \text{and} \quad \dNorm{1}{y}=\max_{k\leq d}d\abs{y_k}.
\end{equation}
The underlying dimension $d\geq 1$ will be kept implicit but will always be clear from the context. The normalized-$\smash{\ell^1}$ norm is meant to measure elements of $\smash{\R^{\D_k}}$ with a scaling that is consistent with our identification of this space with the space of measures $\smash{\M^{(K)}_+}$. The normalized-$\smash{\ell^{1,*}}$ norm serves to measure elements of the dual space, and is defined so that the H\"older-type inequality $x \cdot y \le \Norm{x}_1 \, \dNorm{1}{y}$ is valid.

The key continuity assumption on the projected initial condition $\smash{\psi^{(K)}}$ that will make it possible to establish the well-posedness of the projected Hamilton-Jacobi equations will be Lipschitz continuity with respect to the normalized-$\smash{\ell^1}$ norm. Another way to encode this property is to require the initial condition $\psi:\M_+\to \R$ to be Lipschitz continuous with respect to the total variation distance on~$\M_+$,
\begin{equation}\label{eqn: SBME TV metric}
\TV(\mu,\nu)=\sup\big\{\abs{\mu(A)-\nu(A)}\mid A \text{ is a measurable subset of } [-1,1]\big\}.
\end{equation}
The normalized-$\ell^{1,*}$ norm will play its part when discussing the Lipschitz continuity of the projected non-linearity \eqref{eqn: SBME projected non-linearity}. To determine the convergence of the projected solutions, it will be important to assume that the initial condition $\psi:\M_+\to \R$ is Lipschitz continuous with respect to the Wasserstein distance on the set of probability measures $\Pr[-1,1]$,
\begin{equation}\label{eqn: SBME Wasserstein}
W(\P,\Q)=\sup\bigg\{\Big\lvert \int_{-1}^1 h(x)\ud \P(x)-\int_{-1}^1 h(x)\ud \Q(x)\Big\rvert \mid \norm{h}_{\text{Lip}}\leq 1\bigg\}.
\end{equation}
Here $\norm{\cdot}_{\mathrm{Lip}}$ denotes the Lipschitz semi-norm
\begin{equation}
\norm{h}_{\mathrm{Lip}}=\sup_{x\neq x'\in [-1,1]}\frac{\abs{h(x)-h(x')}}{\abs{x-x'}}
\end{equation}
defined on the space of functions $h:[-1,1]\to \R$. The final assumption on the initial condition will ensure that, in a sense to be made precise, the solution to the projected Hamilton-Jacobi equation has a bounded gradient close to the projected cone $\CC_K$ defined in \eqref{eqn: SBME projected cone}. It would of course be more convenient to assume that the gradient really belongs to $\CC_K$, rather than only being close to it, but unlike in earlier works, this stronger property does not hold in the context of the main application we have in mind in \cite{TD_JC}. To impose the boundedness of the gradient, fix $a>0$, and for each integer $K\geq 1$ introduce the closed convex set
\begin{equation}
\K_{a,K}=\Big\{G^{(K)}x\in \R^{\D_K}\mid x\in \Rp^{\D_K} \text{ and } \Norm{x}_1\leq a\Big\}\subset \CC_K.
\end{equation}
Given a closed convex set $\smash{\K\subset \R^d}$, write
\begin{equation}\label{eqn: SBME enlarged set}
\K'=\K+B_{d^{-1/2}}(0)
\end{equation}
for the neighborhood of radius $\smash{d^{-1/2}}$ around $\K$ in the normalized-$\ell^{1,*}$ norm. Here 
\begin{equation}\label{eqn: SBME l1* ball}
B_r(x)=\big\{x'\in \R^d\mid \dNorm{1}{x'-x}\leq r\big\}
\end{equation}
denotes the closed ball of radius $r>0$ centered around $x\in \R^d$ relative to the normalized-$\ell^{1,*}$ norm. We will say that a Lipschitz continuous function $\smash{h:\Rp^d\to \R}$ has its gradient in $\K$ if
\begin{equation}\label{eqn: SBME gradient in cone}
\nabla h\in L^\infty\big(\Rp^d;\K\big).
\end{equation}
Recall that a Lipschitz continuous function is differentiable almost everywhere by Rademacher's theorem (see Theorem 6 in Chapter 5.8 of \cite{Evans}), so the spatial gradient $\nabla h$ is well-defined as an element of $L^\infty$, and the condition \eqref{eqn: SBME gradient in cone} requires that this object take values in $\K$ almost everywhere. A non-differential criterion for the gradient of a Lipschitz continuous function to lie in a closed convex set is given in \Cref{SBME gradient in convex set}, and will be used frequently throughout the paper. As will be shown below, assuming that the initial condition has its gradient in $\smash{\K_{a,K}'}$ suffices to ensure that the gradient of the solution remains in this set at all times. Notice that this is insufficient to be able to evaluate the non-linearity $\smash{\C_K}$ at the gradient of the solution; however, under suitable Lipschitz continuity properties of the extension $\H_K$, it ensures that the projected Hamilton-Jacobi equation \eqref{eqn: SBME HJ eqn on Rdd H_K} should be an adequate replacement for the Hamilton-Jacobi equation \eqref{eqn: SBME HJ eqn on cone}. In particular, it justifies defining the solution to the infinite-dimensional Hamilton-Jacobi \eqref{eqn: SBME infinite HJ eqn} as the limit of the solutions to the projected Hamilton-Jacobi equation \eqref{eqn: SBME HJ eqn on Rdd H_K}.
Besides some smoothness, the only constraint we will impose on the kernel $g:[-1,1]\to \R$ is that it be strictly positive. Among other things, this assumption ensures that a non-negative measure $\mu \in \M_+$ cannot have a large total mass unless the function~$G_\mu$ takes large values. In summary, the assumptions on the kernel $g:[-1,1]\to \R$ and the initial condition $\psi:\M_+\to \R$ required for the validity of our main results are the following.
\begin{enumerate}[label = \textbf{H\arabic*}]
\item The kernel $g:[-1,1]\to\R$ is continuously differentiable and bounded away from zero by some positive constant $m>0$, 
\begin{equation}\label{eqn: SBME g lower bound}
g(x)\ge m.
\end{equation}\label{SBME H g}
\item The initial condition $\psi:\M_+\to \R$ is Lipschitz continuous with respect to the total variation distance \eqref{eqn: SBME TV metric},
\begin{equation}
\abs{\psi(\mu)-\psi(\nu)}\leq \norm{\psi}_{\mathrm{Lip},\TV}\TV(\mu,\nu)
\end{equation}
for all measures $\nu,\mu\in \M_+$. \label{SBME H TV}
\item There exists $a>0$ such that the initial condition $\psi:\M_+\to \R$ has the property that each of the projected initial conditions \eqref{eqn: SBME projected initial condition} has its gradient in the set $\smash{\K_{a,K}'}$,
\begin{equation}
\nabla \psi^{(K)}\in L^\infty\big(\Rp^d;\K_{a,K}'\big).
\end{equation}\label{SBME H K'}
\item The initial condition $\psi:\Pr[-1,1]\to \R$ is Lipschitz continuous with respect to the Wasserstein distance \eqref{eqn: SBME Wasserstein},
\begin{equation}
\abs{\psi(\P)-\psi(\Q)}\leq \norm{\psi}_{\mathrm{Lip},W} W(\P,\Q)
\end{equation}
for all probability measures $\P,\Q\in \Pr[-1,1]$. \label{SBME H W}
\end{enumerate}
Observe that the hypothesis \eqref{SBME H TV} on the initial condition implies that the projected initial conditions \eqref{eqn: SBME projected initial condition} are Lipschitz continuous with respect to the normalized-$\smash{\ell^1}$ norm,
\begin{equation}\label{eqn: SBME projected non-linearity Lipschitz}
\big\lvert \psi^{(K)}(x)-\psi^{(K)}(x')\big\rvert\leq \norm{\psi}_{\mathrm{Lip},\TV}\TV\big(\mu^{(K)}_x,\mu^{(K)}_{x'}\big)\leq \norm{\psi}_{\mathrm{Lip},\TV}\Norm{x-x'}_1.
\end{equation}

With these assumptions at hand, it is natural to wonder why we cannot simply invoke the main result in \cite{HB_cone} to obtain the well-posedness of the projected Hamilton-Jacobi equation \eqref{eqn: SBME HJ eqn on Rdd H_K}. The setting proposed in \cite{HB_cone} is that of a Hamilton-Jacobi equation posed on a cone $\CC$ and with a non-linearity that is defined over the cone $\CC$ as well; the key assumption to establish well-posedness is that the non-linearity and the initial condition have their gradients in the cone $\CC$. In our context, the non-linearity is initially only well-defined on the cone $\CC_\infty$, or $\CC_K$ for the projected equations, and we must make sure that the gradient of the solution remains in this space. This suggests that we try to use the results in \cite{HB_cone} with $\smash{\CC = \CC_\infty}$, or $\smash{\CC_K}$ for the projected equations. However, our problem, say for the projected equations, is naturally posed over $\smash{\Rp^{\D_K}}$ rather than $\CC_K$, and moreover, the gradient of the non-linearity that appears in our setting is not in $\CC_K$, although it is in $\smash{\Rp^{\D_K}}$. 
To make matters more complicated, the gradient of the finite-dimensional initial condition, and therefore also of the solution, does not quite belong to $\CC_K$, although it is in the closed convex set $\smash{\K_{a,K}'}$. Despite all this, we will show that the somewhat richer geometry of our problem can be dealt with using arguments that are similar to those in \cite{HB_cone}.

We now describe the structure of these arguments in more detail, and state our main results. We will first show that for any $R>0$, it is possible to define a non-linearity $\smash{\H_{K,R}:\R^{\D_K}\to \R}$ which agrees with the projected non-linearity $\C_K$ on a large enough ball $\smash{\CC_K\cap B_{R}(0)}$, and is uniformly Lipschitz continuous. We will then obtain the well-posedness of the projected Hamilton-Jacobi equation
\begin{equation}\label{eqn: SBME projected HJ eqn}
\partial_t f^{(K)}(t,x)=\H_{K,R}\big(\nabla f^{(K)}(t,x)\big) \quad \text{on}\quad \Rpp\times \Rpp^{\D_K}
\end{equation}
subject to the initial condition $\smash{f^{(K)}(0,x)=\psi^{(K)}(x)}$ on $\smash{\Rp^{\D_K}}$. 
Finally, we will show that the solutions to these projected Hamilton-Jacobi equations admit a limit as $K$ tends to infinity. We will verify that this limit does not depend on the choice of the extension $\H_{K,R}$, provided that $R$ is chosen sufficiently large, and define it to be the solution to the infinite-dimensional Hamilton-Jacobi equation \eqref{eqn: SBME infinite HJ eqn}.

To state our main well-posedness results, we introduce additional notation. 
Given functions $h:\Rp^d\to \R$ and $\smash{u:[0,\infty)\times \Rp^d\to \R}$, we define the semi-norms
\begin{equation}
\Norm{h}_{\mathrm{Lip},1}=\sup_{x\neq x'\in\Rp^d}\frac{\abs{h(x)-h(x')}}{\Norm{x-x'}_1} \quad \text{and} \quad [u]_{0}=\sup_{\substack{t>0\\ x\in \Rp^d}}\frac{\abs{u(t,x)-u(0,x)}}{t},
\end{equation}
and introduce the space of functions with Lipschitz initial condition that grow at most linearly in time,
\begin{equation}\label{eqn: SBME L space}
\mathfrak{L}=\big\{u:[0,\infty)\times \Rp^d\to \R\mid u(0,\cdot) \text{ is Lipschitz continuous and } [u]_0<\infty\big\},
\end{equation}
as well as its subset of uniformly Lipschitz continuous functions,
\begin{equation}\label{eqn: SBME L unif space}
\mathfrak{L}_{\mathrm{unif}}=\Big\{u\in \mathfrak{L}\mid \sup_{t\geq 0}\Norm{u(t,\cdot)}_{\mathrm{Lip},1}<\infty\Big\}.
\end{equation}
The main well-posedness results for the projected Hamilton-Jacobi equation \eqref{eqn: SBME projected HJ eqn} and the infinite-dimensional Hamilton-Jacobi equation \eqref{eqn: SBME infinite HJ eqn} now read as follows.

\begin{theorem}\label{SBME WP of projecetd HJ eqn}
Under assumptions \eqref{SBME H g}-\eqref{SBME H K'}, the projected Hamilton-Jacobi equation \eqref{eqn: SBME projected HJ eqn} with $\smash{R>0}$ admits a unique viscosity solution $\smash{f^{(K)}_R\in \mathfrak{L}_{\mathrm{unif}}}$ subject to the initial condition $\smash{\psi^{(K)}}$. Moreover, $\smash{f^{(K)}_R}$ has its gradient in the set $\smash{\K_{a,K}'}$ and satisfies the Lipschitz bound 
\begin{equation}
\sup_{t>0}\Norm{f_R^{(K)}(t,\cdot)}_{\mathrm{Lip},1}=\Norm{\psi^{(K)}}_{\mathrm{Lip},1}\leq \norm{\psi}_{\mathrm{Lip},\TV}.
\end{equation}
\end{theorem}

\begin{theorem}\label{SBME infinite HJ eqn WP}
Suppose \eqref{SBME H g}-\eqref{SBME H W}, and given an integer $K\geq 1$ and a real number $\smash{R>\norm{\psi}_{\mathrm{Lip},\TV}}$, denote by $\smash{f_R^{(K)}\in \mathfrak{L}_{\mathrm{unif}}}$ the unique viscosity solution to the Hamilton-Jacobi equation \eqref{eqn: SBME projected HJ eqn}  constructed in \Cref{SBME WP of projecetd HJ eqn}. For every $t\geq 0$ and every measure $\mu \in \M_+$, the limit
\begin{equation}\label{eqn: SBME infinite HJ eqn WP}
f(t,\mu)=\lim_{K\to \infty}f_R^{(K)}(t,x^{(K)}(\mu))
\end{equation}
exists, is finite and is independent of $R$. The value of this limit is defined to be the solution to the infinite-dimensional Hamilton-Jacobi equation \eqref{eqn: SBME infinite HJ eqn}.
\end{theorem}
Solutions to \eqref{eqn: SBME infinite HJ eqn} satisfy a comparison principle, since a comparison principle also holds for solutions to the projected Hamilton-Jacobi equation \eqref{eqn: SBME projected HJ eqn}, by Corollary~\ref{SBME comparison principle corollary on Rpd}.

As is apparent, and similarly to \cite{JC_upper, JC_NC}, we content ourselves here with identifying the solution to \eqref{eqn: SBME infinite HJ eqn} as the limit of our finite-dimensional approximations. This will suffice for our purposes, and we leave open the question of providing a more intrinsic characterization of the solution to~\eqref{eqn: SBME infinite HJ eqn}, as was achieved in \cite{HB_HJ} in a related context.

In addition to these well-posedness results, we also obtain a Hopf-Lax variational representation for the solution to the infinite-dimensional Hamilton-Jacobi equation \eqref{eqn: SBME infinite HJ eqn} in the case when the non-linearity $\C_\infty$ is convex. Hopf-Lax formulas for related problems have been explored in \cite{HB_Fenchel, HB_tensor, HB_HJ, HB_cone}. In \cite{TD_JC}, this variational representation will allow us to verify that, in the disassortative regime, our conjectured asymptotic mutual information for the sparse stochastic block model coincides with the value of the asymptotic mutual information established in \cite{CKPZ}. The convexity condition on $\C_\infty$ boils down to the requirement that the mapping $(x,y) \mapsto g(xy)$ be non-negative definite, and can be phrased as follows.
\begin{enumerate}[label = \textbf{H\arabic*}]
\setcounter{enumi}{4}
\item The kernel $g:[-1,1]\to \R$ satisfies the property
\begin{equation}\label{eqn: SBME g non-negative definite}
\int_{-1}^1 \int_{-1}^1g(xy) \ud \mu(x)\ud \mu(y)\geq 0
\end{equation}
for every signed measure $\mu\in \M_s$.\label{SBME H conv}
\end{enumerate}

\begin{theorem}\label{SBME infinite Hopf-Lax}
If \eqref{SBME H g}-\eqref{SBME H conv} hold, then the unique solution $\smash{f:[0,\infty)\times \M_+\to \R}$ to the infinite-dimensional Hamilton-Jacobi equation \eqref{eqn: SBME infinite HJ eqn} constructed in \Cref{SBME infinite HJ eqn WP} admits the Hopf-Lax variational representation
\begin{equation}
\label{eqn: infinite Hopf-Lax}
f(t,\mu)=\sup_{\nu\in \M_+}\bigg\{\psi(\mu+t\nu)-\frac{t}{2}\int_{-1}^1 G_\nu(y)\ud \nu(y)\bigg\}
\end{equation}
for every $t>0$ and $\mu \in \M_+$. Moreover, the supremum in \eqref{eqn: infinite Hopf-Lax} is achieved at some $\nu^* \in \M_+$, and whenever the initial condition $\psi$ admits a Gateaux derivative at the measure $\mu + t\nu^*$ with a density $x\mapsto D_\mu\psi(\mu+t\nu^*,x)$ belonging to the cone $\CC_\infty$, we have
\begin{equation}
    G_{\nu^*}=D_\mu\psi(\mu+t\nu^*,\cdot).
\end{equation}
\end{theorem}

For the purposes of our companion work \cite{TD_JC}, it will also be important to identify solutions to equations of the form \eqref{eqn: SBME infinite HJ eqn} with a kernel $g$ that does not satisfy the positivity assumption \eqref{SBME H g}. The idea will be to introduce a new kernel which satisfies \eqref{SBME H g} by translating $g$, and to deduce the well-posedness of the equation with kernel $g$ from the well-posedness of the equation with the translated kernel. For this strategy to work, we will replace the assumption \eqref{SBME H K'} on the initial condition by a stronger assumption which we now describe.

For every $a\in \R$ introduce the set of measures with mass $a$,
\begin{equation}
\M_{a,+}=\big\{\mu \in \M_+\mid \mu[-1,1]=a\big\},
\end{equation}
as well as the set of functions
\begin{equation}
\CC_{a,\infty} = \big\{G_\mu \mid \mu \in \M_{a,+}\big\}.
\end{equation}
The assumption \eqref{SBME H K'} on the initial condition will essentially be replaced by the assumption that its Gateaux derivative lies in the set $\smash{\CC_{a,\infty}}$ for some $a\in \R$. As before, it will be convenient to state this as an assumption on the projected initial conditions \eqref{eqn: SBME projected initial condition}. For every integer $K\geq 1$ introduce the set of projected measures with mass $a$,
\begin{equation}  
\label{e.def.ma+}
\M_{a,+}^{(K)} = \bigg\{ \mu \in \M_+^{(K)} \mid \mu[1,1]=a\bigg\},
\end{equation}
and write
\begin{equation}  
\label{e.def.CCa}
\K_{=a,K} = \Big\{ G^{(K)}x\in \R^{\D_K} \mid x\in \Rp^{\D_K} \text{ and } \Norm{x}_1=a \Big\}
\end{equation}
for its associated set of functions. We replace \eqref{SBME H K'} by the following stronger assumption.

\begin{enumerate}[label = \textbf{H3'}]
\setcounter{enumi}{5}
\item 
\label{SBME H a}
There exists $a>0$ such that the initial condition $\smash{\psi:\M_+\to \R}$ has the property that each of the projected initial conditions \eqref{eqn: SBME projected initial condition} has its gradient in the set $\K_{=a,K}'$,
\begin{equation}
\nabla \psi^{(K)}\in L^\infty\big(\Rp^d;\K_{=a,K}'\big).
\end{equation}
\end{enumerate}

Formal calculations now suggest a way to modify the solution to the infinite-dimensional Hamilton-Jacobi equation \eqref{eqn: SBME infinite HJ eqn} if the kernel $g$ is not assumed to satisfy \eqref{SBME H g} but is translated by a large enough constant so that it becomes positive. Given a continuously differentiable kernel $g:[-1,1]\to \R$, fix $b\in \R$ such that the modified kernel
\begin{equation}
\label{e.def.tdg}
\td g_b (z) = g(z) + b
\end{equation}
is strictly positive. For every $\mu \in \M_+$ define the modified function $\smash{\widetilde G_{b,\mu}}:[-1,1]\to \R$,
\begin{equation}
\widetilde{G}_{b,\mu}(x)=\int_{-1}^1 \td g_{b}(xy)\ud \mu(y),
\end{equation}
the modified cone of functions,
\begin{equation}
\widetilde{\CC}_{b,\infty}=\big\{\widetilde{G}_{b,\mu} \mid \mu\in \M_+\big\},
\end{equation}
and the modified non-linearity $\widetilde{\C}_{b,\infty} : \widetilde{\CC}_{b,\infty} \to \R$,
\begin{equation}
\widetilde{\C}_{b,\infty}\big(\widetilde{G}_{b,\mu}\big)=\frac{1}{2}\int_{-1}^1 \widetilde{G}_{b,\mu}(x)\ud\mu(x)=\frac{1}{2}\int_{-1}^1 \int_{-1}^1 \td g_b(xy)\ud \mu(y)\ud \mu(x).
\end{equation}
Notice that the additional constant $b$ in $\td g_b$ induces a shift in the expression above that depends only on the total mass of the measure $\mu$. This suggests that, under assumption \eqref{SBME H a}, if $\td f_{b}$ is a solution to the infinite-dimensional Hamilton-Jacobi equation
\begin{equation}\label{eqn: SBME infinite HJ eqn with tildes}
\left\{
\begin{aligned}
\partial_t \td f(t,\mu)&=\td \C_{b,\infty}\big(D_\mu \td f(t,\mu)\big) & \text{on }& \Rpp\times \M_+,\\
\td f(0,\mu)&=\td \psi_{b}(\mu)& \text{on }& \M_+,
\end{aligned}
\right.
\end{equation}
for the initial condition $\td \psi_{b}:\M_+\to \R$ defined by
\begin{equation}\label{e.def.tdpsi}
\td \psi_{b}(\mu)=\psi(\mu)+ab\int_{-1}^1 \ud \mu,
\end{equation}
then the function
\begin{equation}\label{e.def.f.or.tdf}
f_{b}(t,\mu)=\td f_{b}(t,\mu)-ab\int_{-1}^1 \ud \mu-\frac{a^2bt}{2}
\end{equation}
should be a solution to the infinite-dimensional Hamilton-Jacobi equation \eqref{eqn: SBME infinite HJ eqn}. We omit the dependence of $\smash{\td f_b}$, $\smash{f_b}$ and $\smash{\td \psi_b}$ on $a$ since this constant is given to us and fixed by \eqref{SBME H a}. The following result renders this construction precise and ensures that it is independent of the choice of $b$.

\begin{theorem}\label{SBME general infinite HJ eqn WP}
Fix a continuously differentiable kernel $g:[-1,1]\to \R$ and assume that \eqref{SBME H TV}, \eqref{SBME H a}, and \eqref{SBME H W} hold. Let $b \in \R$ be such that the function $\smash{\td g_b}$ defined in \eqref{e.def.tdg} is positive on $[-1,1]$, let $\smash{\td \psi_{b}}$ be defined by \eqref{e.def.tdpsi}, and let $\smash{\td f_{b}}$ be the solution to the infinite-dimensional Hamilton-Jacobi equation \eqref{eqn: SBME infinite HJ eqn with tildes} constructed in \Cref{SBME infinite HJ eqn WP}. The function $\smash{f_{b}}$ given by \eqref{e.def.f.or.tdf} does not depend on the choice of $b \in \R$, and it is defined to be the solution to the infinite-dimensional Hamilton-Jacobi equation~\eqref{eqn: SBME infinite HJ eqn}. 
\end{theorem}

Combining this well-posedness result with the Hopf-Lax representation formula in \Cref{SBME infinite Hopf-Lax} shows that under the additional assumption \eqref{SBME H conv}, the function \eqref{e.def.f.or.tdf} admits a variational representation.

\begin{theorem}\label{SBME general infinite Hopf-Lax}
Fix a continuously differentiable kernel $g:[-1,1]\to \R$ satisfying \eqref{SBME H TV}, \eqref{SBME H a} and \eqref{SBME H W}. Suppose that there exists $b\in \R$ such that the translated kernel $\td g_b$ in \eqref{e.def.tdg} is strictly positive on $[-1,1]$ and satisfies \eqref{SBME H conv}. Suppose moreover that for every $\smash{\mu \in \M_+}$, the initial condition $\psi$ admits a Gateaux derivative with density $\smash{x\mapsto D_\mu\psi(\mu,x)}$ belonging to the set $\smash{\CC_{a,\infty}}$. Then, the unique solution $\smash{f:[0,\infty)\times \M_+\to \R}$ to the infinite-dimensional Hamilton-Jacobi equation \eqref{eqn: SBME infinite HJ eqn} constructed in \Cref{SBME general infinite HJ eqn WP} admits the Hopf-Lax variational representation
\begin{equation}
\label{eqn: infinite Hopf-Lax restricted}
f(t,\mu)=\sup_{\nu\in \M_{a,+}}\bigg\{\psi(\mu+t \nu)-\frac{t}{2}\int_{-1}^1 G_\nu(y)\ud \nu(y)\bigg\}
\end{equation}
for every $t>0$ and $\mu \in \M_+$. Moreover, the supremum in \eqref{eqn: infinite Hopf-Lax restricted} is achieved at some $\nu^* \in \M_{a,+}$ with
\begin{equation}
 G_{\nu^*}=D_\mu\psi(\mu+t \nu^*,\cdot).
\end{equation}
\end{theorem}

We briefly review related works on Hamilton-Jacobi equations in infinite dimensions. The study of equations posed on infinite-dimensional Banach spaces was initiated in \cite{cl1,cl2,cl3}. The assumptions imposed on the Banach space preclude the possibility to apply the results presented there to the space of bounded measures. The existence of solutions is obtained via a connection with differential games. An example is also given in which solutions to natural  finite-dimensional approximations fail to converge to the solution of the infinite-dimensional equation. We do not expect this phenomenon to occur for the problem we consider in this paper, and in any case, our definition of the solution as the limit of finite-dimensional approximations is the one we make use of in our companion work \cite{TD_JC}. Moreover, for the equations of transport type appearing in the context of mean-field spin glasses, it was shown in \cite{HB_HJ} that finite-dimensional approximations do converge to the intrinsic viscosity solution of the infinite-dimensional equation. 

Equations that are posed over a space of probability measures, or more general metric spaces, have been considered in a number of works including \cite{amb14, car10, carqui, carsou, fenkat, fenkur, gan08, ganswi}. These works revolve around equations involving derivatives of transport type for probability measures over~$\R^d$. Since transportation of mass over $\R^d$ can be carried without limit, questions of boundary conditions do not arise there, unlike in the more recent works \cite{HB_HJ, JC_upper, JC_NC} already cited above in which probability measures over $\Rp^d$ or the space of non-negative definite matrices are considered. We are not aware of previous works considering equations that involve derivatives of ``affine'' type, as we do here. In this context, the natural ``movements'' are different from those appearing for the transport geometry, and our additional constraint that we must deal with non-negative measures is the source of the necessity to address boundary issues.

We close this section with a brief outline of the paper. In \Cref{SBME section HJ on cone}, we introduce a non-decreasing and uniformly Lipschitz continuous non-linearity $\smash{\H_{K,R}}$ which agrees with the projected non-linearity \eqref{eqn: SBME projected non-linearity} on the intersection between the projected cone \eqref{eqn: SBME projected cone} and a large enough ball, and we define the appropriate notion of solution to the projected Hamilton-Jacobi equation \eqref{eqn: SBME projected HJ eqn}. The definition of the function $\smash{\H_{K,R}}$ is inspired by Proposition 6.8 in \cite{JC_upper} and Lemma 2.5 in \cite{HB_cone}. We then leverage the well-posedness results established in \Cref{SBME app HJ eqn on half-space} to prove \Cref{SBME WP of projecetd HJ eqn}. In Section~\ref{SBME section convergence}, with the well-posedness of the projected Hamilton-Jacobi equations \eqref{eqn: SBME projected HJ eqn} at hand, we modify the arguments in Section 3.2 of \cite{JC_NC} and Section 3.3 of \cite{HB_HJ} to obtain the convergence of solutions as described in \Cref{SBME infinite HJ eqn WP}. In \Cref{SBME section Hopf-Lax} we proceed as in Section 6 of \cite{HB_cone} to obtain an approximate Hopf-Lax variational representation for the solution to the projected Hamilton-Jacobi equation~\eqref{eqn: SBME projected HJ eqn}. By taking an appropriate limit in this variational formula, \Cref{SBME infinite Hopf-Lax} is established in \Cref{SBME section Hopf Lax limit}. \Cref{SBME section general HJ WP} is devoted to the proof of \Cref{SBME general infinite HJ eqn WP} and \Cref{SBME general infinite Hopf-Lax}. So as to not disrupt the flow of the paper, the main technical arguments required to establish the well-posedness of the projected Hamilton-Jacobi equation \eqref{eqn: SBME projected HJ eqn} have been postponed to \Cref{SBME app HJ eqn on half-space}. All the results in this appendix appear in \cite{HB_cone} in some form but have been reproduced here for the reader's convenience. It is also worth pointing out that we treat a slightly different setting to the one in \cite{HB_cone}. Indeed, our initial condition has its gradient close to a convex set and not in it, and our non-linearity is monotonic with respect to a different convex set, as discussed below \eqref{eqn: SBME projected non-linearity Lipschitz}. \Cref{SBME app background} reviews a fundamental duality theorem from convex analysis, establishes a non-differential criterion for a Lipschitz function to have its gradient in a closed convex set and provides a refresher on the basic properties of semi-continuous functions, which play a role in \Cref{SBME app HJ eqn on half-space} when running Perron's argument for the existence of solutions.

\begin{acknowledgements}
We would like to warmly thank Hong-Bin Chen and Jiaming Xia for sharing a preliminary version of \cite{HB_cone} with us. Their work considerably simplified our task, allowing us in particular to discard the much more complicated approach we had originally envisaged.
\end{acknowledgements}

\section{Construction of finite-dimensional approximations}\label{SBME section HJ on cone}

In this section, we prove Theorem~\ref{SBME WP of projecetd HJ eqn}, and in particular establish the well-posedness of the projected Hamilton-Jacobi equation \eqref{eqn: SBME projected HJ eqn}. Throughout the paper, we say that a function $\smash{h:\R^d\to \R}$ is non-decreasing if for every $\smash{x,x'\in \R^d}$ with $\smash{x'-x\in \Rp^d}$, we have $h(x')-h(x)\geq 0$. More generally, given a closed convex cone $\smash{\CC\subset \R^d}$, and denoting its dual cone by $\CC^*$ (see \eqref{e.def.dual.cone} for the definition), we say that a function $\smash{h:\R^d\to \R}$ is $\smash{\CC^*}$-non-decreasing if for every $\smash{x,x'\in \R^d}$ with $x'-x\in \CC^*$, we have $h(x')-h(x)\geq 0$. In the case when $h$ is Lipschitz continuous, this is equivalent to the requirement that $h$ have its gradient in $\CC$ (see \Cref{SBME gradient in convex set}). To alleviate notation and strive for generality, fix an integer dimension $d\geq 1$ and a symmetric matrix $G\in \R^{d\times d}$ for which there exist positive constants $m,M>0$ with
\begin{equation}\label{eqn: SBME min and max of G}
\frac{m}{d^2}\leq G_{kk'}\leq \frac{M}{d^2}
\end{equation}
for all $1\leq k,k'\leq d$.  Consider the cone
\begin{equation}\label{eqn: SBME cone}
\CC=\big\{Gx\in \R^d\mid x\in \Rp^d\big\},
\end{equation}
the closed convex set
\begin{equation}
\K_a=\big\{Gx\in \R^d\mid x\in \Rp^d \text{ and } \Norm{x}_1\leq a\big\},
\end{equation}
and the non-linearity $\C:\CC\to \R$ defined by
\begin{equation}\label{eqn: SBME non-linearity on cone}
\C(Gx)=\frac{1}{2}Gx\cdot x.
\end{equation}
This mapping is well-defined, for the same reason as that explained below \eqref{eqn: SBME infinite non-linearity}. Recall the definition of the enlarged set \eqref{eqn: SBME enlarged set} and of the normalized-$\smash{\ell^1}$ and normalized-$\smash{\ell^{1,*}}$ norms in \eqref{eqn: SBME normalized norms}. The first important result of this section will be the definition of a uniformly Lipschitz continuous and non-decreasing non-linearity $\smash{\H_R:\R^d\to \R}$ which agrees with $\C$ on the intersection of the cone $\CC$ and a large enough ball.  We will then obtain the well-posedness of the Hamilton-Jacobi equation associated with this non-linearity,
\begin{equation}\label{eqn: SBME HJ eqn on Rpp}
\partial_t f(t,x)=\H_R\big(\nabla f(t,x)\big) \quad  \text{on} \quad \Rpp\times \Rpp^{d},
\end{equation}
subject to a Lipschitz continuous initial condition $\psi:\Rp^d\to \R$ with
\begin{equation}\label{eqn: SBME initial condition Lipschitz assumption}
\abs{\psi(x)-\psi(y)}\leq \Norm{\psi}_{\mathrm{Lip},1}\Norm{x-y}_1 \quad \text{and} \quad \nabla \psi\in L^\infty(\Rp^d;\K_a').
\end{equation}
By \eqref{SBME H K'} and \eqref{eqn: SBME projected non-linearity Lipschitz}, the Hamilton-Jacobi equation \eqref{eqn: SBME HJ eqn on Rpp} corresponds to the projected Hamilton-Jacobi equation \eqref{eqn: SBME projected HJ eqn} for the choices $d=\abs{\D_K}$, $G=G^{(K)}$ and $\psi=\psi^{(K)}$. \Cref{SBME WP of projecetd HJ eqn} will therefore be an immediate consequence of the main well-posedness result of this section. 

To establish the well-posedness of the Hamilton-Jacobi equation \eqref{eqn: SBME HJ eqn on Rpp} using the results in \Cref{SBME app HJ eqn on half-space}, it will be important that the extended non-linearity $\smash{\H_R:\R^d\to \R}$ be Lipschitz continuous and non-decreasing. Let us start by verifying that these properties are satisfied locally on the cone by the original non-linearity \eqref{eqn: SBME non-linearity on cone}. It will be convenient to note that
\begin{equation}\label{eqn: SBME x bounded by Gx}
\Norm{x}_1\leq \frac{1}{m}\dNorm{1}{Gx}
\end{equation}
for all $x\in \Rp^d$.

\begin{lemma}\label{SBME non-linearity locally Lipschitz}
The non-linearity \eqref{eqn: SBME non-linearity on cone} is locally Lipschitz continuous with respect to the normalized-$\smash{\ell^{1,*}}$ norm,
\begin{equation}
\abs{\C(y)-\C(y')}\leq \frac{1}{m}\big(\dNorm{1}{y}+\dNorm{1}{y'}\big)\dNorm{1}{y-y'}
\end{equation}
for all $y,y'\in \CC$.
\end{lemma}

\begin{proof}
Fix $y,y'\in \CC$ with $y=Gx$ and $y'=Gx'$ for some $x,x'\in \Rp^d$. The symmetry of $G$ and the Cauchy-Schwarz inequality imply that
$$\abs{\C(y)-\C(y')}\leq  \abs{G(x-x')\cdot x}+\abs{G(x-x')\cdot x'}\leq\big( \Norm{x}_1+\Norm{x'}_1\big)\Norm{y-y'}_{1,*}.$$
It follows by \eqref{eqn: SBME x bounded by Gx} that
$$\abs{\C(y)-\C(y')}\leq \frac{1}{m}\big(\dNorm{1}{y}+\dNorm{1}{y'}\big)\dNorm{1}{y-y'}.$$
This completes the proof.
\end{proof}

\begin{lemma}\label{SBME non-linearity non-decreasing}
The non-linearity \eqref{eqn: SBME non-linearity on cone} is non-decreasing.
\end{lemma}

\begin{proof}
Fix $y,y'\in \CC$ with $y\leq y'$ (by this we mean that $y'-y \in \Rp^d$), and let $x,x'\in \Rp^d$ be such that $y=Gx$ and $y'=Gx'$. Observe that
$$2\C(y)=Gx\cdot x=y\cdot x\leq y'\cdot x=Gx'\cdot x=Gx\cdot x'=x'\cdot y\leq x'\cdot y'=Gx'\cdot x'=2\C(y').$$
This completes the proof.
\end{proof}

Extending the non-linearity \eqref{eqn: SBME non-linearity on cone} to $\smash{\R^d}$ while preserving these two key properties requires some care. For each $R>0$, we will define a non-decreasing function $\H_R:\R^d\to \R$ which is uniformly Lipschitz continuous with respect to the normalized-$\smash{\ell^{1,*}}$ norm and agrees with the non-linearity \eqref{eqn: SBME non-linearity on cone} on the intersection of the cone \eqref{eqn: SBME cone} and the ball $\smash{B_R=B_R(0)}$ defined in \eqref{eqn: SBME l1* ball}. The definition of this extension is inspired by Proposition 6.8 in \cite{JC_upper} and Lemma 2.5 in \cite{HB_cone}.

\begin{proposition}\label{SBME extending the non-linearity}
For every $R>0$, there exists a non-decreasing non-linearity $\H_R:\R^d\to \R$ which agrees with $\C$ on $\CC\cap B_R$ and satisfies the Lipschitz continuity property
\begin{equation}\label{eqn: SBME extending the non-linearity}
\abs{\H_R(y)-\H_R(y')}\leq \frac{8RM}{m^2}\dNorm{1}{y-y'}
\end{equation}
for all $y,y'\in \R^d$.
\end{proposition}

\begin{proof}
The proof proceeds in two steps: first we regularize $\C$ by defining a non-decreasing and uniformly Lipschitz continuous function which agrees with $\C$ on $\smash{\CC\cap B_R}$, and then we extend this regularization to $\smash{\R^d}$.
\\
\step{1: regularizing $\C$.}
\\
By \Cref{SBME non-linearity locally Lipschitz}, the non-linearity \eqref{eqn: SBME non-linearity on cone} satisfies the Lipschitz bound
$$\abs{\C(y)-\C(y')}\leq \frac{4R}{m}\dNorm{1}{y-y'}$$
for all $y,y'\in \CC\cap B_{2R}$. With this in mind, let $L=\frac{4R}{m}$, and define the regularized non-linearity $\smash{\widetilde{\C}_R:\CC\to \R}$ by
$$\widetilde{\C}_R(y)=\begin{cases}
\max\Big(\C(y), \C(0)+2L\big(\dNorm{1}{y}-R\big)\Big) & \text{if } y\in \CC\cap B_{2R},\\
\C(0)+2L\big(\dNorm{1}{y}-R\big) & \text{if } y\in \CC\setminus B_{2R}.
\end{cases}$$
To see that $\widetilde{\C}_R$ agrees with $\C$ on $\CC\cap B_R$, observe that for any $y\in \CC\cap B_R$,
$$\C(0)+2L\big(\dNorm{1}{y}-R\big)\leq \C(0)=0\leq \C(y),$$
where the last inequality uses the non-negativity of the components of $G$. It will also be convenient to note that by Lipschitz continuity of $\C$ on $\CC\cap B_{2R}$,
$$\C(0)+2L\big(\dNorm{1}{y}-R\big)=\C(0)+2LR= \C(0)+L\dNorm{1}{y}\geq \C(y)$$
for any $\smash{y\in \CC\cap \partial B_{2R}}$. This shows that $\smash{\widetilde{\C}_R}$ is continuous.
To establish the non-decreasingness of $\smash{\widetilde{\C}_R}$, fix $y,y'\in \CC$ with $y\leq y'$. If $\smash{y,y'\in B_{2R}}$, then the non-decreasingness of $\C$ in \Cref{SBME non-linearity non-decreasing} implies that $\C(y)\leq \C(y')$. Combining this with the fact that $\smash{\dNorm{1}{y}\leq \dNorm{1}{y'}}$ reveals that $\smash{\widetilde{\C}_R(y)\leq \widetilde{\C}_R(y')}$. On the other hand, if $\smash{y\in B_{2R}}$ and $\smash{y'\in \CC\setminus B_{2R}}$, then
$$\C(y)\leq \C(0)+L\dNorm{1}{y}\leq \C(0)+L\dNorm{1}{y'}+L\big(\dNorm{1}{y'}-2R\big)=\widetilde{\C}_R(y')$$
and
$$\C(0)+2L\big(\dNorm{1}{y}-R\big)\leq \C(0)+2L\big(\dNorm{1}{y'}-R\big)=\widetilde{\C}_R(y').$$
Once again $\smash{\widetilde{\C}_R(y)\leq \widetilde{\C}_R(y')}$. Finally, if $\smash{y\in \CC\setminus B_{2R}}$, then $\smash{2R\leq \dNorm{1}{y}\leq \dNorm{1}{y'}}$ so $\smash{y'\in \CC\setminus B_{2R}}$ and clearly $\smash{\widetilde{\C}_R(y)\leq \widetilde{\C}_R(y')}$. This establishes the non-decreasingness of the regularized non-linearity $\smash{\widetilde{\C}_R}$. We now show that this non-linearity is uniformly Lipschitz continuous. The reverse triangle inequality implies that the map $\smash{y\mapsto \C(0)+2L\big(\dNorm{1}{y}-R\big)}$ is Lipschitz continuous with Lipschitz constant at most $2L$. Recall that the maximum of two Lipschitz continuous maps with Lipschitz constants at most $L_1$ and $L_2$, respectively, is Lipschitz continuous with Lipschitz constant at most $\max(L_1,L_2)$. This means that $\smash{\widetilde{\C}_R}$ is Lipschitz continuous with Lipschitz constant at most $2L$ when it is restricted to $\smash{\CC\cap B_{2R}}$ or $\smash{\CC\setminus B_{2R}}$. For $\smash{y,y'\in \CC}$ with $\smash{y\in B_{2R}}$ and $\smash{y'\in \CC\setminus B_{2R}}$, we distinguish two cases. On the one hand, if $\smash{\widetilde{\C}_R(y)=\C(0)+2L\big(\dNorm{1}{y}-R\big)}$, the reverse triangle inequality shows that
$$\big\lvert \widetilde{\C}_R(y)-\widetilde{\C}_R(y')\big\rvert\leq 2L\big\lvert \dNorm{1}{y}-\dNorm{1}{y'}\big\rvert\leq 2L\dNorm{1}{y-y'}.$$
On the other hand, if $\smash{\widetilde{\C}_R(y)=\C(y)}$, then the reverse triangle inequality reveals that
\begin{align*}
\widetilde{\C}_R(y)-\widetilde{\C}_R(y')&\leq \C(0)+L\dNorm{1}{y}-\C(0)-2L\big(\dNorm{1}{y'}-R\big)\leq L\dNorm{1}{y-y'}+L\big(2R-\dNorm{1}{y'}\big)\\
&\leq L\dNorm{1}{y-y'}
\end{align*}
while the lower bound $\widetilde{\C}_R(y)=\C(y)\geq \C(0)+2L\big(\dNorm{1}{y}-R\big)$ yields
$$\widetilde{\C}_R(y')-\widetilde{\C}_R(y)=2L\big(\dNorm{1}{y'}-\dNorm{1}{y}\big)\leq 2L\dNorm{1}{y-y'}.$$
This shows that $\widetilde{\C}_R$ is a non-decreasing function which agrees with $\C$ on $\CC\cap B_R$ and satisfies the Lipschitz continuity property 
\begin{equation}\label{eqn: SBME regularization is Lipschitz on cone}
\big\lvert \widetilde{\C}_R(y)-\widetilde{\C}_R(y')\big\rvert\leq \frac{8R}{m}\dNorm{1}{y-y'}
\end{equation}
for all $y,y'\in \CC$.\\
\noindent \step{2: extending to $\R^d$.}\\
To extend the regularization of the non-linearity \eqref{eqn: SBME non-linearity on cone} to $\smash{\R^d}$, define the function $\H_R:\R^d\to \R$ by
\begin{equation}\label{eqn: SBME extended non-linearity}
\H_R(y)=\inf\Big\{\widetilde{\C}_R(w)\mid w\in \CC\text{ with } w\geq y\Big\}.
\end{equation}
Let $\iota=(1,\ldots,1)\in \R^d$ and observe that the vector $v=\frac{G\iota}{m}$ belongs to $\CC$ and satisfies the bounds
\begin{equation}\label{eqn: SBME extented non-linearity vector v}
\frac{1}{d}\leq v_k\leq \frac{M}{dm}
\end{equation}
for $1\leq k\leq d$. In particular, the infimum in \eqref{eqn: SBME extended non-linearity} is never taken over the empty set. Moreover, the non-decreasingness of $\smash{\widetilde{\C}_R}$ and the fact that this function agrees with $\C$ on $\CC\cap B_R$ imply that $\smash{\H_R}$ also agrees with $\C$ on $\CC\cap B_R$. To see that $\smash{\H_R}$ is non-decreasing, fix $\smash{y,y'\in \R^d}$ with $y\geq y'$, and let $w\in \CC$ be such that $w\geq y'$. Since $w\geq y$, the definition of $\smash{\H_R}$ gives $\smash{\H_R(y)\leq \widetilde{\C}_R(w)}$, and taking the infimum over all such $w$ shows that $\smash{\H_R(y)\leq \H_R(y')}$. To establish the Lipschitz continuity of $\smash{\H_R}$, fix $\smash{y,y'\in \R^d}$ and let $z=\dNorm{1}{y-y'}v\in \CC$. Recalling \eqref{eqn: SBME extented non-linearity vector v} reveals that for any $1\leq k\leq d$,
$$y_k-y_k'\leq \norm{y-y'}_\infty=\frac{1}{d}\dNorm{1}{y-y'}\leq v_k\dNorm{1}{y-y'}=z_k.$$
This means that $z\geq y-y'$. In particular, if $w\in \CC$ is such that $w\geq y'$, then $w+z\in \CC$ with $w+z\geq y$. It follows by \eqref{eqn: SBME regularization is Lipschitz on cone}, \eqref{eqn: SBME extended non-linearity} and \eqref{eqn: SBME extented non-linearity vector v} that
$$\H_R(y)-\widetilde{\C}_R(w)\leq \widetilde{\C}_R(w+z)-\widetilde{\C}_R(w)\leq \frac{8R}{m}\dNorm{1}{z}\leq \frac{8RM}{m^2}\dNorm{1}{y-y'}.$$
Taking the infimum over all such $w$ and reversing the roles of $y$ and $y'$ completes the proof.
\end{proof}

With this extended non-linearity at hand, we can now establish the well-posedness of the Hamilton-Jacobi equation \eqref{eqn: SBME HJ eqn on Rpp}. The appropriate notion of solution for \eqref{eqn: SBME HJ eqn on Rpp} will be that of a viscosity solution. With \Cref{SBME app HJ eqn on half-space} in mind, let us fix a non-linearity $\smash{\H:\R^d\to \R}$ and a domain $\smash{\D\subset \Rp^d}$, and define the notion of a viscosity solution for the more general Hamilton-Jacobi equation
\begin{equation}\label{eqn: SBME general HJ eqn}
\partial_t f(t,x)=\H\big(\nabla f(t,x)\big) \quad \text{on}\quad \Rpp\times \D.\\
\end{equation}
A brief review of the definition and elementary properties of semi-continuous functions is provided in \Cref{SBME app background}.
\begin{definition}
An upper semi-continuous function $\smash{u:[0,\infty)\times  \D\to \R}$ is said to be a viscosity subsolution to \eqref{eqn: SBME general HJ eqn} if, given any $\smash{\phi\in C^\infty\big((0,\infty)\times  \D\big)}$ with the property that $u-\phi$ has a local maximum at $\smash{(t^*,x^*)\in (0,\infty)\times \D}$,
\begin{equation}\label{eqn: SBME subsolution conditions}
\big(\partial_t \phi-\H(\nabla \phi)\big)(t^*,x^*)\leq 0.
\end{equation}
\end{definition}

\begin{definition}
A lower semi-continuous function $\smash{v:[0,\infty)\times \D\to \R}$ is said to be a viscosity supersolution to \eqref{eqn: SBME general HJ eqn} if, given any $\smash{\phi\in C^\infty\big((0,\infty)\times  \D\big)}$ with the property that $v-\phi$ has a local minimum at $\smash{(t^*,x^*)\in (0,\infty)\times \D}$,
\begin{equation}\label{eqn: SBME supersolution conditions}
\big(\partial_t \phi-\H(\nabla \phi)\big)(t^*,x^*)\geq 0.
\end{equation}
\end{definition}

\begin{definition}\label{SBME definition of viscosity solution}
A continuous function $\smash{f\in C\big([0,\infty)\times \D\big)}$ is said to be a viscosity solution to \eqref{eqn: SBME general HJ eqn} if it is both a viscosity subsolution and a viscosity supersolution to \eqref{eqn: SBME general HJ eqn}.
\end{definition}

The existence and uniqueness results for Hamilton-Jacobi equations on positive half-spaces developed in \Cref{SBME app HJ eqn on half-space} now give the well-posedness of the Hamilton-Jacobi equation \eqref{eqn: SBME HJ eqn on Rpp}. It will be convenient to remember that any closed and convex set $\smash{\K\subset \R^d}$ may be represented as the intersection of the closed and affine half-spaces which contain it,
\begin{equation}\label{eqn: SBME expanded cone as hyperplanes}
\K=\big\{x\in \R^d\mid x\cdot v\geq c \text{ for all } (v,c)\in \A\big\},
\end{equation}
where
\begin{equation}\label{eqn: SBME expanded cone as hyperplanes set A}
\A=\big\{(v,c)\in \R^{d+1}\mid x\cdot v\geq c \text{ for all } x\in \K \text{ and } \norm{v}=1\big\}
\end{equation}
for any norm $\norm{\cdot}$. A proof of this classical result may be found in Corollary 4.2.4 of \cite{Lemarechal}.

\begin{proposition}\label{SBME WP of HJ eqn on Rpp}
For every $R>0$, the Hamilton-Jacobi equation \eqref{eqn: SBME HJ eqn on Rpp} admits a unique viscosity solution $f_R\in \mathfrak{L}_{\mathrm{unif}}$ subject to the initial condition $\psi$. Moreover, $f_R$ has its gradient in the set $\smash{\K_a'}$ and satisfies the Lipschitz bound 
\begin{equation}\label{eqn: SBME WP of HJ eqn on cone}
\sup_{t>0}\Norm{f_R(t,\cdot)}_{\mathrm{Lip},1}=\Norm{\psi}_{\mathrm{Lip},1}.
\end{equation}
\end{proposition}

\begin{proof}
To alleviate notation, we fix $R>0$ and omit all dependencies on $R>0$. We invoke \Cref{SBME comparison principle on Rpp} to find a viscosity solution $f\in \mathfrak{L}_{\mathrm{unif}}$ to the Hamilton-Jacobi equation \eqref{eqn: SBME HJ eqn on Rpp} subject to the initial condition $\psi$ with
\begin{equation}\label{eqn: SBME HJ eqn on Rpp in Lunif}
\sup_{t>0}\Norm{f(t,\cdot)}_{\mathrm{Lip},1}=\Norm{\psi}_{\mathrm{Lip},1}.
\end{equation}
We will now show that $f$ has its gradient in the closed convex set $\smash{\K_a'}$. Denote by $\A$ the set \eqref{eqn: SBME expanded cone as hyperplanes set A} associated with $\smash{\K_a'}$ and the Euclidean norm $\norm{\cdot}_2$. For each $(v,c)\in \A$ introduce the closed convex cone
$$\HH_v=\{x\in \R^d\mid x\cdot v\geq 0\big\}$$
as well as the function $g_{v,c}:[0,\infty)\times \Rp^d\to \R$ defined by $g_{v,c}(t,x)=f(t,x)-cx\cdot v$. It is readily verified that $g_{v,c}$ satisfies the Hamilton-Jacobi equation
$$\partial_tg (t,x)=\widetilde{\H}_R\big(\nabla g(t,x)\big) \quad \text{ on } \Rpp\times \Rp^d$$
subject to the initial condition $\smash{g_{v,c}(0,x)=\psi(x)-cx\cdot v}$ for the non-linearity $\smash{\widetilde{\H}_R(y)=\H_R(y+cv\iota)}$, where $\smash{\iota=(1)_{\D_K}\in \R^{\D_K}}$. Moreover, this initial condition is $\smash{\HH_v^*}$-non-decreasing (the definition of being $\smash{\HH_v^*}$-non-decreasing was introduced at the beginning of this section). Indeed, we have $\smash{\HH_v^*=\R v}$ by the biduality result in \Cref{SBME closed convex cone bidual}. Moreover, for any $\smash{x,x'\in \Rp^d}$ with $x'-x=tv$ for some $t\in \R$, the fact that $(x'-x)\cdot z\geq tc$ for all $\smash{z\in \K_a'}$ and the characterization of $\psi$ having its gradient in the set $\smash{\K_a'}$ given in \Cref{SBME gradient in convex set} imply that
$$g_{v,c}(x')-g_{v,c}(x)=\psi(x')-\psi(x)-c(x'-x)\cdot v\geq tc-tcv\cdot v=0.$$
It follows by \Cref{SBME monotone initial conition gives monotone solution} that $g_{v,c}$ is $\HH_v^*$-non-decreasing. We now fix $x,x'\in \Rp^d$ with the property that for all $z\in \K_a'$, we have $(x'-x)\cdot z\geq c$. If $\smash{v=\frac{x'-x}{\norm{x'-x}_2}}$ and $\smash{c'=\frac{c}{\norm{x'-x}_2}}$, then $(v,c')\in \A$ so the function $\smash{g_{v,c'}}$ is $\smash{\HH_v^*}$-non-decreasing. This implies that
$$f(t,x')-f(t,x)=g_{v,c'}\big(t,x+\norm{x'-x}_2v\big)-g_{v,c'}(t,x)+c'\norm{x'-x}_2v\cdot v\geq c$$
which means that $f$ has its gradient in $\smash{\K_a'}$ by \Cref{SBME gradient in convex set}. This completes the proof.
\end{proof}

\begin{proof}[Proof of Theorem~\ref{SBME WP of projecetd HJ eqn}]
Writing $\H_{K,R}$ for the extension of the non-linearity $\C_K$ constructed in \Cref{SBME extending the non-linearity}, the desired result is now an immediate consequence of \Cref{SBME WP of HJ eqn on Rpp} and the Lipschitz bound~\eqref{eqn: SBME projected non-linearity Lipschitz}.
\end{proof}

\section{Well-posedness of the infinite-dimensional equation}\label{SBME section convergence}

In this section, we establish the convergence of the solutions to the projected Hamilton-Jacobi equations \eqref{eqn: SBME projected HJ eqn} as stated in \Cref{SBME infinite HJ eqn WP}. In the notation of \Cref{SBME WP of projecetd HJ eqn}, given an integer $K\geq 1$ and some $R>0$, write $\smash{f_R^{(K)}\in \mathfrak{L}_{\mathrm{unif}}}$ for the unique solution to the Hamilton-Jacobi equation \eqref{eqn: SBME HJ eqn on Rpp} subject to the initial condition $\smash{\psi^{(K)}}$. Recall that $\smash{f_R^{(K)}}$ has its gradient in the set $\smash{\K_{a,K}'}$ and satisfies the Lipschitz bound
\begin{equation}\label{eqn: SBME projected solution Lipschitz}
\sup_{t\geq 0}\Norm{f^{(K)}_R(t,\cdot)}_{\mathrm{Lip},1}=\Norm{\psi^{(K)}}_{\mathrm{Lip},1}\leq \norm{\psi}_{\mathrm{Lip},\TV}.
\end{equation}
To prove the existence of the limit
\begin{equation}\label{eqn: SBME limit of projections}
f_R(t,\mu)=\lim_{K\to \infty} f^{(K)}_R\big(t,x^{(K)}(\mu)\big)
\end{equation}
we will appropriately adapt the arguments in Section 3.2 of \cite{JC_NC} and Section 3.3 of \cite{HB_HJ}. Given two integers $K'> K$, it will be convenient to introduce the projection map $\smash{P^{(K,K')}: \Rp^{\D_{K'}}\to \Rp^{\D_K}}$ defined by
\begin{equation}\label{eqn: SBME projection operator}
P^{(K,K')}x=x^{(K)}\big(\mu^{(K')}_x\big)
\end{equation}
as well as the lifting map $L^{(K,K')}: \R^{\D_K}\to \R^{\D_{K'}}$ given by
\begin{equation}
L^{(K,K')}x=\big(\widetilde{x}_{k}\big)_{k\in \D_K},
\end{equation}
where $\smash{\widetilde{x}_{k}=(x_{k},\ldots,x_{k})\in \R^{2^{K'-K}}}$. A key observation in proving the existence of the limit \eqref{eqn: SBME limit of projections} is that
\begin{equation}\label{eqn: SBME projection is surjective}
P^{(K,K')}L^{(K,K')}x=x.
\end{equation}
The following technical lemmas will also play their part. The first two translate non-differential properties of a non-differentiable function into differential properties of a smooth function at any point where the difference of these functions is locally maximal. The third analyzes the transformation of the pairs $(v,c)$ in the representation \eqref{eqn: SBME expanded cone as hyperplanes} of a closed convex set by the projection map \eqref{eqn: SBME projection operator}, and the fourth shows that these pairs can be used to quantify the distance from a point to the closed convex set they define.

\begin{lemma}\label{SBME gradient bound on test function}
Fix a Lipschitz function $\smash{u\in C\big((0,\infty)\times \Rp^d\big)}$ with $\smash{L=\sup_{t>0}\Norm{u(t,\cdot)}_{\mathrm{Lip},1}<\infty}$.
If $\smash{\phi \in C^\infty\big((0,\infty)\times \Rpp^d\big)}$ is a smooth function with the property that $u-\phi$ has a local maximum at the point $\smash{(t^*,x^*)\in (0,\infty)\times \Rpp^d}$, then $\smash{\dNorm{1}{\nabla \phi(t^*,x^*)}\leq L}$. An identical statement holds at a local minimum.
\end{lemma}

\begin{proof}
Since $u-\phi$ has a local maximum at $(t^*,x^*)\in (0,\infty)\times \Rpp^d$, for every $\epsilon>0$ small enough and $x\in \Rp^d$,
$$\phi(t^*,x^*+\epsilon x)-\phi(t^*,x^*)\geq u(t^*,x^*+\epsilon x)-u(t^*,x^*)\geq -\epsilon L\Norm{x}_1.$$
Dividing by $\epsilon$ and letting $\epsilon$ tend to zero reveals that
$$\nabla \phi(t^*,x^*)\cdot x\geq -L\Norm{x}_1.$$
Choosing $\smash{x_k=-d\sgn\big(\partial_{x_k}\phi(t^*,x^*)\big)e_k}$ for each $1\leq k\leq d$ completes the proof.
\end{proof}

\begin{lemma}\label{SBME gradient of test function in cone}
Fix a closed convex set $\K'\subset \R^d$ and a Lipschitz function $\smash{u\in C\big((0,\infty)\times \Rp^d\big)}$ with $\smash{\nabla u\in L^\infty((0,\infty)\times}\Rp^d;\K')$. Any smooth function $\smash{\phi \in C^\infty\big((0,\infty)\times \Rpp^d\big)}$ with the property that $u-\phi$ has a local maximum at $\smash{(t^*,x^*)\in (0,\infty)\times \Rpp^d}$ is such that $\nabla \phi(t^*,x^*) \in \K'$. An identical statement holds at a local minimum.
\end{lemma}

\begin{proof}
Recall the representation \eqref{eqn: SBME expanded cone as hyperplanes} of the closed convex set $\K'$ as the intersection of the closed and affine half-spaces which contain it, and fix $(v,c)\in \A$. Since $u-\phi$ has a local maximum at $(t^*,x^*)\in (0,\infty)\times \Rpp^d$, for every $\epsilon>0$ small enough,
$$\phi(t^*,x^*+\epsilon v)-\phi(t^*,x^*)\geq u(t^*,x^*+\epsilon v)-u(t^*,x^*)\geq \epsilon c.$$
The second inequality uses the characterization of $\nabla u\in \K'$ given in \Cref{SBME gradient in convex set} and the fact that $x\cdot \epsilon v\geq \epsilon c$ for all $\smash{x\in \K'}$. Dividing by $\epsilon$ and letting $\epsilon$ tend to zero reveals that
$$\nabla \phi(t^*,x^*)\cdot v\geq c$$
for all $(v,c)\in \A$. It follows that $\nabla \phi(t^*,x^*)\in \K'$. This completes the proof.
\end{proof}

\begin{lemma}\label{SBME projecting cones}
Fix two integers $K'>K$ large enough and a pair of vectors $\smash{(v,c)\in \R^{\D_{K'}}\times \R}$ with $\Norm{v}_1=1$ such that, for every $\smash{x\in \K_{a,K'}'}$, we have $v\cdot x\geq c$. Then, for every $\smash{y\in \K_{a,K}'}$, 
\begin{equation}
P^{(K,K')}v\cdot y\geq c-\frac{2}{2^{K/2}}.
\end{equation}
\end{lemma}

\begin{proof}
Fix $\smash{y\in \K_{a,K'}}$, and find vectors $\smash{u^{(K)}\in \Rp^{\D_K}}$ and $\smash{w^{(K)}\in \R^{\D_K}}$ with
$$y=G^{(K)}u^{(K)}+w^{(K)}, \quad \Norm{u^{(K)}}_1\leq a \quad \text{and} \quad \dNorm{1}{w^{(K)}}\leq \frac{1}{2^{K/2}}.$$
Consider the vector 
$$u_{k'}^{(K')}=\frac{\abs{\D_{K'}}}{\abs{\D_K}}u^{(K)}_{k'}\1\{k'\in \D_K\}$$
in $\smash{\R^{\D_{K'}}}$, and for each $k'\in \D_{K'}$, write $\underline{k}'$ for the unique dyadic $\underline{k}'\in \D_K$ with $k'\in [\underline{k}',\underline{k}'+2^{-K})$. Observe that
$$P^{(K,K')}v\cdot y=P^{(K,K')}v\cdot G^{(K)}u^{(K)}+P^{(K,K')}v\cdot w^{(K)}=v\cdot \big(G^{(K')}u^{(K')}+\alpha^{(K')}\big)+P^{(K,K')}v\cdot w^{(K)}$$
for the vector $\alpha^{(K')}\in \R^{\D_{K'}}$ defined by
$$\alpha^{(K')}_{k'}=\frac{1}{\abs{\D_{K'}}^2}\sum_{k'\in \D_{K'}}\big(g(\underline{k}k')-g(kk')\big)u^{(K')}_{k'}.$$
Since $G^{(K')}u^{(K')}\in \K_{a,K'}'$, the defining property of $v$, the fact that $\Norm{v}_1=1$ and Hölder's inequality give the lower bound
\begin{equation}\label{eqn: SBME projecting cones key}
P^{(K,K')}v\cdot y\geq c-\dNorm{1}{\alpha^{(K')}}-\dNorm{1}{w^{(K)}}\geq c-\dNorm{1}{\alpha^{(K')}}-\frac{1}{2^{K/2}},
\end{equation}
where we have used that $\Norm{P^{(K,K')}v}_1\leq \Norm{v}_1$ and $\dNorm{1}{w^{(K)}}\leq 2^{-K/2}$. The mean value theorem reveals that
$$\dNorm{1}{\alpha^{(K')}}\leq \frac{\norm{g'}_\infty}{2^K}\Norm{u^{(K)}}_1\leq \frac{a\norm{g'}_\infty}{2^K}.$$
Substituting this into \eqref{eqn: SBME projecting cones key} and taking $K$ large enough completes the proof.
\end{proof}

\begin{lemma}\label{SBME distance to K' through hyperplanes}
Fix a closed convex set $\smash{\K'\subset \R^d}$, and recall its representation \eqref{eqn: SBME expanded cone as hyperplanes} with $\smash{\norm{\cdot}=\Norm{\cdot}_1}$ as the intersection of the closed and affine half-spaces which contain it. If $\smash{x\in \R^d}$ and $\epsilon>0$ are such that $x\cdot v\geq c-\epsilon$ for all $(v,c)\in \A$, then there exist $\smash{y\in \K'}$ and $\smash{z\in \R^d}$ with $x=y+z$ and $\smash{\dNorm{1}{z}\leq \epsilon}$.
\end{lemma}

\begin{proof}
Let $\smash{y\in \K'}$ denote a projection of $\smash{x\in \R^d}$ onto the set $\smash{\K'}$ with respect to the normalized-$\smash{\ell^{1,*}}$ norm. More precisely, let $\smash{y\in \K'}$ be any minimizer of the map $\smash{y'\mapsto \dNorm{1}{y'-x}}$ over points $\smash{y'\in \K'}$. The existence of such a projection is guaranteed by the fact that $\smash{\K'}$ is closed. If $y = x$, then the desired conclusion is immediate, so from now on we assume that $y \neq x$. Introduce the set
$$\I=\big\{k\leq d\mid d\abs{x_k-y_k}=\dNorm{1}{x-y}\big\}$$
of indices at which $\smash{\dNorm{1}{x-y}}$ is achieved, and define the vector $\smash{v\in \R^d}$ by
$$v_k=\frac{d}{\abs{\I}}\sgn(y_k-x_k)\1\{k\in \I\}.$$
We now show that $(v,c)\in \A$ for $c=v\cdot y$. By construction $\Norm{v}_1=1$, so suppose for the sake of contradiction that there exists $\smash{y'\in \K'}$ with $\smash{(y'-y)\cdot v=y'\cdot v-c<0}$. This means that
$$(y-y')\cdot v=\frac{d}{\abs{\I}}\sum_{k\in \I}(y_k-y'_k)\sgn(y_k-x_k)>0.$$
In particular, a coordinate $k^*\in \I$ at which the quantity $\smash{(y_k-y'_k)\sgn(y_k-x_k)}$ is maximized over $k\in \I$ must satisfy
$$(y_{k^*}-y'_{k^*})\sgn(y_{k^*}-x_{k^*})>0.$$
At this point, fix $t\in (0,1)$ small enough so that $\sgn(y_k-x_k+t(y_k'-y_k))=\sgn(y_k-x_k)$ for every $k\in \I$. For such a value of $t>0$,
$$\dNorm{1}{y-x+t(y'-y)}=d(y_{k^*}-x_{k^*}+t(y'_{k^*}-y_{k^*}))\sgn(y_{k^*}-x_{k^*})<\dNorm{1}{x-y}.$$
Since the point $y''=y+t(y'-y)$ is a convex combination of $y,y'\in \K'$, it must lie in the convex set $\K'$. This contradicts the fact that $y$ minimizes the map $\smash{y''\mapsto \dNorm{1}{y''-x}}$ over points $\smash{y''\in \K'}$, and shows that $(v,c)\in \A$. It follows that
$$\epsilon \geq c-x\cdot v=v\cdot (y-x)=\dNorm{1}{x-y}.$$
Setting $z=x-y$ completes the proof.
\end{proof}

We are now in a position to prove \Cref{SBME infinite HJ eqn WP}.

\begin{proof}[Proof of \Cref{SBME infinite HJ eqn WP}.]
To alleviate notation, until otherwise stated, we fix $\smash{R>\norm{\psi}_{\mathrm{Lip},\TV}}$ and keep all dependencies on $R$ implicit. The existence of the limit \eqref{eqn: SBME limit of projections} will be established by showing that the sequence $\smash{(f^{(K)}(t,x^{(K)}(\mu)))_K}$
is Cauchy. With this in mind, fix $K'>K$ and introduce the function
\begin{equation}
f^{(K,K')}(t,x)=f^{(K)}\big(t,P^{(K,K')}x\big)
\end{equation}
defined on 
$\smash{\Rp\times \Rp^{\D_{K'}}}$. Since $x^{(K)}(\mu)=P^{(K,K')}x^{(K')}(\mu)$, the Cauchy condition may be expressed in terms of this function as
\begin{equation}\label{eqn: SBME projected solution Cauchy}
\big\lvert f^{(K')}\big(t,x^{(K')}(\mu)\big)-f^{(K)}\big(t,x^{(K)}(\mu)\big)\big\rvert=\big\lvert f^{(K')}\big(t,x^{(K')}(\mu)\big)-f^{(K,K')}\big(t,x^{(K')}(\mu)\big)\big\rvert.
\end{equation}
To control the right-hand side of this expression, we will first show that $f^{(K,K')}$ is an approximate viscosity solution to the Hamilton-Jacobi equation \eqref{eqn: SBME HJ eqn on Rpp} satisfied by $f^{(K')}$, and then we will leverage the comparison principle in \Cref{SBME comparison principle on Rpp}.\\
\noindent \step{1: $f^{(K,K')}$ is an approximate viscosity solution.}\\
Consider a function $\phi_{K'}\in C^\infty\big((0,\infty)\times \Rpp^{\D_{K'}}\big)$ with the property that $\smash{f^{(K,K')}-\phi_{K'}}$ achieves a local maximum at $\smash{(t^*,x^*)\in (0,\infty)\times \Rpp^{\D_{K'}}}$. To be more precise, suppose that 
$$\sup_{B_{K'}(r)}\big(f^{(K,K')}-\phi_{K'}\big)=\big(f^{(K,K')}-\phi_{K'}\big)(t^*,x^*),$$
where 
$$B_{K'}(r)=\Big\{(t,x)\in (0,\infty)\times \Rp^{\D_{K'}}\mid \abs{t-t^*}+\Norm{x-x^*}_1\leq r\Big\}$$
is the ball of radius $r>0$ centered at $(t^*,x^*)$. Decreasing $r>0$ if necessary, assume without loss of generality that
$$B_{K'}(r)\subset (0,\infty)\times \Rpp^{\D_{K'}}.$$
Assume also that $\smash{\phi_{K'}\in C^\infty\big((0,\infty)\times \R^{\D_{K'}}\big)}$; this can be ensured by replacing $\smash{\phi_{K'}}$ with $\smash{\eta\phi_{K'}}$ for some $\smash{\eta\in C^\infty\big(\R^{\D_{K'}}\big)}$ which is identically one on $\smash{B_{K'}(r)}$ and vanishes outside $\smash{\Rp^{\D_{K'}}}$. With these simplifications at hand, introduce the smooth function
$$\phi_K(t,y)=\phi_{K'}\big(t,x^*+L^{(K,K')}y-L^{(K,K')}P^{(K,K')}x^*\big)$$
defined on $\smash{(0,\infty)\times \Rpp^{\D_K}}$.
We will now show that the function $\phi_K$ admits a local maximum at $\smash{(t^*,P^{(K,K')}x^*)}$. It will be convenient to notice that for any $\smash{y\in \Rp^{\D_{K}}}$
\begin{equation}\label{eqn: SBME projection of lifted point}
P^{(K,K')}\big(x^*+L^{(K,K')}y-L^{(K,K')}P^{(K,K')}x^*\big)=P^{(K,K')}x^*+y-P^{(K,K')}x^*=y\in \Rpp^{\D_{K}}
\end{equation}
by \eqref{eqn: SBME projection is surjective}.
To simplify notation, let $\smash{y^*=P^{(K,K')}x^*\in \Rp^{\D_K}}$ and introduce the ball
$$B_K(r)=\big\{(t,y)\in (0,\infty)\times \Rpp^{\D_K}\mid \abs{t-t^*}+\Norm{y-y^*}_1\leq r\big\}\subset (0,\infty)\times \Rpp^{\D_{K}}$$
of radius $r>0$ centered at $(t^*,y^*)$. Given $(t,y)\in B_K(r)$, let $z_y=x^*+L^{(K,K')}y-L^{(K,K')}P^{(K,K')}x^*$ in such a way that by \eqref{eqn: SBME projection of lifted point},
$$f^{(K)}(t,y)-\phi_K(s,y)=f^{(K,K')}(t,z_y)-\phi_{K'}(t,z_y).$$
Observe that
$$\abs{t-t^*}+\Norm{z_y-x^*}_1=\abs{t-t^*}+\Norm{L^{(K,K')}y-L^{(K,K')}y^*}_1=\abs{t-t^*}+\Norm{y-y^*}_1\leq r$$
so $(t,z_y)\in B_{K'}(r)$. It follows that
$$\sup_{B_K(r)}\big(f^{(K)}-\phi_K\big)\leq \sup_{B_{K'}(r)}\big(f^{(K,K')}-\phi_{K'}\big)=\big(f^{(K,K')}-\phi_{K'}\big)(t^*,x^*)=\big(f^{(K)}-\phi_K\big)(t^*,y^*)$$
which means that $\phi_K$ admits a local maximum at $\smash{(t^*,y^*)\in (0,\infty)\times \Rpp^{\D_K}}$. Since $\smash{f^{(K)}}$ is a viscosity subsolution to the Hamilton-Jacobi equation \eqref{eqn: SBME HJ eqn on Rpp} with $\nabla f^{(K)}\in \K_{a,K}'$, \Cref{SBME gradient of test function in cone} implies that
$$\nabla \phi_K(t^*,y^*)\in \K_{a,K}' \quad \text{and} \quad \big(\partial_t \phi_K-\H_K\big(\nabla \phi_K\big)\big)(t^*,y^*)\leq 0.$$
To write this expression in terms of the original test function $\phi_{K'}$, notice that
$$\partial_t \phi_K(t^*,y^*)=\partial_t\phi_{K'}(t^*,x^*) \quad \text{ and } \quad \nabla \phi_K(t^*,y^*)=\frac{\abs{\D_{K'}}}{\abs{\D_K}}P^{(K,K')}\nabla \phi_{K'}(t^*,x^*).$$
This means that 
$$\frac{\abs{\D_{K'}}}{\abs{\D_K}}P^{(K,K')}\nabla \phi_{K'}(t^*,x^*)\in \K_{a,K}'\quad \text{and} \quad \partial_t\phi_{K'}(t^*,x^*)-\H_K\bigg(\frac{\abs{\D_{K'}}}{\abs{\D_K}}P^{(K,K')}\nabla \phi_{K'}(t^*,x^*)\bigg)\leq 0.$$
The first of these conditions gives vectors $\smash{u^{(K)}\in \Rp^{\D_K}}$ and $\smash{w^{(K)}\in \R^{\D_K}}$ with 
\begin{equation}\label{eqn: projected solution subsolution derivative}
\frac{\abs{\D_{K'}}}{\abs{\D_K}}P^{(K,K')}\nabla \phi_{K'}(t^*,x^*)=G^{(K)}u^{(K)}+w^{(K)}, \quad \Norm{u^{(K)}}_1\leq a \quad \text{and} \quad \dNorm{1}{w^{(K)}}\leq \frac{1}{2^{K/2}}.
\end{equation}
Observe that
$$\dNorm{1}{G^{(K)}u^{(K)}}\leq \frac{\abs{\D_{K'}}}{\abs{\D_K}}\dNorm{1}{P^{(K,K')}\nabla \phi_{K'}(t^*,x^*)}+\frac{1}{2^{K/2}}\leq \dNorm{1}{\nabla \phi_{K'}(t^*,x^*)}+\frac{1}{2^{K/2}}.$$
Since $f^{(K,K')}-\phi_{K'}$ achieves a local maximum at $(t^*,x^*)$, \Cref{SBME gradient bound on test function} and \eqref{eqn: SBME projected solution Lipschitz} imply that 
$$\dNorm{1}{\nabla \phi_{K'}(t^*,x^*)}\leq \norm{\psi}_{\mathrm{Lip},\TV}.$$
Recalling that $\smash{R>\norm{\psi}_{\mathrm{Lip},\TV}}$ and taking $K$ large enough ensures that $G^{(K)}u^{(K)}\in \CC_K\cap B_R$. It follows by the Lipschitz continuity of $\H_K$ established in \Cref{SBME extending the non-linearity} that
\begin{align*}
\partial_t\phi_{K'}(t^*,x^*)-\C_K\big(G^{(K)}u^{(K)}\big)&\leq \partial_t\phi_{K'}(t^*,x^*)-\H_K\bigg(\frac{\abs{\D_{K'}}}{\abs{\D_K}}P^{(K,K')}\nabla \phi_{K'}(t^*,x^*)\bigg)+\frac{8RM}{2^{K/2}m^2}\\
&\leq \frac{8RM}{2^{K/2}m^2}.
\end{align*}
At this point, introduce the vector $u^{(K')}\in \Rp^{\D_{K'}}$ defined by
$$u^{(K')}_{k'}=\frac{\abs{\D_{K'}}}{\abs{\D_K}}u_{k'}^{(K)}\1\{k'\in \D_K\}$$
in such a way that
$$\C_{K'}\big(G^{(K')}u^{(K')}\big)=\frac{1}{\abs{\D_K}^2}\sum_{k,k'\in \D_K}g(kk')u^{(K)}_{k}u^{(K)}_{k'}=\C_K(G^{(K)}u^{(K)}\big),$$
and therefore,
\begin{equation}\label{eqn: SBME convergence subsolution for CK'}
\partial_t\phi_{K'}(t^*,x^*)-\C_{K'}\big(G^{(K')}u^{(K')}\big)\leq \frac{8RM}{2^{K/2}m^2}.
\end{equation}
We now show that, up to an error vanishing with $K$, the term $\smash{G^{(K')}u^{(K')}}$ in this expression may be replaced by $\smash{\nabla \phi_{K'}(t^*,x^*)}$. This is where \Cref{SBME projecting cones} will play its part. Recall the representation \eqref{eqn: SBME expanded cone as hyperplanes} with $\smash{\norm{\cdot}=\Norm{\cdot}_1}$ of $\smash{\K_{a,K'}'}$ as the intersection of the closed and affine half-spaces which contain it, and fix $\smash{(v,c)\in \A}$ with $\Norm{v}_1=1$. The characterization of $\smash{\nabla f^{(K)}\in \K_{a,K}'}$ given in \Cref{SBME gradient in convex set} and \Cref{SBME projecting cones} imply that for every $\epsilon>0$ small enough,
\begin{align*}
\phi_{K'}(t^*,x^*+\epsilon v)-\phi_{K'}(t^*,x^*)
&\geq f^{(K)}\big(t^*,P^{(K,K')}x^*+\epsilon P^{(K,K')}v\big)-f^{(K)}\big(t^*,P^{(K,K')}x^*\big)\\
&\geq \epsilon\bigg(c-\frac{2}{2^{K/2}}\bigg)
\end{align*}
Dividing by $\epsilon$ and letting $\epsilon$ tend to zero reveals that $\smash{\nabla \phi_{K'}(t^*,x^*)\cdot v\geq c-\frac{2}{2^{K/2}}}$. Invoking \Cref{SBME distance to K' through hyperplanes} gives $\smash{\alpha^{(K')}\in \Rp^{\D_{K'}}}$ and $\smash{\beta^{(K')}\in \R^{\D_{K'}}}$ with
\begin{equation}\label{eqn: lifted solution subsolution derivative}
\nabla\phi_{K'}(t^*,x^*)=G^{(K')}\alpha^{(K')}+\beta^{(K')},\quad \Norm{\alpha^{(K')}}_1\leq a \quad \text{and} \quad \dNorm{1}{\beta^{(K')}}\leq \frac{2}{2^{K/2}}.
\end{equation}
At this point, fix $k\in \D_{K}$ and $k'\in [k,k+2^{-K})$. The mean value theorem implies that
\begin{align*}
\abs{\D_{K'}}\big\lvert \big(G^{(K')}u^{(K')}\big)_{k'}-\partial_{x_{k'}} \phi_{K'}(t^*,x^*)\big\rvert
&=\bigg\lvert \frac{1}{\abs{\D_K}}\sum_{k''\in \D_K}g(k'k'')u^{(K)}_{k''}-\abs{\D_{K'}}\partial_{x_{k'}} \phi_{K'}(t^*,x^*)\bigg\rvert\\
&\le\big\lvert \abs{\D_K}\big(G^{(K)}u^{(K)}\big)_{k}-\abs{\D_{K'}}\partial_{x_{k'}} \phi_{K'}(t^*,x^*)\big\rvert\\
&\qquad\qquad\qquad\qquad\qquad\qquad\quad+\frac{\norm{g'}_\infty\Norm{u^{(K)}}_1}{2^K}.
\end{align*}
Remembering \eqref{eqn: projected solution subsolution derivative} and \eqref{eqn: lifted solution subsolution derivative}, noticing that $\abs{\D_K}=2^{K+1}$ and using the mean value theorem once again shows that
\begin{align*}
\abs{\D_K}\big(G^{(K)}u^{(K)}\big)_{k}&=\abs{\D_K}\sum_{\ell=0}^{2^{K'-K}-1}\partial_{x_{k+\frac{\ell}{2^{K'}}}}\phi_{K'}(t^*,x^*)-\abs{\D_K}w^{(K)}_k\\
&=\frac{\abs{\D_K}}{\abs{\D_{K'}}^2}\sum_{\ell=0}^{2^{K'-K}-1}\sum_{k''\in \D_{K'}}g\Big(k+\frac{\ell}{2^{K'}}\cdot k''\Big)\alpha_{k''}^{(K')}+\abs{\D_K}\sum_{\ell=0}^{2^{K'-K}-1}\beta^{(K')}_{k+\frac{\ell}{2^{K'}}}-\abs{\D_K}w_k^{(K)}\\
&=\frac{\abs{\D_K}}{\abs{\D_{K'}}^2}\sum_{\ell=0}^{2^{K'-K}-1}\sum_{k''\in \D_{K'}}g(k'k'')\alpha_{k''}^{(K')}+\BigO_1\bigg(\frac{\norm{g'}_\infty\Norm{\alpha^{(K')}}_1}{2^K}+\frac{3}{2^{K/2}}\bigg)\\
&=\abs{\D_{K'}}\big(G^{(K')}\alpha^{(K')}\big)_{k'}+\BigO_1\bigg(\frac{4}{2^{K/2}}\bigg)=\abs{\D_{K'}}\partial_{x_{k'}}\phi_{K'}(t^*,x^*)+\BigO_1\bigg(\frac{5}{2^{K/2}}\bigg),
\end{align*}
where we have written $X=Y+\BigO_1(Z)$ to mean that $\abs{X-Y}\leq Z$. In the third equality we used that $\smash{\dNorm{1}{\beta^{(K')}}+\dNorm{1}{w^{(K)}}\leq 3\cdot 2^{-K/2}}$, and in the fourth equality we used that $\smash{\Norm{\alpha^{(K')}}_1\leq a}$ and increased $K$ if necessary. It follows that
$$\dNorm{1}{G^{(K')}u^{(K')}-\nabla \phi_{K'}(t^*,x^*)}\leq \frac{5}{2^{K/2}}+\frac{a\norm{g'}_\infty}{2^K}\leq \frac{6}{2^{K/2}},$$
where we have used that $\Norm{u^{(K)}}_1\leq a$ and increased $K$ if necessary. Combining this bound with the Lipschitz continuity of $\H_{K'}$ established in \Cref{SBME extending the non-linearity} and \eqref{eqn: SBME convergence subsolution for CK'} reveals that
\begin{equation}
\partial_t\phi_{K'}(t^*,x^*)-\H_{K'}\big(\nabla \phi_{K'}(t^*,x^*)\big)\leq \EE_K
\end{equation}
for the error term
$$\EE_K=\frac{56RM}{2^{K/2}m^2}.$$
In particular, the function $\smash{(t,x)\mapsto f^{(K,K')}(t,x)-\EE_K t}$ is a viscosity subsolution to the Hamilton-Jacobi equation \eqref{eqn: SBME HJ eqn on Rpp} satisfied by $\smash{f^{(K')}}$. An identical argument shows that $\smash{(t,x)\mapsto f^{(K,K')}(t,x)+\EE_K t}$ is a viscosity supersolution to the Hamilton-Jacobi equation \eqref{eqn: SBME HJ eqn on Rpp} satisfied by $\smash{f^{(K')}}$.\\
\step{2: comparison principle.}\\
Using \eqref{eqn: SBME projected solution Lipschitz} and \eqref{eqn: SBME projected non-linearity Lipschitz}, it is readily verified that $\smash{f^{(K,K')}}$ and $\smash{f^{(K')}}$ are uniformly Lipschitz continuous in the $x$ variable relative to the normalized-$\smash{\ell^1}$ norm with Lipschitz constant at most $\smash{L=\norm{\psi}_{\mathrm{Lip},\TV}}$. Indeed, for any $t>0$ and all $\smash{x,x'\in \Rp^{\D_{K'}}}$,
$$\abs{f^{(K,K')}(t,x)-f^{(K,K')}(t,x')}\leq \norm{\psi}_{\mathrm{Lip},\TV}\Norm{P^{(K,K')}x-P^{(K,K')}x'}_1\leq \norm{\psi}_{\mathrm{Lip},\TV}\Norm{x-x'}_1.$$
If $V=\Norm{\H_{K}}_{\mathrm{Lip},1,*}$, then the comparison principle in \Cref{SBME comparison principle on Rpp} implies that for any $R'\in \R$, the map
\begin{equation}\label{eqn: SBME convergence of projections comparison}
(t',x')\mapsto f^{(K,K')}(t',x')-f^{(K')}(t',x')-(2L+1)\big(\Norm{x'}_1+Vt'-R'\big)_+-\EE_Kt'
\end{equation}
achieves its supremum on $\{0\}\times \Rp^{\D_{K'}}$. We now choose $\smash{R'=\Norm{x^{(K')}(\mu)}_1+Vt}$ and distinguish two cases. On the one hand, if $t'=0$ and $\smash{\Norm{x'}_1\geq (2L+1)R'}$, then \eqref{eqn: SBME convergence of projections comparison} is bounded by
\begin{equation}\label{eqn: SBME convergence of projections comparison 1}
2L\Norm{x'}_1-(2L+1)(\Norm{x'}_1-R')=R'-\Norm{x'}_1\leq 0,
\end{equation}
where we have used the fact that $f^{(K,K')}(0,0)=f^{(K')}(0,0)$. On the other hand, if $t'=0$ and $\Norm{x'}_1\leq (2L+1)R'$, then \eqref{SBME H W} implies that \eqref{eqn: SBME convergence of projections comparison} is bounded by
\begin{equation}\label{eqn: SBME convergence of projections comparison 2}
\big\lvert \psi^{(K')}(x')-\psi^{(K)}(P^{(K,K')}x')\big\rvert\leq \norm{\psi}_{\mathrm{Lip},W}\Norm{x'}_1W\big(\bar \mu^{(K')}_{x'}, \bar \mu^{(K)}_{P^{(K,K')}x'}\big).
\end{equation}
To estimate this Wasserstein distance, fix a Lipschitz function $h:[-1,1]\to \R$ with $\norm{h}_{\text{Lip}}\leq 1$ and observe that
\begin{align*}
\Big\lvert \int_{-1}^1 h(y) \ud \big(\mu^{(K')}_{x'}-\mu^{(K)}_{P^{(K,K')}x'}\big)(y)\Big\rvert&\leq \frac{1}{\abs{\D_{K'}}}\sum_{k\in\D_K}\sum_{\ell=0}^{2^{K'-K-1}}x'_{k+\frac{\ell}{2^{K'}}}\Big\lvert h\Big(k+\frac{\ell}{2^{K'}}\Big)-h(k)\Big\rvert\\
&\leq \frac{1}{\abs{\D_{K'}}}\sum_{k\in\D_K}\sum_{\ell=0}^{2^{K'-K-1}}x'_{k+\frac{\ell}{2^{K'}}}\frac{\ell}{2^{K'}}\leq \frac{\Norm{x'}_1}{2^K}.
\end{align*}
Taking the supremum over all such $h$ and recalling \eqref{eqn: SBME convergence of projections comparison 2} shows that \eqref{eqn: SBME convergence of projections comparison} is bounded by
\begin{equation}\label{eqn: SBME convergence of projections comparison bound}
\frac{\norm{\psi}_{\mathrm{Lip},W}(2L+1)R'}{2^K}
\end{equation}
whenever $t'=0$ and $\Norm{x'}_1\leq (2L+1)R'$. Combining this with \eqref{eqn: SBME convergence of projections comparison 1} reveals that the map \eqref{eqn: SBME convergence of projections comparison} is uniformly bounded by \eqref{eqn: SBME convergence of projections comparison bound}.  Choosing $t'=t$ and $x'=x^{(K')}(\mu)$, and recalling the choice of $R'$ yields
$$f^{(K,K')}\big(t,x^{(K')}(\mu)\big)-f^{(K')}\big(t,x^{(K')}(\mu)\big)\leq \frac{\norm{\psi}_{\mathrm{Lip},W}(2L+1)}{2^K}\big(\Norm{x^{(K')}(\mu)}_1+Vt\big)+\EE_Kt.$$
Together with \eqref{eqn: SBME projected solution Cauchy} and an identical argument with the roles of $f^{(K,K')}$ and $f^{(K')}$ reversed, this implies that
$$\big\lvert f^{(K')}\big(t,x^{(K')}(\mu)\big)-f^{(K)}\big(t,x^{(K)}(\mu)\big)\big\rvert\leq \frac{\norm{\psi}_{\mathrm{Lip},W}(2L+1)}{2^K}\big(\mu[-1,1]+Vt\big)+\EE_Kt.$$
Since $V=\Norm{\H_{K}}_{\mathrm{Lip},1,*}$ is independent of $K$ by \Cref{SBME extending the non-linearity} and $\EE_K$ tends to zero as $K$ tends to infinity, the sequence $\smash{(f^{(K)}(t,x^{(K)}(\mu)))_K}$ is Cauchy. This establishes the existence of the limit \eqref{eqn: SBME limit of projections} for each fixed $\smash{R>\norm{\psi}_{\mathrm{Lip},\TV}}$. All that remains is to show that this limit is independent of $R$.\\
\step{3: independence on $R$.}\\
To show that the limit \eqref{eqn: SBME limit of projections} is independent of $R$, fix $\smash{R'>R>\norm{\psi}_{\mathrm{Lip},\TV}}$ as well as $K\geq 1$ large enough. The idea will be to show that, up to an error vanishing with $K$, the function $\smash{f_R^{(K)}}$ satisfies the Hamilton-Jacobi equation defining $\smash{f_{R'}^{(K)}}$. The equality of the limit \eqref{eqn: SBME limit of projections} associated with $R$ and $R'$ will then follow from the comparison principle in \Cref{SBME comparison principle on Rpp}. Consider $\smash{\phi\in C^\infty\big((0,\infty)\times \Rpp^{\D_K}\big)}$ with the property that $\smash{f_R^{(K)}-\phi}$ achieves a local maximum at the point $\smash{(t^*,x^*)\in (0,\infty)\times \Rpp^{\D_K}}$. Since $\smash{f_R^{(K)}}$ is a viscosity subsolution to the Hamilton-Jacobi equation \eqref{eqn: SBME HJ eqn on Rpp} associated with the non-linearity $\H_{K,R}$,
$$\big(\partial_t \phi-\H_{K,R}(\nabla \phi)\big)(t^*,x^*)\leq 0.$$
The fact that $\smash{f_R^{(K)}}$ has its gradient in the set $\smash{\K_{a,K}'}$ together with \eqref{eqn: SBME projected solution Lipschitz}, \Cref{SBME gradient bound on test function} and \Cref{SBME gradient of test function in cone} implies that
$$\nabla \phi(t^*,x^*)\in \K_{a,K}' \quad \text{and} \quad \dNorm{1}{\nabla \phi(t^*,x^*)}\leq \norm{\psi}_{\mathrm{Lip},\TV}.$$
It is therefore possible to find $\smash{u\in \Rp^{\D_K}}$ and $\smash{w\in \R^{\D_K}}$ with
$$\nabla \phi(t^*,x^*)=G^{(K)}u+w, \quad \Norm{u}_1\leq a \quad \text{and} \quad \dNorm{1}{w}\leq \frac{1}{2^{K/2}}.$$
Observe that
$$\dNorm{1}{G^{(K)}u}\leq \dNorm{1}{\nabla \phi(t^*,x^*)}+\dNorm{1}{w}\leq \norm{\psi}_{\mathrm{Lip},\TV}+\frac{1}{2^{K/2}},$$
so increasing $K$ if necessary, it is possible to ensure that $\smash{G^{(K)}u\in \CC_K\cap B_R\subset \CC_K\cap B_{R'}}$. It follows by the Lipschitz continuity of $\H_{K,R}$ established in \Cref{SBME extending the non-linearity} that
\begin{align*}
\big(\partial_t\phi-\H_{K,R'}(\nabla \phi)\big)(t^*,x^*)&\leq \partial_t\phi(t^*,x^*)-\C_K\big(G^{(K)}u\big)+\frac{8R'M}{2^{K/2}m^2}\\
&\leq \big(\partial_t\phi-\H_{K,R}(\nabla \phi)\big)(t^*,x^*)+\frac{8(R'+R)M}{2^{K/2}m^2}\leq \EE_K
\end{align*}
for the error term
$$\EE_K=\frac{8(R'+R)M}{2^{K/2}m^2}.$$
In particular, the function $\smash{(t,x)\mapsto f_R^{K}(t,x)-\EE_K t}$ is a viscosity subsolution to the Hamilton-Jacobi equation \eqref{eqn: SBME HJ eqn on Rpp} defining $\smash{f_{R'}^{K}}$. An identical argument shows that $\smash{(t,x)\mapsto f_R^{K}(t,x)+\EE_K t}$ is a viscosity supersolution to the Hamilton-Jacobi equation \eqref{eqn: SBME HJ eqn on Rpp} defining $\smash{f_{R'}^{K}}$. It follows by the comparison principle in \Cref{SBME comparison principle on Rpp} that for every $\mu \in \M_+$ and $t\geq 0$,
$$\big\lvert f_R^{(K)}\big(t,x^{(K)}(\mu)\big)-f_{R'}^{(K)}\big(t,x^{(K)}(\mu)\big)\big\rvert\leq \EE_Kt.$$
Letting $K$ tend to infinity completes the proof.
\end{proof}

\section{Approximate Hopf-Lax formula in finite dimensions}\label{SBME section Hopf-Lax}

In this section, we revisit the Hamilton-Jacobi equation studied in \Cref{SBME section HJ on cone}, and under the additional assumption that the matrix $\smash{G\in \R^{d\times d}}$ in \eqref{eqn: SBME min and max of G} is non-negative definite, we establish an approximate Hopf-Lax formula for the unique solution to the Hamilton-Jacobi equation \eqref{eqn: SBME HJ eqn on Rpp}. By an approximate Hopf-Lax formula we mean that the error between the Hopf-Lax function we will define and the solution to the Hamilton-Jacobi equation constructed in \Cref{SBME WP of HJ eqn on Rpp} tends to zero when the dimension $d$ tends to infinity. This will be used in the next section to establish an exact Hopf-Lax formula for the solution to the infinite-dimensional Hamilton-Jacobi equation \eqref{eqn: SBME infinite HJ eqn}. It will be convenient to introduce the bilinear form
\begin{equation}
(x,y)_G=Gx\cdot y
\end{equation}
associated with the non-negative definite matrix $G$, as well as its induced semi-norm
\begin{equation}
\norm{x}_G=\sqrt{(x,x)_G}.
\end{equation}
In this notation, the non-linearity \eqref{eqn: SBME non-linearity on cone} may be written as
\begin{equation}\label{eqn: SBME non-linearity on cone convex}
\C(Gx)=\frac{1}{2}Gx\cdot x=\frac{1}{2}\norm{x}_G^2.
\end{equation}
In particular, the non-linearity \eqref{eqn: SBME non-linearity on cone} is a convex function. This convexity property will allow us to establish an approximate Hopf-Lax formula for the Hamilton-Jacobi equation \eqref{eqn: SBME HJ eqn on Rpp}. We define the Hopf-Lax function $\smash{\fHL:[0,\infty)\times \Rp^d\to \R}$ by
\begin{equation}\label{eqn: SBME Hopf-Lax formula}
\fHL(t,x)=\sup_{y\in \Rp^d}\bigg\{\psi(x+y)-\frac{\|y\|_G^2}{2t}\bigg\}.
\end{equation}
The main result of this section is the following.

\begin{proposition}[Hopf-Lax]\label{SBME Hopf-Lax}
Fix an initial condition $\smash{\psi:\Rp^d\to\R}$ satisfying \eqref{eqn: SBME initial condition Lipschitz assumption} with its gradient in the set $\smash{\K_a'}$, let $\smash{R>\Norm{\psi}_{\mathrm{Lip},1}}$, and denote by $\smash{f:[0,\infty)\times \Rp^d\to \R}$ the unique solution to the Hamilton-Jacobi equation \eqref{eqn: SBME HJ eqn on Rpp} constructed in \Cref{SBME WP of HJ eqn on Rpp}. If $\smash{G\in \R^{d\times d}}$ is non-negative definite, then for all $\smash{(t,x)\in [0,\infty)\times \Rp^d}$,
\begin{equation}\label{eqn: SBME Hopf-Lax}
\abs{ f(t,x)-\fHL(t,x)} \leq \frac{t}{\sqrt{d}}\bigg(R+a+\frac{8RM}{m^2}\bigg).
\end{equation}
\end{proposition}
To prove this result, we first verify that the convex dual of the mapping $y \mapsto \|y\|_G^2/2$ is the non-linearity $\C$. We next show that the function \eqref{eqn: SBME Hopf-Lax formula} satisfies the right initial condition, and that the supremum in its definition is attained. We next argue that this function satisfies a semigroup property, and deduce that it belongs to $\smash{\mathfrak{L}_{\mathrm{unif}}}$. Finally we show that, in a sense to be made precise, it is an approximate solution to the the Hamilton-Jacobi equation \eqref{eqn: SBME HJ eqn on Rpp}. The estimate \eqref{eqn: SBME Hopf-Lax} will then be a consequence of the comparison principle in \Cref{SBME comparison principle on Rpp}. It will be convenient to note that for every $z\in \Rp^d$, we have
\begin{equation}  
\label{e.easy.jensen}
\|z\|_G^2 \ge \frac{m}{d^2} \sum_{k,k' = 1}^d z_k \, z_{k'} = m \, \Norm{z}_1^2.
\end{equation}

\begin{lemma}  
\label{l.convex.dual}
If $\smash{G\in \R^{d\times d}}$ is non-negative definite, then for every $z \in \CC$, 
\begin{equation}  
\C(z) = \sup_{y \in \Rp^d} \bigg\{  y \cdot z - \frac{\|y\|_G^2}{2} \bigg\} .
\end{equation}
Moreover, the supremum is attained at any point $x\in \Rp^d$ with $z=Gx$.
\end{lemma}
\begin{proof}
We can represent each $z \in \CC$ in the form of $Gx$ for some $x \in \Rp^d$. Using that $G$ is non-negative definite, we can appeal to the Cauchy-Schwarz inequality to assert that
\begin{equation*} 
y \cdot Gx =(x,y)_G \le \|x\|_G \, \|y\|_G \le \frac 1 2 \|x\|_G^2 + \frac 1 2 \|y\|_G^2.
\end{equation*}
We thus obtain that 
\begin{equation*}  
\C(Gx) = \frac 1 2 \|x\|_G^2 \ge \sup_{y \in \Rp^d} \Ll(  y \cdot Gx - \frac{\|y\|_G^2}{2} \Rr) .
\end{equation*}
For the converse inequality, we simply test the supremum with $y = x$.
\end{proof}

\begin{lemma}\label{SBME Hopf-Lax initial condition}
If $\smash{G\in \R^{d\times d}}$ is non-negative definite and $\psi:\Rp^d\to \R$ is a Lipschitz continuous initial condition with $\smash{\nabla \psi\in L^\infty(\Rp^d;\K_a')}$, then for every $\smash{x \in \Rp^d}$,
\begin{equation}
\fHL(0,x)=\psi(x).
\end{equation}
\end{lemma}
\begin{proof}
For $t = 0$, we interpret the definition of $f$ as
\begin{equation}  
\label{e.def.f0.sup}
\fHL(0,x) = \sup_{\substack{y \in \Rp^d \\ \|y\|_G = 0}} \psi(x+y).
\end{equation}
Recalling \eqref{e.easy.jensen}, we see that the only $y\in \Rp^d$ with $\|y\|_G = 0$ is $y=0$. Together with \eqref{e.def.f0.sup}, this completes the proof.
\end{proof}

It will slightly simplify our arguments below to notice that the supremum in \eqref{eqn: SBME Hopf-Lax formula} is achieved.
\begin{lemma}\label{SBME Hopf-Lax maximizer}
Fix an initial condition $\smash{\psi:\Rp^d\to \R}$ satisfying \eqref{eqn: SBME initial condition Lipschitz assumption}. If $\smash{G\in \R^{d\times d}}$ is non-negative definite, then for any point $\smash{(t,x)\in (0,\infty)\times \Rp^d}$, there exists $\smash{y\in \Rp^d}$ with
\begin{equation}
\fHL(t,x)=\psi(x+y)-\frac{\|y\|_G^2}{2t}.
\end{equation}
\end{lemma}
\begin{proof}
Combining \eqref{e.easy.jensen} with the Lipschitz continuity of $\psi$ reveals that
$$\psi(x+y) - \frac{\|y\|_G^2}{2t} \le \psi(x) + \Norm{y}_1 \Big(\Norm{\psi}_{\mathrm{Lip},1}  - \frac{m}{2t} \Norm{y}_1\Big).$$
We can thus restrict the supremum in \eqref{eqn: SBME Hopf-Lax formula} to those $y$'s in $\Rp^d$ that satisfy $\Norm{y}_1 \le \frac{2t}{m} \Norm{\psi}_{\mathrm{Lip},1}$. Since we are now optimizing a continuous function over a compact set, it is clear that the supremum is achieved. 
\end{proof}

\begin{lemma}[Semigroup property]\label{SBME Hopf-Lax semigroup property}
Fix an initial condition $\smash{\psi:\Rp^d\to \R}$ satisfying \eqref{eqn: SBME initial condition Lipschitz assumption}. If $\smash{G\in \R^{d\times d}}$ is non-negative definite, then for every pair $t>s>0$ and $\smash{x\in \Rp^d}$,
\begin{equation}
\fHL(t,x)=\sup_{y\in \Rp^d}\bigg\{\fHL(s,x+y)-\frac{\|y\|_G^2}{2(t-s)}\bigg\}.
\end{equation}
\end{lemma}

\begin{proof}
Fix $y,z\in \Rp^d$. Since $\|\cdot\|_G^2$ is a convex mapping, we have
$$\Ll\|\frac{y+z}{t}\Rr\|_G^2\leq \frac{s}{t}\Ll\|\frac{y}{s}\Rr\|_G^2  + \frac{t-s}{t}\Ll\|\frac{z}{t-s}\Rr\|_G^2.$$
Substituting this bound into \eqref{eqn: SBME Hopf-Lax formula} yields
$$\fHL(t,x)\geq \psi(x+y+z)-\frac t 2  \Ll\|\frac{y+z}{t}\Rr\|_G^2\geq \psi(x+y+z)-\frac s 2 \Ll\|\frac{y}{s}\Rr\|_G^2-\frac {t-s}{2} \Ll\|\frac{z}{t-s}\Rr\|_G^2.$$
Taking the supremum over all $y\in \Rp^d$ gives
$$\fHL(t,x)\geq \fHL(s,x+z)-\frac{\|z\|_G^2}{2(t-s)},$$
and taking the supremum over all $z\in \Rp^d$ establishes the lower bound
$$\fHL(t,x)\geq \sup_{y\in \Rp^d}\bigg\{\fHL(s,x+y)-\frac{\|y\|_G^2}{2(t-s)}\bigg\}.$$
To obtain the matching upper bound, we invoke \Cref{SBME Hopf-Lax maximizer} to find $y\in \Rp^d$ such that
$$\fHL(t,x)=\psi(x+y)-\frac{\|y\|_G^2}{2t}.$$
Defining $z=\frac{t-s}{t}y\in \Rp^d$, we observe that
\begin{equation}\label{eqn: SBME Hopf-Lax semigroup key identities}
\frac{z}{t-s}=\frac{y}{t}=\frac{y-z}{s}.
\end{equation}
In particular, testing the supremum in \eqref{eqn: SBME Hopf-Lax formula} with $y-z=\frac{s}{t}y\in \Rp^d$ gives
$$\fHL(s,x+z)\geq \psi(x+z+y-z)-\frac{\|y-z\|_G^2}{2s}=\psi(x+y)-\frac{\|y-z\|_G^2}{2s},
$$
and thus, using also \eqref{eqn: SBME Hopf-Lax semigroup key identities}, we obtain
\begin{align*}
\fHL(s,x+z)-\frac{\|z\|_G^2}{2(t-s)}
& \geq \psi(x+y)-\frac{\|y-z\|_G^2}{2s}-\frac{\|z\|_G^2}{2(t-s)}
\\
& = \psi(x+y) - \frac{\|y\|_G^2}{2t}.
\end{align*}
Taking the supremum over $z\in \Rp^d$ and then over $y\in \Rp^d$ completes the proof.
\end{proof}
We next prove some regularity properties of the function $\fHL$ in \eqref{eqn: SBME Hopf-Lax formula}.
\begin{lemma}\label{SBME Hopf-Lax formula in L}
Fix an initial condition $\smash{\psi:\Rp^d\to \R}$ satisfying \eqref{eqn: SBME initial condition Lipschitz assumption}. If $\smash{G\in \R^{d\times d}}$ is non-negative definite, then $\fHL\in \mathfrak{L}_{\mathrm{unif}}$ with
\begin{equation}\label{eqn: SBME Hopf-Lax formula in L}
\sup_{t>0}\Norm{\fHL(t,\cdot)}_{\mathrm{Lip},1}\leq \Norm{\psi}_{\mathrm{Lip},1} \quad \text{and} \quad [\fHL]_0\leq \frac{\Norm{\psi}_{\mathrm{Lip},1}^2}{2m}.
\end{equation}
\end{lemma}

\begin{proof}
Fix $(t,x,x')\in (0,\infty)\times \Rp^d\times \Rp^d$ and invoke \Cref{SBME Hopf-Lax maximizer} to find $y\in \Rp^d$ with
$$\fHL(t,x)=\psi(x+y)-\frac{\|y\|_G^2}{2t}.$$
Taking this $y\in \Rp^d$ in \eqref{eqn: SBME Hopf-Lax formula} gives the lower bound
$$\fHL(t,x')\geq \psi(x'+y)-\frac{\|y\|_G^2}{2t},$$
and thus
$$\fHL(t,x)-\fHL(t,x')\leq \psi(x+y)-\psi(x'+y)\leq \Norm{\psi}_{\mathrm{Lip},1}\Norm{x-x'}_1.$$
Reversing the roles of $x$ and $x'$ gives $y'\in \Rp^d$ with
$$\fHL(t,x')-\fHL(t,x)\leq \psi(x'+y')-\psi(x+y')\leq \Norm{\psi}_{\mathrm{Lip},1}\Norm{x-x'}_1.$$
Combining these two bounds yields the first inequality in \eqref{eqn: SBME Hopf-Lax formula in L}. To establish Lipschitz continuity in time, fix $x\in \Rp^d$ as well as $t>s\geq 0$. The semigroup property in \Cref{SBME Hopf-Lax semigroup property} with $y=0$ implies that
\begin{equation}\label{eqn: SBME Hopf-Lax Lipschitz time 1}
\fHL(t,x)\geq \fHL(s,x).
\end{equation}
Using \Cref{SBME Hopf-Lax semigroup property} in combination with the first inequality in \eqref{eqn: SBME Hopf-Lax formula in L} and \eqref{e.easy.jensen} gives
\begin{equation*}
\fHL(t,x)  \leq \fHL(s,x)+\sup_{y\in \Rp^d}\bigg\{\Norm{\psi}_{\mathrm{Lip},1}\Norm{y}_1-\frac{m\Norm{y}_1^2}{2(t-s)}\bigg\}
 \le \fHL(s,x) + \frac{\Norm{\psi}_{\mathrm{Lip},1}^2}{2m}(t-s),
\end{equation*}
where we have used the fact that $r\mapsto r-\frac{1}{2}ar^2$ achieves its maximum at $r=1/a$. Combining this with \eqref{eqn: SBME Hopf-Lax Lipschitz time 1} completes the proof.
\end{proof}

We are finally in a position to prove \Cref{SBME Hopf-Lax}.

\begin{proof}[Proof of \Cref{SBME Hopf-Lax}.]
Denote by $\EE_d$ an error term that will be defined in the course of the proof. We will proceed in three steps; first, we will show that the function $\smash{f_+(t,x)= \fHL(t,x)+\EE_dt}$ is a viscosity supersolution to the Hamilton-Jacobi equation \eqref{eqn: SBME HJ eqn on Rpp}, then we will show that the function $\smash{f_-(t,x)= \fHL(t,x)-\EE_dt}$ is a viscosity subsolution to the Hamilton-Jacobi equation \eqref{eqn: SBME HJ eqn on Rpp}, and finally we will conclude using the comparison principle in \Cref{SBME comparison principle on Rpp}.\\
\step{1: $f_+$ viscosity supersolution.}\\
Consider a smooth function $\smash{\phi\in C^\infty\big((0,\infty)\times \Rpp^d\big)}$ with the property that $\smash{f_+-\phi}$ has a local minimum at $\smash{(t^*,x^*)\in (0,\infty)\times \Rpp^d}$. Using \Cref{SBME gradient in convex set}, it is readily verified that $\fHL$ has its gradient in $\smash{\K_a'}$ as $\psi$ does. It follows by \Cref{SBME gradient of test function in cone} that $\nabla \phi(t^*,x^*)\in \K_a'$. It is therefore possible to find $\smash{u\in \Rp^d}$ and $\smash{w\in \R^d}$ with
$$\nabla \phi(t^*,x^*)=Gu+w, \quad \Norm{u}_1\leq a \quad \text{and} \quad \dNorm{1}{w}\leq \frac{1}{\sqrt{d}}.$$
On the one hand, if $s>0$ is sufficiently small that $t^*-s>0$, then
$$f_+(t^*-s,x^*+su)-\phi(t^*-s,x^*+su)\geq f_+(t^*,x^*)-\phi(t^*,x^*).$$
On the other hand, taking $su\in \Rp^d$ in \Cref{SBME Hopf-Lax semigroup property} reveals that
$$\fHL(t^*,x^*)\geq \fHL(t^*-s, x^*+su)-s \frac{\|u\|_G^2}{2} .$$
It follows that
$$\phi(t^*,x^*)-\phi(t^*-s,x^*+su)+s \frac{\|u\|_G^2}{2}-\EE_ds  \geq 0.$$
Dividing by $0<s<t^*$ and letting $s\to 0$ yields
$$\partial_t\phi(t^*,x^*)-u\cdot \nabla \phi(t^*,x^*)+\frac{\|u\|_G^2}{2}-\EE_d\geq 0.$$
Recalling that $\nabla \phi(t^*,x^*)=Gu+w$ and using \Cref{l.convex.dual}, we obtain 
$$\partial_t\phi(t^*,x^*)-\C(Gu)-u\cdot w-\EE_d\geq 0.$$
By \Cref{SBME gradient bound on test function} and \Cref{SBME Hopf-Lax formula in L}, we have
$$\dNorm{1}{Gu}\leq \dNorm{1}{\nabla \phi(t^*,x^*)}+\frac{1}{\sqrt{d}}\leq \Norm{\psi}_{\mathrm{Lip},1}+\frac{1}{\sqrt{d}}\leq R,$$
so the Lipschitz continuity of the non-linearity $\H_R$ established in \Cref{SBME extending the non-linearity} implies that
$$\big(\partial_t\phi-\H_R(\nabla \phi)\big)(t^*,x^*)\geq \EE_d-\Norm{u}_1\dNorm{1}{w}-\frac{8RM}{m^2\sqrt{d}}\geq \EE_d-\frac{a}{\sqrt{d}}-\frac{8RM}{m^2\sqrt{d}}.$$
This shows that $f_+$ is a supersolution to the Hamilton-Jacobi equation \eqref{eqn: SBME HJ eqn on Rpp} provided that
$$\EE_d\geq \frac{a}{\sqrt{d}}+\frac{8RM}{m^2\sqrt{d}}.$$
\step{2: $f_-$ viscosity subsolution.}\\
Consider a smooth function $\smash{\phi\in C^\infty\big((0,\infty)\times \Rpp^d\big)}$ with the property that $\smash{f_--\phi}$ has a local maximum at $\smash{(t^*,x^*)\in (0,\infty)\times \Rpp^d}$. Since $\fHL$ has its gradient in $\smash{\K_a'}$ as $\psi$ does, we have $\nabla \phi(t^*,x^*)\in \K_a'$ by \Cref{SBME gradient of test function in cone}. It is therefore possible to find $\smash{u\in \Rp^d}$ and $\smash{w\in \R^d}$ with
$$\nabla \phi(t^*,x^*)=Gu+w, \quad \Norm{u}_1\leq a \quad \text{and} \quad \dNorm{1}{w}\leq \frac{1}{\sqrt{d}}.$$
Suppose for the sake of contradiction that there exists $\delta>0$ with
$$\big(\partial_t\phi-\H_R(\nabla \phi)\big)(t^*,x^*)\geq \delta>0.$$
Arguing as in the previous step, this implies that
$$\partial_t\phi(t^*,x^*)-\C(Gu)\geq \delta-\frac{8RM}{m^2\sqrt{d}}.$$
By \Cref{l.convex.dual}, this may be recast as the assumption that
$$\partial_t\phi(t^*,x^*)-y\cdot Gu+\frac{\norm{y}_G^2}{2}\geq \delta-\frac{8RM}{m^2\sqrt{d}}$$
for all $y\in \Rp^d$. By continuity, of $\partial_t\phi$ and $\nabla\phi$, up to redefining $\delta>0$, we may in fact assume that
\begin{equation}\label{eqn: Hopf-Lax formula subsolution absurd}
\partial_t\phi(t',x')-y\cdot \nabla \phi(t',x')+\frac{\|y\|_G^2}{2}\geq \delta-\frac{8RM}{m^2\sqrt{d}}-\Norm{y}_1\dNorm{1}{w}
\end{equation}
for all $y\in \Rp^d$ and $(t',x')$ sufficiently close to $(t^*,x^*)$. Recalling \Cref{SBME Hopf-Lax semigroup property} and arguing as in the proof of \Cref{SBME Hopf-Lax maximizer}, it is possible to find $R>0$ such that, for every $s>0$ sufficiently small, there exists $y_s\in \Rp^d$ with $\Norm{y_s}_1\leq Rs$ and
$$\fHL(t^*,x^*)=\fHL(t^*-s,x^*+y_s)-\frac{\|y_s\|_G^2}{2s}.$$
It follows by the fundamental theorem of calculus and the absurd assumption \eqref{eqn: Hopf-Lax formula subsolution absurd} used with $y=\frac{y_s}{s} \in \Rp^d$ that
\begin{align*}
\phi(t^*,x^*)-\phi(t^*-s,x^*+y_s)&=\int_0^1 \frac{\mathrm{d}}{\mathrm{d} r}\phi\big(rt^*+(1-r)(t^*-s),rx^*+(1-r)(x^*+y_s)\big)\ud r\\
&=\int_0^1 \big(s\partial_t\phi-y_s\cdot \nabla \phi\big)(t^*+(r-1)s, x^*+(1-r)y_s)\ud r\\
&\geq s\delta -\frac{\|y_s\|_G^2}{2s}-s\frac{8RM}{m^2\sqrt{d}}-\Norm{y_s}_1\dNorm{1}{w}\\
&\geq \fHL(t^*,x^*)-\fHL(t^*-s,x^*+y_s)+s\bigg(\delta-\frac{8RM}{m^2\sqrt{d}}-\frac{R}{\sqrt{d}}\bigg).
\end{align*}
Rearranging shows that for $s$ sufficiently small,
$$f_-(t^*-s,x^*+y_s)-\phi(t^*-s,x^*+y_s)\geq s\bigg(\delta-\frac{8RM}{m^2\sqrt{d}}-\frac{R}{\sqrt{d}}+\EE_d\bigg) +f_-(t^*,x^*)-\phi(t^*,x^*).$$
This contradicts the fact that $f-\phi$ admits a local maximum at $(t^*,x^*)$ provided that
$$\EE_d\geq \frac{R}{\sqrt{d}}+\frac{8RM}{m^2\sqrt{d}}.$$
\step{3: comparison principle.}\\
Combining step 1 and step 2 shows that, if we define
$$\EE_d=\frac{1}{\sqrt{d}}\bigg(R+a+\frac{8RM}{m^2}\bigg),$$
then $\smash{f_+}$ is a viscosity supersolution to the Hamilton-Jacobi equation \eqref{eqn: SBME HJ eqn on Rpp} while $\smash{f_-}$ is a viscosity subsolution to this equation. Together with \Cref{SBME Hopf-Lax initial condition}, \Cref{SBME Hopf-Lax formula in L} and the comparison principle in \Cref{SBME comparison principle on Rpp}, this implies that for every ${(t,x)\in [0,\infty)\times \Rp^d}$,
$$\abs{\fHL(t,x)-f(t,x)}\leq \EE_dt$$
as required.
\end{proof}

\section{Hopf-Lax formula for the infinite-dimensional equation}\label{SBME section Hopf Lax limit}

In this section, we apply \Cref{SBME Hopf-Lax} to the projected Hamilton-Jacobi equation \eqref{eqn: SBME projected HJ eqn} and let $K$ tend to infinity in the resulting variational formula to establish \Cref{SBME infinite Hopf-Lax}. In addition to the assumptions \eqref{SBME H g}-\eqref{SBME H W}, we will suppose that the kernel $g:[-1,1]\to \R$ is non-negative definite in the sense that it satisfies \eqref{SBME H conv}. This assumption is equivalent to the non-negative definiteness of each of the matrices \eqref{eqn: SBME matrix GK}, and therefore to the convexity of each of the projected non-linearities~\eqref{eqn: SBME projected non-linearity}. In particular, \Cref{SBME Hopf-Lax} implies that the unique solution $\smash{f^{(K)}:[0,\infty)\times \R^{\D_K}\to \R}$ to the projected Hamilton-Jacobi equation \eqref{eqn: SBME projected HJ eqn} in $\mathfrak{L}_{\mathrm{unif}}$ subject to the initial condition $\smash{\psi^{(K)}}$ satisfies
\begin{equation}
f^{(K)}\big(t,x^{(K)}(\mu)\big)=\sup_{y\in \Rp^{\D_K}}\bigg\{\psi^{(K)}\big(x^{(K)}(\mu)+y\big)-\frac{\norm{y}_G^2}{2t}\bigg\}+\BigO\big(t\abs{\D_K}^{-1/2}\big).
\end{equation}
Remembering the definition of the projected initial condition \eqref{eqn: SBME projected initial condition}, the projected non-linearity \eqref{eqn: SBME projected non-linearity} and its relationship \eqref{eqn: SBME normalized norms} to the non-linearity $\C_\infty$ in \eqref{eqn: SBME infinite non-linearity} shows that
\begin{equation}
f^{(K)}\big(t,x^{(K)}(\mu)\big)=\sup_{\nu\in \M_+^{(K)}}\bigg\{\psi(\mu+\nu)-\frac{\C_\infty(G_\nu)}{t}\bigg\}+\BigO\big(t\abs{\D_K}^{-1/2}\big),
\end{equation}
where we have made the substitution $\smash{y=x^{(K)}(\nu)}$ for $\smash{\nu\in \M_+^{(K)}}$.
Using \Cref{SBME infinite HJ eqn WP} and a simple continuity argument to let $K$ tend to infinity in this expression gives the variational representation formula
\begin{align}
f(t,\mu)&=\sup_{\nu\in \M_+}\bigg\{\psi(\mu+\nu)-\frac{1}{2t}\int_{-1}^1 G_\nu(y)\ud \nu(y)\bigg\}\\
&=\sup_{\nu\in \M_+}\bigg\{\psi(\mu+t\nu)-\frac{t}{2}\int_{-1}^1 G_\nu(y)\ud \nu(y)\bigg\}.\label{eqn: SBME infinite Hopf-Lax unconstrained}
\end{align}
The second of these expressions follows from the first by setting $\nu'=t\nu$. To establish \Cref{SBME infinite Hopf-Lax}, we need to show that the supremum in \eqref{eqn: SBME infinite Hopf-Lax unconstrained} is achieved at some $\nu^*\in \M_+$, and that whenever the initial condition admits a Gateaux derivative at the measure $\mu+t\nu^*$ with density $x\mapsto D_\mu \psi(\mu+t\nu^*,x)$ in $\CC_\infty$,
\begin{equation}
G_{\nu^*}=D_\mu(\mu+t\nu^*,\cdot).
\end{equation}
If we ignore the constraint that the optimizers in \eqref{eqn: SBME infinite Hopf-Lax unconstrained} must be non-negative measures, then this latter property is clear from the first order conditions on a maximizer. To prove this rigorously, we first show that a maximizer exists, and we then establish a Cauchy-Schwarz inequality for the non-negative definite kernel $\smash{\tilde{g}(x,y)=g(xy)}$.

\begin{lemma}\label{SBME supremum in Parisi functional attained}
For every $t\geq 0$ and $\mu \in \M_+$, there exists $\nu^*\in \M_+$ with
\begin{equation}\label{eqn: SBME supremum in Parisi functional attained}
f(t,\mu)=\psi(\mu+t\nu^*)-\frac{t}{2}\int_{-1}^1 G_{\nu^*}(y)\ud \nu^*(y).
\end{equation}
\end{lemma}

\begin{proof}
Fix a probability measure $\nu\in \Pr[-1,1]$ and a positive constant $\lambda>0$. The Lipschitz continuity \eqref{SBME H TV} of the initial condition implies that
$$\psi(\mu+\lambda t\nu)\leq \psi(\mu)+\norm{\psi}_{\mathrm{Lip},\TV}\TV(0,\lambda t\nu)\leq \psi(\mu)+2\lambda t\norm{\psi}_{\mathrm{Lip},\TV}.$$
On the other hand,
$$\int_{-1}^1 G_{\lambda \nu}(y)\ud (\lambda\nu)(y)=\lambda^2 \int_{-1}^1 \int_{-1}^1 g(xy)\ud \nu(x)\ud \nu(y)\geq \lambda^2m.$$
Combining these two bounds reveals that
$$\psi(\mu+t\nu)-\frac{t}{2}\int_{-1}^1 G_\nu(y)\ud \nu(y)\leq \psi(\mu)+2\lambda t \norm{\psi}_{\mathrm{Lip},\TV}-\frac{\lambda^2tm}{2}.$$
The supremum in \eqref{eqn: SBME infinite Hopf-Lax unconstrained} can therefore be restricted to measures in $\M_+$ with bounded total mass. The existence of a maximizer is now an immediate consequence of Prokhorov's theorem. Indeed, if $\smash{(\nu_n)\subset \M_+}$ denotes a maximizing sequence, we may assume without loss of generality that each measure in this sequence has total mass bounded by the same constant. It follows by Prokhorov's theorem that this sequence is pre-compact, and therefore admits a subsequential limit with respect to the weak convergence of measures. By continuity of the functional being maximized in \eqref{eqn: SBME infinite Hopf-Lax unconstrained}, this weak limit must be a maximizer. This completes the proof.
\end{proof}

\begin{lemma}\label{SBME Cauchy-Schwarz for g}
If $g$ satisfies \eqref{SBME H conv} and $\mu,\nu\in \M_s$ are signed measures, then
\begin{equation}
\bigg(\int_{-1}^1 G_\nu(x)\ud \mu(x)\bigg)^2\leq \bigg(\int_{-1}^1 G_\mu(x)\ud \mu(x)\bigg)\bigg(\int_{-1}^1 G_\nu(x)\ud \nu(x)\bigg).
\end{equation}
\end{lemma}

\begin{proof}
This is the Cauchy-Schwarz inequality for the non-negative definite kernel $\smash{\tilde{g}(x,y)=g(xy)}$, and can be proved in a standard way. Indeed, for every $t \in \R$, let
\begin{align*}
P(t)&=\int_{-1}^1 \int_{-1}^1 g(xy)\ud(\mu+t\nu)(x)\ud (\mu+t\nu)(y)\\
&=\int_{-1}^1 G_\mu(x)\ud \mu(x)+2t\int_{-1}^1 G_\nu(x)\ud \mu(x)+t^2\int_{-1}^1 G_\nu(x)\ud \nu(x).
\end{align*}
This polynomial is non-negative by \eqref{SBME H conv}. In particular, its discriminant cannot be positive. This means that
$$2^2\bigg(\int_{-1}^1 G_\nu(x)\ud \mu(x)\bigg)^2-4\bigg(\int_{-1}^1G_\mu(x)\ud\mu(x)\bigg)\bigg(\int_{-1}^1 G_\nu(x)\ud \nu(x)\bigg)\leq 0.$$
Rearranging completes the proof.
\end{proof}

We are finally in a position to prove \Cref{SBME infinite Hopf-Lax}.

\begin{proof}[Proof of \Cref{SBME infinite Hopf-Lax}.]
Fix $t>0$ and $\mu \in \M_+$. Combining \Cref{SBME infinite HJ eqn WP} with \eqref{eqn: SBME infinite Hopf-Lax unconstrained} shows that the unique solution to the infinite-dimensional Hamilton-Jacobi equation \eqref{eqn: SBME infinite HJ eqn} admits the Hopf-Lax variational representation \eqref{eqn: infinite Hopf-Lax}. Moreover, \Cref{SBME supremum in Parisi functional attained} ensures that the supremum in \eqref{eqn: SBME infinite Hopf-Lax unconstrained} is achieved at some $\nu^*\in \M_+$. To establish the final statement in \Cref{SBME infinite Hopf-Lax}, suppose that the initial condition $\psi$ admits a Gateaux derivative at the measure $\mu+t\nu^*$ with density $x\mapsto D_\mu \psi(\mu+t\nu^*,x)$ in $\CC_\infty$. For any measure $\eta\in\M_+$, the Gateaux derivative of the functional
$$\nu\mapsto \psi(\mu+t\nu)-\frac{t}{2}\int_{-1}^1 G_{\nu}(y)\ud \nu(y)$$
at the measure $\nu^*$ in the direction of $\eta-\nu^*$ is
\begin{equation}
\label{e.Gateaux.diff.optim}
    D_\mu\psi\big(\mu+t\nu^*; t(\eta-\nu^*)\big)-t\int_{-1}^1\int_{-1}^1 g(xy)\ud (\eta-\nu^*)(x)\ud \nu^*(y).
\end{equation}
Moreover, for every $\epsilon \in [0,1]$, we have that $\nu^* + \epsilon (\eta - \nu^*)$ belongs to $\M_+$, and is thus a valid candidate for the optimization problem in \eqref{eqn: SBME infinite Hopf-Lax unconstrained}. As a consequence, the quantity in \eqref{e.Gateaux.diff.optim} must be non-positive. Using also the definition of the Gateaux derivative density in \eqref{eqn: SBME Gateaux derivative density}, we get that
\begin{equation}\label{eqn: SBME first order condition}
t\int_{-1}^1 \Big(D_\mu \psi(\mu+t\nu^*,x)-\int_{-1}^1 g(xy)\ud \nu^*(y)\Big)\ud \big(\eta-\nu^*\big)(x)\leq 0,
\end{equation}
for every $\eta\in \M_+$. The assumption that the density $x\mapsto D_\mu \psi(\mu+t\nu^*,x)$ belongs to the cone $\CC_\infty$ gives a measure $\eta^*\in \M_+$ with
$\smash{G_{\eta^*}(x)=D_\mu\psi(\mu+t\nu^*,x)}$. Applying \eqref{eqn: SBME first order condition} to the measure $\eta=\eta^*$ reveals that
$$\int_{-1}^1\int_{-1}^1 g(xy)\ud (\eta^*-\nu^*\big)(y)\ud (\eta^*-\nu^*\big)(x)\leq 0.$$
Together with \eqref{eqn: SBME g non-negative definite}, this implies that 
$$\int_{-1}^1\int_{-1}^1 g(xy)\ud (\eta^*-\nu^*\big)(y)\ud (\eta^*-\nu^*\big)(x)= 0.$$
Applying \Cref{SBME Cauchy-Schwarz for g} to the signed measures $\eta^*-\nu^*$ and $\delta_{x}$ for some $x\in [-1,1]$ shows that
$$\int_{-1}^1 g(xy)\ud\big(\eta^*-\nu^*\big)(y)=\int_{-1}^1 g(yz)\ud\big(\eta^*-\nu^*\big)(y)\ud \delta_{x}(z)=0.$$
Rearranging gives $G_{\nu^*}(x)=G_{\eta^*}(x)=D_\mu(\mu+t\nu^*,x)$. Since $x\in [-1,1]$ is arbitrary, this completes the proof.
\end{proof}

\section{The infinite-dimensional equation with arbitrary kernel}\label{SBME section general HJ WP}

In this section, we extend the main results of this paper to the Hamilton-Jacobi equation \eqref{eqn: SBME infinite HJ eqn} associated with a kernel $g$ that is not necessarily assumed to be positive. Fix $b\in \R$, and recall the definition of the modified kernel $\smash{\td g_b}$ in \eqref{e.def.tdg}, of the modified Hamilton-Jacobi equation \eqref{eqn: SBME infinite HJ eqn with tildes} and of the modified solution $\smash{f_b}$ in \eqref{e.def.f.or.tdf}. Introduce the symmetric matrix
\begin{equation}
\widetilde{G}^{(K)}_b=\frac{1}{\abs{\D_K}^2}\big(\td g_b(kk')\big)_{k,k'\in \D_K}\in \R^{\D_K\times\D_K}
\end{equation}
the projected cone
\begin{equation}
\widetilde{\CC}_{b,K}=\Big\{\widetilde{G}_b^{(K)}x^{(K)}(\mu)\in \R^{\D_K}\mid \mu \in \M_+^{(K)}\Big\}=\Big\{\widetilde{G}_b^{(K)}x\in \R^{\D_K}\mid x\in \Rp^{\D_K}\Big\}
\end{equation}
and the projected non-linearity $\widetilde{\C}_{b,K}:\widetilde{\CC}_{b,K}\to \R$ defined by
\begin{equation}
\widetilde{\C}_{b,K}\big(\widetilde{G}^{(K)}_bx\big)=\frac{1}{2}\widetilde{G}_b^{(K)}x\cdot x=\frac{1}{2\abs{\D_K}^2}\sum_{k,k'\in \D_K}\td g_b(kk')x_kx_{k'}.
\end{equation}
Also define the closed convex set
\begin{equation}
\widetilde{\K}_{=a,b,K}=\Big\{\widetilde{G}_b^{(K)}x\in \R^{\D_K}\mid x\in \Rp^{\D_K}\text{ and } \Norm{x}_1= a\Big\}
\end{equation}
The first order of business will be to verify that the function $\smash{f_b}$ is well-defined by ensuring that \eqref{SBME H g}-\eqref{SBME H W} are satisfied in the context of the infinite-dimensional Hamilton-Jacobi equation \eqref{eqn: SBME infinite HJ eqn with tildes}.

\begin{lemma}\label{SBME H1-H4 for tilted kernel}
Under the assumptions of \Cref{SBME general infinite HJ eqn WP}, the kernel $\smash{\td g_b}$ in \eqref{e.def.tdg} and the initial condition $\smash{\td \psi_b}$ in \eqref{e.def.tdpsi} satisfy \eqref{SBME H g}-\eqref{SBME H W}. Moreover, each projected initial condition $\smash{\td \psi_b^{(K)}}$ has its gradient in $\smash{\td \K_{=a,b,K}}$.
\end{lemma}

\begin{proof}
The kernel $\td g_b$ satisfies \eqref{SBME H g} by the choice of $b$, while the initial condition $\td \psi_b$ satisfies \eqref{SBME H TV} by the triangle inequality and the bound
$$\Big\lvert ab\int_{-1}^1 \ud \mu-ab\int_{-1}^1 \ud \nu\Big\rvert\leq a\abs{b}\big\lvert \mu[-1,1]-\nu[-1,1]\big\rvert\leq a\abs{b}\TV(\mu,\nu).$$
An identical argument shows that the initial condition $\td \psi_b$ satisfies \eqref{SBME H W}. To verify \eqref{SBME H K'}, introduce the closed convex set
$$\widetilde{\K}_{a,b,K}=\Big\{\widetilde{G}_b^{(K)}x\in \R^{\D_K}\mid x\in \Rp^{\D_K}\text{ and } \Norm{x}_1\leq a\Big\},$$
and fix $c\in \R$ and $\smash{x,x'\in \R^{\D_K}}$ such that $(x'-x)\cdot z\geq c$ for every $\smash{z\in \widetilde{\K}_{a,b,K}'}$. Now fix $\smash{y\in \K_{=a,K}'}$, and represent it as $y=G^{(K)}u+w$ for some $\smash{u\in \Rp^d}$ and $\smash{w\in \R^d}$ with $\smash{\Norm{u}_1=a}$ and $\smash{\dNorm{1}{w}\leq 2^{-K/2}}$. If $\smash{z=\widetilde{G}^{(K)}_bu+w}$, then
$$z=G^{(K)}u+b\Norm{u}_1\iota_K+w=y+ab\iota_K$$
for the vector $\smash{\iota_K=(\abs{\D_K}^{-1})_{k\in\D_K}\in \Rp^{\D_K}}$. Since $\smash{z\in \td \K_{=a,b,K}\subset  \widetilde{\K}_{a,b,K}'}$,
$$(x'-x)\cdot y=(x'-x)\cdot z-(x'-x)\cdot ab\iota_K\geq c-(x'-x)\cdot ab\iota_K.$$
The assumption \eqref{SBME H a} and \Cref{SBME gradient in convex set} therefore imply that
$$\psi^{(K)}(x')-\psi^{(K)}(x)\geq c-(x'-x)\cdot ab\iota_K.$$
Noticing that $x\cdot \iota_K=\Norm{x}_1$ and rearranging reveals that
$$\td \psi^{(K)}_b(x')-\td \psi^{(K)}_b(x)\geq c.$$
Together with \Cref{SBME gradient in convex set}, this establishes \eqref{SBME H K'}. Notice that this argument only needed the assumption that $(x'-x)\cdot z\geq c$ for every $\smash{z\in \td\K_{=a,b,K}}$ so it also shows that $\smash{\td \psi_b^{(K)}}$ has its gradient in $\smash{\td \K_{=a,b,K}}$ by \Cref{SBME gradient in convex set}. This completes the proof.
\end{proof}

Together with \Cref{SBME WP of HJ eqn on Rpp} this result implies that for every $R>0$, the Hamilton-Jacobi equation
\begin{equation}\label{eqn: SBME projected HJ eqn b}
\partial_t \td f^{(K)}(t,x)=\widetilde{\H}_{b,K, R}\big(\nabla \td f^{(K)}(t,x)\big) \quad\text{on}\quad \Rpp\times \Rpp^{\D_K}
\end{equation}
admits a unique viscosity solution $\smash{\td f^{(K)}_{b,R}\in \mathcal{L}_{\mathrm{unif}}}$ subject to the initial condition $\smash{\td \psi_b^{(K)}}$ which satisfies the Lipschitz bound
\begin{equation}
\sup_{t\geq 0}\Norm{\td f_{b,R}^{(K)}(t,\cdot)}_{\mathrm{Lip},1}=\Norm{\widetilde{\psi}_b^{(K)}}_{\mathrm{Lip},1}.
\end{equation}
Here $\smash{\widetilde{\H}_{b,K,R}:\R^{\D_K}\to \R}$ denotes the extension of the non-linearity $\smash{\widetilde{\C}_{b,K}}$ provided by \Cref{SBME extending the non-linearity}. Since the projected initial condition $\smash{\td \psi_b^{(K)}}$ has its gradient in the closed convex set $\smash{\K'_{=a, b, K}}$ by \Cref{SBME H1-H4 for tilted kernel}, an identical argument to that in \Cref{SBME WP of HJ eqn on Rpp} shows that the solution $\smash{\td f^{(K)}_{b,R}}$ also has its gradient in $\smash{\K'_{=a, b, K}}$. Moreover, \Cref{SBME infinite HJ eqn WP} allows us to define the solution to the infinite-dimensional Hamilton-Jacobi equation \eqref{eqn: SBME infinite HJ eqn with tildes} by
\begin{equation}
\td f_b(t,\mu)=\lim_{K\to \infty}\td f^{(K)}_{b,R}\big(t,x^{(K)}(\mu)\big),
\end{equation}
and guarantees that this limit is independent of $R>0$ provided that $\smash{R>\norm{\td \psi_b}_{\mathrm{Lip},\TV}}$. Using the comparison principle in \Cref{SBME comparison principle on Rpp} we now show that the limit defining the function \eqref{e.def.f.or.tdf},
\begin{equation}
f_b(t,\mu)=\lim_{K\to \infty}\bigg(\td f_{b,R}^{(K)}\big(t,x^{(K)}(\mu)\big)-ab\Norm{x^{(K)}(\mu)}_1-\frac{a^2bt}{2}\bigg),
\end{equation}
is independent of $b$ by establishing \Cref{SBME general infinite HJ eqn WP}.

\begin{proof}[Proof of Theorem~\ref{SBME general infinite HJ eqn WP}]
Let $b,b'\in \R$ be such that the kernels $\td g_b$ and $\td g_{b'}$ are positive on $[-1,1]$, and fix $\smash{R>\norm{\td \psi_b}_{\mathrm{Lip},\TV}+\norm{\td \psi_{b'}}_{\mathrm{Lip},\TV}}$. The idea will be to show that the function
$$f^{(K)}_{b,b'}(t,x)=\td f_b^{(K)}(t,x)-a(b-b')\Norm{x}_1-\frac{a^2(b-b')t}{2}$$
satisfies the Hamilton-Jacobi equation defining $\smash{\td f_{b'}^{(K)}}$ up to an error vanishing with $K$. We have omitted the dependence on $R$, and will continue to do so throughout this proof, as this constant will remain fixed. The equality of $f_b$ and $f_{b'}$ will then follow from the comparison principle in \Cref{SBME comparison principle on Rpp}. Consider $\smash{\phi\in C^\infty\big((0,\infty)\times \Rpp^{\D_K}\big)}$ with the property that $\smash{f_{b,b'}^{(K)}-\phi}$ achieves a local maximum at the point $\smash{(t^*,x^*)\in (0,\infty)\times \Rpp^{\D_K}}$. Since $\smash{\td f_b^{(K)}}$ is a viscosity subsolution to the Hamilton-Jacobi equation \eqref{eqn: SBME projected HJ eqn b},
$$\frac{a^2(b-b')}{2}+\partial_t \phi(t^*,x^*)-\td \H_{b,K}\big(a(b-b')\iota_K+\nabla \phi(t^*,x^*)\big)\leq 0$$
for the vector $\smash{\iota_K=(\abs{\D_K}^{-1})_{k\in\D_K}\in \Rp^{\D_K}}$. The fact that $\smash{\td f_b^{(K)}}$ has its gradient in $\smash{\td \K_{=a,b,K}'}$ together with \eqref{eqn: SBME projected solution Lipschitz}, \Cref{SBME gradient bound on test function} and \Cref{SBME gradient of test function in cone} implies that
$$a(b-b')\iota_K+\nabla \phi(t^*,x^*)\in \td \K_{=a,b,K}' \quad \text{and} \quad \dNorm{1}{a(b-b')\iota_K+\nabla \phi(t^*,x^*)}\leq \norm{\td \psi_b}_{\mathrm{Lip},\TV}.$$
It is therefore possible to find $\smash{u\in \Rp^{\D_K}}$ and $\smash{w\in \R^{\D_K}}$ with
$$a(b-b')\iota_K+\nabla \phi(t^*,x^*)=\td G_b^{(K)}u+w, \quad \Norm{u}_1= a \quad \text{and} \quad \dNorm{1}{w}\leq \frac{1}{2^{K/2}}.$$
Observe that
$$\dNorm{1}{\td G_b^{(K)}u}\leq \dNorm{1}{a(b-b')\iota_K+\nabla \phi(t^*,x^*)}+\dNorm{1}{w}\leq \norm{\td \psi_b}_{\mathrm{Lip},\TV}+\frac{1}{2^{K/2}},$$
so increasing $K$ if necessary, it is possible to ensure that $\smash{\td G_b^{(K)}u\in \td \CC_{b,K}\cap B_R}$. It follows by the Lipschitz continuity of $\td \H_{b,K}$ established in \Cref{SBME extending the non-linearity} that
$$\frac{a^2(b-b')}{2}+\partial_t\phi(t^*,x^*)-\td \C_{b,K}\big(\td G_b^{(K)}u\big)\leq \frac{8RM}{2^{K/2}m^2}.$$
Observe that
\begin{align*}
\td \C_{b,K}\big(\td G_b^{(K)} u\big)&=\frac{1}{2}\td G_b^{(K)}u\cdot u=\frac{1}{2}G^{(K)}u\cdot u+\frac{1}{2}b\Norm{u}_1^2=\frac{1}{2}\td G_{b'}^{(K)}u\cdot u+\frac{1}{2}(b-b')a^2\\
&=\td \C_{b',K}\big(\td G_{b'}^{(K)}u\big)+\frac{a^2(b-b')}{2}
\end{align*}
so in fact
$$\partial_t\phi(t^*,x^*)-\td \C_{b'}^{(K)}\big(\td G_{b'}^{(K)}u\big)\leq \frac{8RM}{2^{K/2}m}.$$
To replace $\td G_{b'}^{(K)}u$ by $\nabla\phi(t^*,x^*)$ observe that
\begin{align*}
\td G_{b'}^{(K)}u&=\td G^{(K)}_{b}u+(b'-b)\iota_K\Norm{u}_1=a(b-b')\iota_K+\nabla \phi(t^*,x^*)-w+a(b'-b)\iota_K\\
&=\nabla\phi(t^*,x^*)-w,
\end{align*}
and leverage the Lipschitz continuity of $\td \H_{b',K}$ established in \Cref{SBME extending the non-linearity} to deduce that
$$\big(\partial_t\phi-\td \H_{b',K}(\nabla \phi)\big)(t^*,x^*)\leq \EE_K$$
for the error term
$$\EE_K=\frac{16RM}{2^{K/2}m}.$$
In particular, the function $\smash{(t,x)\mapsto f_{b,b'}^{K}(t,x)-\EE_K t}$ is a viscosity subsolution to the Hamilton-Jacobi equation defining $\smash{\td f_{b'}^{K}}$. An identical argument shows that $\smash{(t,x)\mapsto f_{b,b'}^{K}(t,x)+\EE_K t}$ is a viscosity supersolution to the Hamilton-Jacobi equation defining $\smash{\td f_{b'}^{K}}$. It follows by the comparison principle in \Cref{SBME comparison principle on Rpp} that for every $\mu \in \M_+$,
$$\big\lvert f_b^{(K)}\big(t,x^{(K)}(\mu)\big)-f_{b'}^{(K)}\big(t,x^{(K)}(\mu)\big)\big\rvert=\big\lvert f_{b,b'}^{(K)}\big(t,x^{(K)}(\mu)\big)-\td f_{b'}^{(K)}\big(t,x^{(K)}(\mu)\big)\big\rvert\leq \EE_Kt.$$
Letting $K$ tend to infinity completes the proof.
\end{proof}

Combining this well-posedness result with the Hopf-Lax representation formula in \Cref{SBME infinite Hopf-Lax} we now prove the Hopf-Lax variational representation for $f=f_b$ stated in \Cref{SBME general infinite Hopf-Lax}.

\begin{proof}[Proof of Theorem~\ref{SBME general infinite Hopf-Lax}]
Fix $b>0$ large enough so the kernel $\td g_b$ is positive on $[-1,1]$ and satisfies \eqref{SBME H conv}. \Cref{SBME H1-H4 for tilted kernel} and the Hopf-Lax representation formula in \Cref{SBME infinite Hopf-Lax} imply that for any $t>0$ and $\mu\in \M_+$,
\begin{equation}\label{eqn: SMBE general Hopf Lax tilde}
\td f_b(t,\mu)=\sup_{\nu\in \M_+}\bigg\{\td \psi_b(\mu+t\nu)-\frac{t}{2}\int_{-1}^1 \td G_{b,\nu}(y)\ud \nu(y)\bigg\}.
\end{equation}
Since the Gateaux derivative density $x\mapsto D_\mu\psi(\mu+t\nu)$ belongs to the set $\CC_{a,\infty}$ by assumption, there exists a measure $\eta\in \M_{a,+}$ with $D_\mu\psi(\mu+t\nu,\cdot)=G_{\eta}$. This means that
$$D_\mu\td \psi_b(\mu+t\nu,x)=D_\mu\psi(\mu+t\nu,x)+ab=\int_{-1}^1 g(xy)\ud \eta(y)+ab=\int_{-1}^1 \td g_b(xy)\ud \eta(y),$$
so another application of \Cref{SBME infinite Hopf-Lax} implies that the supremum in \eqref{eqn: SMBE general Hopf Lax tilde} is achieved at some $\nu^*\in \M_+$ with
$$\td G_{b,\nu^*}=D_\mu\td \psi_b(\mu+t\nu^*,\cdot)=\td G_{b, \eta}.$$
Evaluating this equality at $x=0$ reveals that
$$\td g_b(0)\int_{-1}^1 \ud \nu^*(y)=\td g_b(0)\int_{-1}^1 \ud \eta(y)=\td g_b(0)a.$$
Since $\td g_b(0)>0$ by the choice of $b$, this means that $\nu^*\in \M_{a,+}$ and
$$\td f_b(t,\mu)=\sup_{\nu\in \M_{a,+}}\bigg\{\td \psi_b(\mu+t\nu)-\frac{t}{2}\int_{-1}^1 \td G_{b,\nu}(y)\ud \nu(y)\bigg\}.$$
It follows by \eqref{e.def.f.or.tdf} that
\begin{align*}
f_{b}(t,\mu)&=\sup_{\nu\in \M_{a,+}}\bigg\{\td \psi_b(\mu+t\nu)-\frac{t}{2}\int_{-1}^1 \td G_{b,\nu}(y)\ud \nu(y)-ab\int_{-1}^1 \ud \mu-\frac{a^2bt}{2}\bigg\}\\
&=\sup_{\nu\in \M_{a,+}}\bigg\{\psi(\mu+t\nu)+abt\int_{-1}^1 \ud \nu-\frac{t}{2}\int_{-1}^1 G_\nu(y)\ud \nu(y)-\frac{bt}{2}\int_{-1}^1\int_{-1}^1 \ud \nu\ud \nu -\frac{a^2bt}{2}\bigg\}\\
&=\sup_{\nu\in \M_{a,+}}\bigg\{\psi(\mu+t\nu)-\frac{t}{2}\int_{-1}^1 G_\nu(y)\ud \nu(y)\bigg\}.
\end{align*}
This completes the proof.
\end{proof}

\begin{appendix}

\section{Hamilton-Jacobi equations on positive half-spaces}\label{SBME app HJ eqn on half-space}

In this appendix we fix a non-linearity $\H:\R^d\to \R$ and an initial condition $\psi:\Rp^d\to \R$ with the properties that
\begin{enumerate}[label = \textbf{A\arabic*}]
\item the non-linearity $\H:\R^d\to \R$ is Lipschitz continuous with respect to the normalized-$\ell^{1,*}$ norm,
\begin{equation}
\abs{\H(y)-\H(y')}\leq \Norm{\H}_{\mathrm{Lip},1,*}\dNorm{1}{y-y'}
\end{equation}
for all $y,y'\in\R^d$;\label{eqn: SBME app A1}
\item the non-linearity is non-decreasing,
\begin{equation}
\H(y)\leq \H(y')
\end{equation}
for all $y,y'\in \R^d$ with $y\leq y'$;\label{eqn: SBME app A2}
\item the initial condition $\psi:\Rp^d\to \R$ is Lipschitz continuous with respect to the normalized-$\ell^1$ norm,
\begin{equation}
\abs{\psi(x)-\psi(x')}\leq \Norm{\psi}_{\mathrm{Lip},1}\Norm{x-x'}_1
\end{equation}
for all $x,x'\in \Rp^d$;\label{eqn: SBME app A3}
\end{enumerate}
and we establish the well-posedness of the Hamilton-Jacobi equations
\begin{equation}\label{eqn: SBME app HJ eqn on Rpp}
\partial_t f(t,x)=\H\big(\nabla f(t,x)\big) \quad \text{on} \quad \Rpp\times \Rpp^{d}
\end{equation}
and
\begin{equation}\label{eqn: SBME app HJ eqn on Rp}
\partial_t f(t,x)=\H\big(\nabla f(t,x)\big)  \quad \text{on}\quad \Rpp\times \Rp^{d}
\end{equation}
subject to the initial condition $\psi$. The appropriate notion of solution for these equations will be that of a viscosity solution as described in \Cref{SBME definition of viscosity solution}. To establish the well-posedness of these equations we will closely follow \cite{HB_cone}. We found it useful to provide full details when applying Perron's method, although our arguments will certainly be seen as classical by experts, and we hope that the reader will also find these details helpful. To be more specific, first, we will use the assumptions \eqref{eqn: SBME app A1} and \eqref{eqn: SBME app A3} to prove a comparison principle for the Hamilton-Jacobi equation~\eqref{eqn: SBME app HJ eqn on Rp} which will ensure the uniqueness of solutions to this equation. This will be the content of Section~\ref{SBME app comparison}. In Section~\ref{app SBME Perron}, we will combine \eqref{eqn: SBME app A1}-\eqref{eqn: SBME app A3} with the classical Perron method to obtain the existence of a solution to the Hamilton-Jacobi equation \eqref{eqn: SBME app HJ eqn on Rp}. Finally, in Section~\ref{SBME app equivalence}, we will leverage \eqref{eqn: SBME app A2} to show that solutions to \eqref{eqn: SBME app HJ eqn on Rpp} and  \eqref{eqn: SBME app HJ eqn on Rp} coincide. In the last section of this appendix we will show that the unique solution to the Hamilton-Jacobi equations \eqref{eqn: SBME app HJ eqn on Rpp} and \eqref{eqn: SBME app HJ eqn on Rp} preserves the monotonicity of its initial condition. It will be convenient to remember the definition of the function spaces \eqref{eqn: SBME L space} and \eqref{eqn: SBME L unif space}.

\subsection{Comparison principle and Lipschitz continuity of solutions on \texorpdfstring{$\Rp^d$}{Rpd}}\label{SBME app comparison}

In this section, we use the arguments in Proposition 3.2 of \cite{JC_NC} to obtain a comparison principle on $\smash{\mathfrak{L}_{\mathrm{unif}}}$ for the Hamilton-Jacobi equation \eqref{eqn: SBME app HJ eqn on Rp}. Together with the observation that any solution in $\smash{\mathfrak{L}}$ to the Hamilton-Jacobi equation \eqref{eqn: SBME app HJ eqn on Rp} is uniformly Lipschitz continuous with Lipschitz constant bounded by that of its initial condition, this comparison principle will imply the uniqueness of solutions in $\smash{\mathfrak{L}}$ to the Hamilton-Jacobi equation \eqref{eqn: SBME app HJ eqn on Rp}. The Lipschitz continuity of the solutions to the Hamilton-Jacobi equation \eqref{eqn: SBME app HJ eqn on Rp} will be obtained using the arguments in Proposition 3.4 of \cite{JC_NC}. For every $r \in \R$, we denote the positive part of $r$ by $r_+ = \max(r,0)$.

\begin{proposition}\label{SBME comparison principle on Rpd}
Fix a non-linearity $\H:\R^d\to \R$ satisfying \eqref{eqn: SBME app A1}, and let $\smash{u,v\in \mathfrak{L}_{\mathrm{unif}}}$ be a subsolution and a supersolution to \eqref{eqn: SBME app HJ eqn on Rp}. Write $\smash{L=\max\big(\sup_{t>0}\Norm{u(t,\cdot)}_{\mathrm{Lip},1},\sup_{t>0}\Norm{v(t,\cdot)}_{\mathrm{Lip},1}\big)}$ and $\smash{V=\Norm{\H}_{\mathrm{Lip},1,*}}$. For every $Q>2L$ and all $R\in \R$, the map
\begin{equation}
(t,x)\mapsto u(t,x)-v(t,x)-Q\big(\Norm{x}_1+Vt-R\big)_+
\end{equation}
achieves its supremum on $\smash{\{0\}\times \Rp^d}$.
\end{proposition}

\begin{proof}
Suppose for the sake of contradiction that there exists $T>0$ with
\begin{equation}\label{eqn: SBME comparison principle absurd}
\sup_{[0,T]\times \Rp^d}\big(u(t,x)-v(t,x)-\p(t,x)\big)>\sup_{\Rp^d}\big(u(0,x)-v(0,x)-\p(0,x)\big),
\end{equation}
where $\smash{\p(t,x)=Q\big(\Norm{x}_1+Vt-R\big)_+}$. The proof proceeds in three steps: first we smoothen and perturb \eqref{eqn: SBME comparison principle absurd}, then we use a variable doubling argument to obtain a system of inequalities, and finally we contradict this system of inequalities.\\
\noindent \step{1: smoothing and perturbing.}\\
Given $\epsilon_0\in (0,1)$ to be determined, let $\theta\in C^\infty(\R)$ be an increasing function with $$(r-\epsilon_0)_+\leq \theta(r)\leq r_+$$
for all $r\in \R$. Consider the smoothed normalized-$\ell^1$ norm,
$$\Norm{x}_{1,\epsilon_0}=\frac{1}{d}\sum_{k=1}^d\big(x_k^2+\epsilon_0\big)^{\frac{1}{2}},$$
and introduce the function
$$\Phi(t,x)=Q\theta\big(\Norm{x}_{1,\epsilon_0}+Vt-R\big)$$
defined on $\Rp\times \Rp^d$. The choice of $\theta$ and the bound $(a+b)_+\leq a_++b_+$ imply that
\begin{align*}
\p(t,x)
&\leq \Phi(t,x)+Q\epsilon_0\leq \p(t,x)+Q\epsilon_0^{1/2}+Q\epsilon_0,
\end{align*}
where we have used that $(a+b)^{\frac{1}{2}}\leq a^{\frac{1}{2}}+b^{\frac{1}{2}}$ for $a,b>0$.
It follows by \eqref{eqn: SBME comparison principle absurd} that
\begin{align*}
\sup_{\Rp^d}\big(u(0,x)-v(0,x)-\Phi(0,x)\big)
&< \sup_{[0,T]\times \Rp^d}\big(u(t,x)-v(t,x)-\Phi(t,x)\big)+Q\epsilon_0+Q\epsilon_0^{1/2},
\end{align*}
so choosing $\epsilon_0>0$ small enough guarantees that
\begin{equation}
\sup_{[0,T]\times \Rp^d}\big(u(t,x)-v(t,x)-\Phi(t,x)\big)>\sup_{\Rp^d}\big(u(0,x)-v(0,x)-\Phi(0,x)\big).
\end{equation}
This is a smoothed version of the absurd hypothesis \eqref{eqn: SBME comparison principle absurd}. For technical reasons, it will be convenient to perturb the function $\Phi$. Given $\epsilon>0$ to be determined, introduce the function
$$\chi(t,x)=\Phi(t,x)+\frac{\epsilon}{T-t},$$
defined on $\Rp\times \Rp^d$. Choosing $\epsilon>0$ small enough ensures that
\begin{equation}\label{eqn: SBME comparison principle absurd smoothed}
\sup_{[0,T]\times \Rp^d}\big(u(t,x)-v(t,x)-\chi(t,x)\big)>\sup_{\Rp^d}\big(u(0,x)-v(0,x)-\chi(0,x)\big).
\end{equation}
This is a smoothed and perturbed version of the absurd hypothesis \eqref{eqn: SBME comparison principle absurd}.\\
\step{2: system of inequalities.}\\
For each $\alpha\geq 1$, define the function $\Psi_\alpha:[0,T]\times \Rp^d\times [0,T] \times \Rp^d\to \R\cup\{-\infty\}$ by
\begin{equation}\label{eqn: SBME comparison principle Psi}
\Psi_\alpha(t,x,t',x')=u(t,x)-v(t',x')-\frac{\alpha}{2}\big(\abs{t-t'}^2+\Norm{x-x'}_{1,\epsilon_0}\big)-\chi(t,x).
\end{equation}
By doubling the variables and introducing the potential in this way, we ensure that whenever $\alpha>4(L+1)$, the function $\Psi_\alpha$ achieves its supremum at a point $(t_\alpha,x_\alpha,t_\alpha',x_\alpha')$ which remains bounded as $\alpha$ tends infinity. Indeed, if we write $C>0$ for a constant that depends on $T$, $Q$, $R$, $V$, $u(0,0)$, $[u]_0$, $v(0,0)$ and $[v]_0$ whose value might not be the same at each occurrence, then for any $x\in \Rp^d$ with $\Norm{x}_{1,\epsilon_0}>R+1$, the bound $\smash{\Phi(t,x)\geq Q(\Norm{x}_{1,\epsilon_0}+Vt-R-1)_+}$ reveals that
\begin{align*}
\Psi_\alpha(t,x,t',x')&\leq u(0,x)-v(0,x')-\frac{\alpha}{2}\Norm{x-x'}_{1,\epsilon_0}-\Phi(t,x)+C\\
&\leq L\big(\Norm{x}_1+\Norm{x'}_1\big)-\frac{\alpha}{2}\Norm{x-x'}_{1,\epsilon_0}-Q\Norm{x}_{1,\epsilon_0} +C\\
&\leq (2L-Q)\Norm{x}_1+\Big(L-\frac{\alpha}{2}\Big)\Norm{x-x'}_{1,\epsilon_0}+C\\
&\leq (2L- Q)\Norm{x}_1+C.
\end{align*}
Remembering that $Q>2L$ and observing that the supremum of \eqref{eqn: SBME comparison principle Psi} is bounded from below by $\smash{\Psi_\alpha(0,0,0,0)}$ ensures that this upper semi-continuous function achieves its supremum at some point $\smash{(t_\alpha,x_\alpha,t_\alpha',x_\alpha')}$ which remains bounded with respect to the normalized-$\ell^1$ norm as $\alpha$ tends to infinity. In particular, the potential
$$\alpha\big(\abs{t_\alpha-t_\alpha'}^2+\Norm{x_\alpha-x_\alpha'}_{1}\big)\leq \alpha\big(\abs{t_\alpha-t_\alpha'}^2+\Norm{x_\alpha-x_\alpha'}_{1,\epsilon_0}\big)$$
must remain bounded as $\alpha$ tends to infinity. It follows that, up to the extraction of a subsequence, there exist $t_0\in [0,T]$ and $x_0\in \Rp^d$ such that $t_\alpha\to t_0$, $t_\alpha'\to t_0$, $x_\alpha\to x_0$ and $x_\alpha'\to x_0$ as $\alpha\to \infty$. The term $\smash{\frac{\epsilon}{T-t}}$ in the definition of $\chi$ guarantees that $t_0\in [0,T)$. On the other hand, the semi-continuity of $u$, $v$ and $\chi$ together with the bounds
$$\big(u-v-\chi\big)(t,x)\leq \Psi_\alpha(t_\alpha,x_\alpha,t_\alpha',x_\alpha')\leq u(t_\alpha,x_\alpha)-v(t_\alpha',x_\alpha')-\chi(t_\alpha,x_\alpha)$$
imply that
$$(u-v-\chi)(t_0,x_0)=\sup_{[0,T]\times\Rp^d}(u-v-\chi).$$
By \eqref{eqn: SBME comparison principle absurd smoothed} it must be the case that $t_0\in (0,T)$. This means that $t_\alpha,t_\alpha'\in (0,T)$ for all $\alpha$ large enough. We have therefore found a sequence of quadruples $((t_\alpha,x_\alpha,t_\alpha',x_\alpha'))_\alpha$ with $t_\alpha,t_\alpha'\in (0,T)$ such that $\Psi_\alpha$ achieves its supremum at $(t_\alpha,x_\alpha,t_\alpha',x_\alpha')$ for $\alpha$ large enough. With this in mind, fix $\alpha\geq 1$ large enough, and introduce the smooth functions $\smash{\phi,\phi'\in C^\infty\big((0,T)\times \Rp^d\big)}$ defined by
\begin{align*}
\phi(t,x)&=v(t_\alpha',x_\alpha')+\frac{\alpha}{2}\big(\abs{t-t_\alpha'}^2+\Norm{x-x_\alpha'}_{1,\epsilon_0}\big)+\chi(t,x),\\
\phi'(t',x')&=u(t_\alpha,x_\alpha)-\frac{\alpha}{2}\big(\abs{t'-t_\alpha}^2+\Norm{x'-x_\alpha}_{1,\epsilon_0}\big)-\chi(t_\alpha,x_\alpha).
\end{align*}
Since $(t_\alpha,x_\alpha, t_\alpha', x_\alpha')$ maximizes $\Psi_\alpha$, the function $u-\phi$ achieves a local maximum at the point $\smash{(t_\alpha,x_\alpha)\in (0,\infty)\times \Rp^d}$ while $v-\phi'$ achieves a local minimum at $\smash{(t_\alpha', x_\alpha')\in (0,\infty)\times \Rp^d}$. It follows by definition of a viscosity solution that
\begin{equation}\label{eqn: SBME comparison principle phi and phi' conditions}
\big(\partial_t\phi-\H(\nabla \phi)\big)(t_\alpha,x_\alpha)\leq 0 \quad \text{and} \quad \big(\partial_t\phi' -\H(\nabla \phi')\big)(t_\alpha',x_\alpha')\geq 0.
\end{equation}
This is the system of inequalities that we now strive to contradict.\\
\noindent \step{3: reaching a contradiction.}\\
A direct computation gives
\begin{align*}
\big(\partial_t\phi-\H(\nabla \phi)\big)(t_\alpha,x_\alpha)&=\alpha(t_\alpha-t_\alpha')+\partial_t\Phi(t_\alpha,x_\alpha)+\frac{\epsilon}{(T-t)^2}-\H\big(\nabla \phi(t_\alpha,x_\alpha)\big),\\
\big(\partial_t\phi'-\H(\nabla \phi')\big)(t_\alpha',x_\alpha')&=\alpha(t_\alpha-t_\alpha')-\H\big(\nabla \phi'(t_\alpha',x_\alpha')\big),
\end{align*}
where
\begin{align*}
\partial_{x_k}\phi(t_\alpha,x_\alpha)&=\frac{\alpha}{2d}\cdot \frac{\big(x_\alpha-x_\alpha'\big)_k}{((x_\alpha-x_\alpha')_k^2+\epsilon_0)^{\frac{1}{2}}}+\partial_{x_k}\Phi(t_\alpha,x_\alpha),\\
\partial_{x_k}\phi'(t_\alpha',x_\alpha')&=\frac{\alpha}{2d}\cdot \frac{\big(x_\alpha-x_\alpha'\big)_k}{((x_\alpha-x_\alpha')_k^2+\epsilon_0)^{\frac{1}{2}}}.
\end{align*}
It follows by the definition of $\smash{V=\Norm{\H}_{\mathrm{Lip},1,*}}$ that 
\begin{align*}
\big(\partial_t\phi'-\H(\nabla \phi')\big)(t_\alpha',x_\alpha')&\leq \alpha(t_\alpha-t_\alpha')-\H\big(\nabla \phi(t_\alpha,x_\alpha)\big)+V\dNorm{1}{\nabla \phi(t_\alpha,x_\alpha))-\nabla \phi'(t_\alpha',x_\alpha')}\\
&\leq \alpha(t_\alpha-t_\alpha')-\H\big(\nabla \phi(t_\alpha,x_\alpha)\big)+Vd\max_{1\leq k\leq d}\lvert \partial_{x_k}\Phi(t_\alpha,x_\alpha)\rvert.
\end{align*}
A direct computation shows that $Vd\abs{\partial_{x_k}\Phi(t_\alpha,x_\alpha)}\leq \partial_t\Phi(t_\alpha,x_\alpha)$, so this can be bounded further by
\begin{align*}
\big(\partial_t\phi'-\H(\nabla \phi')\big)(t_\alpha',x_\alpha')&\leq \alpha(t_\alpha-t_\alpha')+\partial_t\Phi(t_\alpha,x_\alpha)-\H\big(\nabla \phi(t_\alpha,x_\alpha)\big)\\
&<\big(\partial_t\phi-\H(\nabla \phi)\big)(t_\alpha,x_\alpha)\leq 0,
\end{align*}
where the strict inequality is due to the term $\frac{\epsilon}{(T-t)^2}$ and the final inequality leverages the first inequality in \eqref{eqn: SBME comparison principle phi and phi' conditions}. This contradicts the second inequality in \eqref{eqn: SBME comparison principle phi and phi' conditions} and completes the proof.
\end{proof}

\begin{corollary}\label{SBME comparison principle corollary on Rpd}
Fix a non-linearity $\H:\R^d\to \R$ satisfying \eqref{eqn: SBME app A1}. If $\smash{u,v\in \mathfrak{L}_{\mathrm{unif}}}$ are respectively a subsolution and a supersolution to \eqref{eqn: SBME app HJ eqn on Rp}, then
\begin{equation}
\sup_{\Rp\times \Rp^d}\big(u(t,x)-v(t,x)\big)=\sup_{\Rp^d}\big(u(0,x)-v(0,x)\big).
\end{equation}
\end{corollary}

\begin{proof}
Suppose for the sake of contradiction that there exists $(t^*,x^*)\in (0,\infty)\times \Rp^d$ with $$u(t^*,x^*)-v(t^*,x^*)>\sup_{\Rp^d}\big(u(0,x)-v(0,x)\big).$$
Applying the comparison principle in \Cref{SBME comparison principle on Rpd} with $R=\Norm{x^*}_1+Vt^*$ yields a contradiction and completes the proof.
\end{proof}

\begin{proposition}\label{SBME solution to HJ eqn on Rp is Lipschitz}
Fix a non-linearity $\H:\R^d\to \R$ satisfying \eqref{eqn: SBME app A1}. If $\smash{f\in \mathfrak{L}}$ is a viscosity solution to the Hamilton-Jacobi equation \eqref{eqn: SBME app HJ eqn on Rp}, then 
\begin{equation}
\sup_{t\geq 0}\Norm{f(t,\cdot)}_{\mathrm{Lip},1}=\Norm{f(0,\cdot)}_{\mathrm{Lip},1}.
\end{equation}
\end{proposition}

\begin{proof}
Let $L=\Norm{f(0,\cdot)}_{\mathrm{Lip},1}$, and suppose for the sake of contradiction that there exists $T>0$ with
\begin{align}\label{eqn: SBME solution stays Lipschitz absurd}
\sup_{[0,T]\times \Rp^d\times\Rp^d}\big(f(t,x)-f(t,x')-L&\Norm{x-x'}_1\big) \notag\\
&>0\geq \sup_{x,x'\in \Rp^d}\big(f(0,x)-f(0,x')-L\Norm{x-x'}_1\big).
\end{align}
The proof proceeds in three steps: first we perturb \eqref{eqn: SBME solution stays Lipschitz absurd}, then we use a variable doubling argument to obtain a system of inequalities, and finally we contradict this system of inequalities.\\
\step{1: perturbing.}\\
Given $\epsilon_0\in (0,1)$ to be determined, let $\theta\in C^\infty(\R)$ be an increasing function with
$$(r-\epsilon_0)_+\leq \theta(r)\leq r_+$$
for all $r\in \R$, and consider the smoothed normalized-$\ell^1$ norm,
$$\Norm{x}_{1,\epsilon_0}=\frac{1}{d}\sum_{k=1}^d\big(x_k^2+\epsilon_0\big)^{\frac{1}{2}}.$$
For a constant $R\in \R$ to be chosen, $Q>2L$ and $V=\Norm{\H}_{\mathrm{Lip},1,*}$, introduce the function
$$\Phi(t,x)=Q\theta\big(\Norm{x}_{1,\epsilon_0}+Vt-R\big)$$
defined on $\Rp\times \Rp^d$. Given $\epsilon>0$ to be determined, consider the functions
$$u_\epsilon(t,x)=f(t,x)-\Phi(t,x)-\frac{\epsilon}{T-t}\quad \text{and} \quad v(t,x)=f(t,x)+\Phi(t,x)$$
defined on $[0,T]\times \Rp^d$. Remembering \eqref{eqn: SBME solution stays Lipschitz absurd} and choosing $R>0$ large enough and $\epsilon,\epsilon_0>0$ small enough guarantees that
\begin{align}\label{eqn: SBME solution stays Lipschitz absurd perturbed}
\sup_{[0,T]\times \Rp^d\times \Rp^d}\big(u_\epsilon(t,x)-v(t,x')-L&\Norm{x-x'}_{1,\epsilon_0}\big) \notag\\
&>\sup_{\Rp^d\times \Rp^d}\big(u_\epsilon(0,x)-v(0,x')-L\Norm{x-x'}_1\big).
\end{align}
This is a perturbed version of the absurd hypothesis \eqref{eqn: SBME solution stays Lipschitz absurd}. Before moving onto the variable doubling argument, observe that $u_\epsilon$ is a viscosity subsolution to the Hamilton-Jacobi equation \eqref{eqn: SBME app HJ eqn on Rp} while $v$ is a viscosity supersolution to this equation. Indeed, fix $\smash{\phi\in C^\infty\big((0,\infty)\times \Rp^d)}$ with the property that $u_\epsilon-\phi$ has a local maximum at $(t^*,x^*)\in (0,\infty)\times \Rp^d$. This means that the map
$$(t,x)\mapsto f(t,x)-\Big(\phi(t,x)+\Phi(t,x)+\frac{\epsilon}{T-t}\Big)$$
has a local maximum at $(t^*,x^*)\in (0,\infty)\times \Rp^d$. It follows by the viscosity subsolution criterion for $f$ that 
$$\partial_t\phi(t^*,x^*)+\partial_t\Phi(t^*,x^*)+\frac{\epsilon}{(T-t^*)^2}-\H\big(\nabla (\phi+\Phi)\big)(t^*,x^*)\leq 0.$$
A direct computation shows that $dV\abs{\partial_{x_k}\Phi(t^*,x^*)}\leq \partial_t\Phi(t^*,x^*)$, so the definition of $\smash{V=\Norm{\H}_{\mathrm{Lip},1,*}}$ implies that
$$\big(\partial_t\phi-\H(\nabla \phi)\big)(t^*,x^*)\leq \partial_t\phi(t^*,x^*)-\H\big(\nabla(\phi+\Phi)\big)(t^*,x^*)+V\dNorm{1}{\nabla\Phi(t^*,x^*)}\leq 0$$
which means that $u_\epsilon$ is a viscosity subsolution to the Hamilton-Jacobi equation \eqref{eqn: SBME app HJ eqn on Rp}. An identical argument reveals that $v$ is a viscosity supersolution to this equation.\\
\step{2: system of inequalities.}\\
Fix $\delta\in (0,1)$, and for each $\alpha\geq 1$ define the function $\Psi_\alpha:\Rp\times \Rp^d\times \Rp\times \Rp^d\to \R\cup\{-\infty\}$ by
\begin{equation}\label{eqn: SBME solution stays Lipschitz Psi}
\Psi_\alpha(t,x,t',x')=u_\epsilon(t,x)-v(t',x')-\frac{\alpha}{2}\abs{t-t'}^2-(L+\delta t)\Norm{x-x'}_{1,\epsilon_0}.
\end{equation}
By doubling the variables and introducing the potential in this way, we ensure that the function $\Psi_\alpha$ achieves its supremum at a point $(t_\alpha,x_\alpha,t_\alpha',x_\alpha')$ which remains bounded as $\alpha$ tends infinity. Indeed, if we write $C>0$ for a constant that depends on $T$, $Q$, $R$, $V$, $f(0,0)$ and $[f]_0$ whose value might not be the same at each occurrence, then for any $x\in \Rp^d$ with $\Norm{x}_{1,\epsilon_0}>R+1$,
$$\Psi_\alpha(t,x,t',x')\leq L\big(\Norm{x}_1+\Norm{x'}_1\big)-\Phi(t,x)-L\Norm{x-x'}_1+C\leq (2L-Q)\Norm{x}_1+C$$
where we have used the bound $\smash{\Phi(t,x)\geq Q(\Norm{x}_{1,\epsilon_0}+Vt-R-1)_+}$. An analogous bound can be obtained with $x$ replaced by $x'$. Remembering that $Q>2L$ and observing that the supremum of \eqref{eqn: SBME solution stays Lipschitz Psi} is bounded from below by $\smash{\Psi_\alpha(0,0,0,0)}$ ensures that this function achieves its supremum at some point $\smash{(t_\alpha,x_\alpha,t_\alpha',x_\alpha')}$ which remains bounded with respect to the normalized-$\ell^1$ norm as $\alpha$ tends to infinity. In particular, the term $\smash{\alpha \abs{t-t'}^2}$ must remain bounded as $\alpha$ tends to infinity. It follows that, up to the extraction of a subsequence, there exist $t_0\in [0,T]$ and $x_0,x_0'\in \Rp^d$ such that $t_\alpha\to t_0$, $t_\alpha'\to t_0$, $x_\alpha\to x_0$ and $x_\alpha'\to x_0'$ as $\alpha\to \infty$. The term $\smash{\frac{\epsilon}{T-t}}$ in the definition of $u_\epsilon$ guarantees that $t_0\in [0,T)$.  On the other hand, if $\smash{C_1>T(\Norm{x_\alpha}_{1,\epsilon_0}+\Norm{x_\alpha'}_{1,\epsilon_0})}$ for all $\alpha\geq 1$, then the semi-continuity of $u$ together with the bound
$$\Psi_\alpha(t_\alpha,x_\alpha,t_\alpha',x_\alpha')\geq -C_1\delta +\sup_{[0,T]\times \Rp^d\times \Rp^d}\big(u_\epsilon(t,x)-v(t,x')-L\Norm{x-x'}_{1,\epsilon_0}\big)$$
implies that
\begin{equation}\label{eqn: SBME solution stays Lipschitz Psi lower bound}
\Psi_\alpha(t_0,x_0,t_0,x_0')\geq -C_1\delta +\sup_{[0,T]\times \Rp^d\times \Rp^d}\big(u_\epsilon(t,x)-v(t,x')-L\Norm{x-x'}_{1,\epsilon_0}\big).
\end{equation}
To leverage this bound, observe that
\begin{align}\label{eqn: SBME solution stays Lipschitz Psi upper bound t} 
\Psi_\alpha(0,x_0,0,x_0')&=u_\epsilon(0,x_0)-v(0,x_0')-(L+\delta t)\Norm{x_0-x_0'}_{1,\epsilon_0} \notag\\
&\leq \sup_{\Rp^d\times \Rp^d}\big(u_\epsilon(0,x)-v(0,x')-L\Norm{x-x'}_{1}\big)
\end{align}
and
\begin{align*}
\Psi_\alpha(t_0,x_0,t_0,x_0)&=u_\epsilon(t_0,x_0)-v(t_0,x_0)\leq 
\sup_{[0,T]\times \Rp^d}\Big(-2\Phi(t,x)-\frac{\epsilon}{T-t}\Big).
\end{align*}
Since $\theta$ is non-decreasing this can be bounded further by
\begin{equation}\label{eqn: SBME solution stays Lipschitz Psi upper bound x} 
\Psi_\alpha(0,x_0,0,x_0')\leq \sup_{\Rp^d}\Big(-2\Phi(0,x)-\frac{\epsilon}{T}\Big)\leq \sup_{\Rp^d\times \Rp^d}\big(u_\epsilon(0,x)-v(0,x')-L\Norm{x-x'}_1\big),
\end{equation}
where we have used the fact that $\smash{-2\Phi(0,x)-\frac{\epsilon}{T}=u_\epsilon(0,x)-v(0,x)}$. Combining \eqref{eqn: SBME solution stays Lipschitz Psi lower bound}, \eqref{eqn: SBME solution stays Lipschitz Psi upper bound t} and \eqref{eqn: SBME solution stays Lipschitz Psi upper bound x} with the absurd assumption \eqref{eqn: SBME solution stays Lipschitz absurd perturbed} and choosing $\delta$ small enough shows that $t_0\in (0,T)$ and $x_0\neq x_0'$. This means that, taking a subsequence if necessary, it is possible to guarantee that $t_\alpha,t_\alpha'\in (0,T)$ and $x_\alpha\neq x_\alpha'$ for all $\alpha$ large enough. We have therefore found a sequence of quadruples $((t_\alpha,x_\alpha,t_\alpha',x_\alpha'))_\alpha$ with $t_\alpha,t_\alpha'\in (0,T)$ and $x_\alpha\neq x_\alpha'$ such that $\Psi_\alpha$ achieves its supremum at $(t_\alpha,x_\alpha,t_\alpha',x_\alpha')$ for $\alpha$ large enough. With this in mind, fix $\alpha\geq 1$ large enough, and introduce the smooth functions $\smash{\phi,\phi'\in C^\infty\big((0,T)\times \Rp^d\big)}$ defined by
\begin{align*}
\phi(t,x)&=v(t_\alpha',x_\alpha')+\frac{\alpha}{2}\abs{t-t_\alpha'}^2+(L+\delta t)\Norm{x-x_\alpha'}_{1,\epsilon_0},\\
\phi'(t',x')&=u_\epsilon(t_\alpha,x_\alpha)-\frac{\alpha}{2}\abs{t'-t_\alpha}^2-(L+\delta t_\alpha)\Norm{x_\alpha-x'}_{1,\epsilon_0}.
\end{align*}
Since $(t_\alpha,x_\alpha, t_\alpha', x_\alpha')$ maximizes $\Psi_\alpha$, the function $u_\epsilon-\phi$ achieves a local maximum at the point $\smash{(t_\alpha,x_\alpha)\in (0,\infty)\times \Rp^d}$ while $v-\phi'$ achieves a local minimum at $\smash{(t_\alpha', x_\alpha')\in (0,\infty)\times \Rp^d}$. It follows by the observation that $u_\epsilon$ is a viscosity subsolution while $v$ is a viscosity supersolution that
\begin{equation}\label{eqn: SBME solution stays Lipschitz phi and phi' conditions}
\big(\partial_t\phi-\H(\nabla \phi)\big)(t_\alpha,x_\alpha)\leq 0 \quad \text{and} \quad \big(\partial_t\phi' -\H(\nabla \phi')\big)(t_\alpha',x_\alpha')\geq 0.
\end{equation}
This is the system of inequalities that we now strive to contradict.\\
\step{3: reaching a contradiction.}\\
A direct computation gives
\begin{align*}
\big(\partial_t\phi-\H(\nabla \phi)\big)(t_\alpha,x_\alpha)&=\alpha(t_\alpha-t_\alpha')+\delta \Norm{x-x_\alpha'}_{1,\epsilon_0}-\H\big(\nabla \phi(t_\alpha,x_\alpha)\big),\\
\big(\partial_t\phi'-\H(\nabla \phi')\big)(t_\alpha',x_\alpha')&=\alpha(t_\alpha-t_\alpha')-\H\big(\nabla \phi'(t_\alpha',x_\alpha')\big),
\end{align*}
where
$$\partial_{x_k}\phi(t_\alpha,x_\alpha)=\frac{L+\delta t_\alpha}{d}\cdot \frac{\big(x_\alpha-x_\alpha'\big)_k}{((x_\alpha-x_\alpha')_k^2+\epsilon_0)^{\frac{1}{2}}}=\partial_{x_k}\phi'(t_\alpha',x_\alpha').$$
It follows by the second inequality in \eqref{eqn: SBME solution stays Lipschitz phi and phi' conditions} that
$$\big(\partial_t\phi -\H(\nabla \phi)\big)(t_\alpha,x_\alpha)=\big(\partial_t\phi' -\H(\nabla \phi')\big)(t_\alpha',x_\alpha')+\delta\Norm{x-x_\alpha'}_{1,\epsilon_0}>0.$$
This contradicts the first inequality in \eqref{eqn: SBME solution stays Lipschitz phi and phi' conditions} and completes the proof.
\end{proof}

\begin{corollary}\label{SBME uniqueness to HJ eqn on Rp}
Fix a non-linearity $\H:\R^d\to \R$ and an initial condition $\psi:\R^d\to \R$ satisfying \eqref{eqn: SBME app A1} and \eqref{eqn: SBME app A3}. If $f_1,f_2\in \mathfrak{L}$ are viscosity solutions to the Hamilton-Jacobi equation \eqref{eqn: SBME app HJ eqn on Rp} with initial condition $\psi$, then $f_1=f_2$.
\end{corollary}

\begin{proof}
This is an immediate consequence of \Cref{SBME solution to HJ eqn on Rp is Lipschitz} and \Cref{SBME comparison principle corollary on Rpd}.
\end{proof}

\subsection{Existence of solutions on \texorpdfstring{$\Rp^d$}{Rpd}}\label{app SBME Perron}

In this section, we fix a non-linearity $\H:\R^d \to \R$ and an initial condition $\psi:\Rp^d\to \R$ satisfying \eqref{eqn: SBME app A1}-\eqref{eqn: SBME app A3}, and we use the classical Perron method to establish the existence of solutions to the Hamilton-Jacobi equation \eqref{eqn: SBME app HJ eqn on Rp}. We closely follow the arguments in Chapter 5 of \cite{Bardi}. It will be convenient to fix a positive constant
\begin{equation}
K>\sup\big\{\abs{\H(y)}\mid y\in \R^d \text{ with } \dNorm{1}{y}\leq \Norm{\psi}_{\mathrm{Lip},1}\big\},
\end{equation}
and to define the continuous functions $\underline{u},\overline{u}:[0,\infty)\times \Rp^d\to \R$ by
\begin{equation}\label{eqn: SBME Perron method sub and super solution}
\underline{u}(t,x)=\psi(x)-Kt \quad \text{and} \quad \overline{u}(t,x)=\psi(x)+Kt.
\end{equation}
The importance of these functions stems from the fact that they are a viscosity subsolution and a viscosity supersolution to the Hamilton-Jacobi equation \eqref{eqn: SBME app HJ eqn on Rp}, respectively.

\begin{lemma}\label{eqn: SBME underline u and overline u solutions}
Fix a non-linearity $\H:\R^d \to \R$ and an initial condition $\psi:\Rp^d\to \R$ satisfying \eqref{eqn: SBME app A1}-\eqref{eqn: SBME app A3}. The functions $\underline{u}$ and $\overline{u}$ defined in \eqref{eqn: SBME Perron method sub and super solution} are a subsolution and a supersolution to the Hamilton-Jacobi equation \eqref{eqn: SBME app HJ eqn on Rp}, respectively.
\end{lemma}

\begin{proof}
Consider a smooth function $\smash{\phi\in C^\infty\big((0,\infty)\times \Rp^d\big)}$ with the property that $\underline{u}-\phi$ has a local maximum at a point $(t^*,x^*)\in (0,\infty)\times \Rp^d$. For any $x\in \Rp^d$ and every $\epsilon>0$,
$$\phi(t^*,x^*+\epsilon x)-\phi(t^*,x^*)\geq \underline{u}(t^*,x^*+\epsilon x)-\underline{u}(t^*,x^*)=\psi(x^*+\epsilon x)-\psi(x^*).$$
Dividing by $\epsilon$ and letting $\epsilon$ tend to zero shows that $\smash{\nabla \phi(t^*,x^*)\cdot x\geq \nabla \psi(x^*)}\cdot x$ for all $\smash{x\in \Rp^d}$. It follows that $\nabla \phi(t^*,x^*)\geq \nabla \psi(x^*)$, so \eqref{eqn: SBME app A2} and the fact that $t^*>0$ imply that
\begin{equation}\label{eqn: SBME underline u and overline u solutions key}
\big(\partial_t\phi-\H(\nabla \phi)\big)(t^*,x^*)\leq \partial_t\underline{u}(t^*,x^*)-\H\big(\nabla \psi(x^*)\big)=-K-\H\big(\nabla \psi(x^*)\big).
\end{equation}
To bound this further, observe that for every $x\in \Rpp^d$,  $x'\in \R^d$ and  $\epsilon>0$ small enough,
$$\psi(x+\epsilon x')-\psi(x)\leq \epsilon \Norm{\psi}_{\mathrm{Lip},1}\Norm{x'}_1.$$
Dividing by $\epsilon$ and letting $\epsilon$ tend to zero gives
$$\nabla \psi(x)\cdot x'\leq \Norm{\psi}_{\mathrm{Lip},1}\Norm{x'}_1.$$
Choosing $x'=d\sgn(\partial_{x_k}\psi(x))e_k$ shows that $\dNorm{1}{\nabla \psi(x)}\leq \Norm{\psi}_{\mathrm{Lip},1}$. It follows by continuity that $\dNorm{1}{\nabla \psi(x^*)}\leq \Norm{\psi}_{\mathrm{Lip},1}$ so \eqref{eqn: SBME underline u and overline u solutions key} and the definition of $K$ imply that $\underline{u}$ is a subsolution to the Hamilton-Jacobi equation \eqref{eqn: SBME app HJ eqn on Rp}. An identical argument shows that $\overline{u}$ is a supersolution to the Hamilton-Jacobi equation \eqref{eqn: SBME app HJ eqn on Rp}. This completes the proof.
\end{proof}

The main result of this section will be that the function $f:[0,\infty)\times \Rp^d\to \R$ defined by
\begin{equation}\label{eqn: SBME Perron method solution}
f(t,x)=\sup_{u\in \S}u(t,x)
\end{equation}
for the set
\begin{equation}
\S=\big\{u:[0,\infty)\times \Rp^d\to \R \mid \underline{u}\leq u\leq \overline{u} \text{ and } u^\star \text{ is a subsolution to } \eqref{eqn: SBME app HJ eqn on Rp}\big\}
\end{equation}
is a viscosity solution to the  Hamilton-Jacobi equation \eqref{eqn: SBME app HJ eqn on Rp}. We refer to \Cref{SBME app background} for the definitions and basic properties of lower and upper semi-continuous envelopes of a function $u$, which we denote by $u_\star$ and $u^\star$ respectively. The strategy will be to show that $\smash{f^\star}$ is a viscosity subsolution to the Hamilton-Jacobi equation \eqref{eqn: SBME app HJ eqn on Rp} while $\smash{f_\star}$ is a viscosity supersolution to this equation. The comparison principle in \Cref{SBME comparison principle corollary on Rpd} will then imply that $f$ is a viscosity solution to the Hamilton-Jacobi equation \eqref{eqn: SBME app HJ eqn on Rp}. Throughout this subsection, we will write
\begin{equation}
B_r(t^*,x^*)=\big\{(t,x)\in (0,\infty)\times \Rp^d\mid \abs{t-t^*}^2+\norm{x-x^*}_2^2\leq r^2\big\}.
\end{equation}
for the Euclidean ball of radius $r>0$ centered at the point $(t^*,x^*)\in \R\times \Rp^d$. It is readily verified that $f^\star$ is a subsolution.

\begin{lemma}\label{SBME Perron subsolution}
If $\H:\R^d\to \R$ satisfies \eqref{eqn: SBME app A1}, then the upper semi-continuous envelope $f^\star$ of the function \eqref{eqn: SBME Perron method solution} is a viscosity subsolution to the Hamilton-Jacobi equation \eqref{eqn: SBME app HJ eqn on Rp}. In particular, $f\in \S$.
\end{lemma}

\begin{proof}
Consider a smooth function $\smash{\phi\in C^\infty\big((0,\infty)\times \Rp^d\big)}$ with the property that $f^\star-\phi$ has a strict local maximum at the point $\smash{(t^*,x^*)\in (0,\infty)\times \Rp^d}$. To be more precise, suppose that
$$\big(f^\star-\phi\big)(t^*,x^*)>\big(f^\star-\phi\big)(t,x)$$
for all $(t,x)\in B_r(t^*,x^*)\setminus \{(t^*, x^*)\}$. By definition of the upper semi-continuous envelope and continuity of $\phi$, it is possible to find points $(t_n,x_n)\in B_r(t^*,x^*)$ converging to $(t^*,x^*)$ with
$$(f-\phi)(t_n,x_n)\geq (f^\star-\phi)(t^*,x^*)-\frac{1}{n}$$
for every integer $n\geq 1$. Similarly, by definition of $f$, it is possible find a sequence of functions $(u_n)\subset \S$ with
$$f(t_n,x_n)-\frac{1}{n}\leq u_n(t_n,x_n)$$
for all integer $n\geq 1$. If $(t_n',x_n')\in B_r(t^*,x^*)$ denotes the maximum of $u_n^\star-\phi$ on $B_r(t^*,x^*)$, then the fact that $u_n^\star$ is a subsolution to \eqref{eqn: SBME app HJ eqn on Rp} implies that
\begin{equation}\label{eqn: SBME Perron subsolution key}
\big(\partial_t\phi-\H(\nabla \phi)\big)(t_n',x_n')\leq 0.
\end{equation}
Notice that $u_n^\star-\phi$ achieves its maximum on the compact set $B_r(t^*,x^*)$ as it is an upper semi-continuous function by \Cref{SBME properties of semi-continuous envelope}. Remembering that $u_n^\star\geq u_n$ reveals that
$$(f^\star-\phi)(t_n',x_n')\geq (u_n^\star-\phi)(t_n',x_n')\geq (u_n-f)(t_n,x_n)+(f-\phi)(t_n,x_n)\geq (f^\star-\phi)(t^*,x^*)-\frac{2}{n},$$
where we have used that $u_n^\star \leq f^\star$ as $u_n\leq f$. If $(t'_\infty,x'_\infty)$ denotes any subsequential limit of $(t_n',x_n')$, then the upper semi-continuity of $f^\star$ established in \Cref{SBME properties of semi-continuous envelope} gives
$$(f^\star-\phi)(t'_\infty,x'_\infty)\geq (f^\star-\phi)(t^*,x^*).$$
Since $(t^*,x^*)$ is a strict local maximum of $f^\star-\phi$ on $B_r(t^*,x^*)$, this implies that $(t'_\infty,x'_\infty)=(t^*,x^*)$. Letting $n$ tend to infinity in \eqref{eqn: SBME Perron subsolution key} shows that $f^\star$ is viscosity subsolution to the Hamilton-Jacobi equation \eqref{eqn: SBME app HJ eqn on Rp}. It is clear by the definition in \eqref{eqn: SBME Perron method solution} that $\underline{u}\leq f\leq \overline{u}$ so $f\in\S$. This completes the proof.
\end{proof}

Showing that $f_\star$ is a viscosity supersolution requires more work, and relies upon the following modification of Lemma 2.12 in \cite{Bardi}. 

\begin{lemma}\label{SBME Perron bump lemma}
Fix a non-linearity $\H:\R^d\to \R$ satisfying \eqref{eqn: SBME app A1} and \eqref{eqn: SBME app A2}. If $u\in \S$ is such that $u_\star$ is not a viscosity supersolution to the Hamilton-Jacobi equation \eqref{eqn: SBME app HJ eqn on Rp}, then there exists $v\in \S$ with $v(t,x)>u(t,x)$ for some $(t,x)\in (0,\infty)\times \Rp^d$.
\end{lemma}

\begin{proof}
The assumption that $u_\star$ is not a viscosity supersolution to the Hamilton-Jacobi equation \eqref{eqn: SBME app HJ eqn on Rp} gives $\smash{\phi \in C^\infty\big((0,\infty)\times \Rp^d\big)}$ and  $(t^*,x^*)\in (0,\infty)\times \Rp^d$ with the property that $u_\star-\phi$ has a strict local minimum at $(t^*,x^*)$ but $\smash{(\partial_t \phi-\H(\nabla \phi))(t^*,x^*)<0}$.
To be more precise, we will assume that
\begin{equation}\label{eqn: SBME Perron method bump lemma assumption 1}
\big(u_\star-\phi)(t,x)>\big(u_\star-\phi)(t^*,x^*)
\end{equation}
for all $(t,x)\in B_r(t^*,x^*)\setminus\{(t^*,x^*)\}$ and
\begin{equation}\label{eqn: SBME Perron method bump lemma assumption 2}
\big(\partial_t\phi-\H(\nabla \phi)\big)(t,x)<-\epsilon
\end{equation}
for some $\epsilon>0$ and all $(t,x)\in B_r(t^*,x^*)$. Notice that we must have $u_\star(t^*,x^*)<\overline{u}(t^*,x^*)$. Indeed, if this were not the case, the assumption that $u\in \S$ would imply that $u_\star(t^*,x^*)=\overline{u}(t^*,x^*)$, and therefore
$$\big(\overline{u}-\phi\big)(t,x)\geq \big(u_\star -\phi\big)(t,x)> (u_\star-\phi)(t^*,x^*)=\big(\overline{u}-\phi\big)(t^*,x^*)$$
for all $(t,x)\in B_r(t^*,x^*)\setminus\{(t^*,x^*)\}$. In other words, the supersolution $\overline{u}$ would be such that $\overline{u}-\phi$ achieves a local maximum at $(t^*,x^*)$. This would contradict \eqref{eqn: SBME Perron method bump lemma assumption 2}. Decreasing $r>0$ if necessary and using the continuity of $\phi$ and $\overline{u}$, it is therefore possible to find $\delta>0$ with
\begin{equation}\label{eqn: SBME Perron method bump lemma delta}
u_\star(t^*,x^*)+\delta< \overline{u}(t^*,x^*)-\delta\leq \overline{u}(t,x) \quad \text{and} \quad \phi(t,x)\leq \phi(t^*,x^*)+\frac{\delta}{2}
\end{equation}
for all $(t,x)\in B_r(t^*,x^*)$. With this in mind, given $\epsilon'<\frac{1}{4}\min(r^2,\delta)$, introduce the function
$$w(t,x)=\phi(t,x)+\epsilon'-\norm{x-x^*}_2^2-\abs{t-t^*}^2+(u_\star-\phi)(t^*,x^*),$$
and define $v:[0,\infty)\times \Rp^d\to \R$ by
$$v(t,x)=\begin{cases}
\max\big(u(t,x),w(t,x)\big) & \text{ if } (t,x)\in B_r(t^*,x^*),\\
u(t,x) & \text{ if } (t,x)\notin B_r(t^*,x^*).
\end{cases}$$
It is clear from the assumption $u\in \S$ that $v\geq u\geq \underline{u}$. Moreover, for $(t,x)\in B_r(t^*,x^*)$,
$$w(t,x)\leq \phi(t^*,x^*)+\frac{\delta}{2}+\frac{\delta}{2}+(u_\star-\phi)(t^*,x^*)=u_\star(t^*,x^*)+\delta\leq \overline{u}(t,x),$$ 
where we have used \eqref{eqn: SBME Perron method bump lemma delta} and the fact that $\epsilon'\leq \delta/2$. Together with the assumption $u\in \S$, this shows that $v\leq \overline{u}$, and therefore $\underline{u}\leq v\leq \overline{u}$. Furthermore, the definition of the lower semi-continuous envelope gives points $(t_n,x_n)\in B_r(t^*,x^*)$ with $(t_n,x_n)\to (t^*,x^*)$ and $u(t_n,x_n)\to u_\star(t^*,x^*)$. Since $v\geq w$ on $B_r(t^*,x^*)$, it follows that
$$\liminf_{n\to \infty}v(t_n,x_n)\geq \liminf_{n\to\infty}w(t_n,x_n)=\phi(t^*,x^*)+\epsilon'+(u_\star-\phi)(t^*,x^*)=u_\star(t^*,x^*)+\epsilon'.$$
This means that for any $n$ large enough
$$v(t_n,x_n)\geq u(t_n,x_n)+\frac{\epsilon'}{2}>u(t_n,x_n),$$
so there exists a point $(t,x)\in (0,\infty)\times \Rp^d$ with $v(t,x)>u(t,x)$. All that remains is to verify that $v^\star$ is a subsolution to the Hamilton-Jacobi equation \eqref{eqn: SBME app HJ eqn on Rp}. 
Consider $\smash{\beta\in C^\infty\big((0,\infty)\times \Rp^d\big)}$ and $(t_0,x_0)\in (0,\infty)\times \Rp^d$ with the property that $v^\star-\beta$ has a strict local maximum on $B_{r'}(t_0,x_0)$ at $(t_0,x_0)$. The definition of the upper semi-continuous envelope gives points $(t_n,x_n)\in B_{r'}(t_0,x_0)$ converging to $(t_0,x_0)$ with
\begin{equation}\label{eqn: Perron method bump lemma (t_n,x_n)}
v(t_n,x_n)\geq v^\star(t_0,x_0)-\frac{1}{n}.
\end{equation}
Since $v(t_n,x_n)$ is either equal to $w(t_n,x_n)$ or $u(t_n,x_n)$, passing to a subsequence, it is possible to assume that $v(t_n,x_n)=w(t_n,x_n)$ for all $n\geq 1$ or that $v(t_n,x_n)=u(t_n,x_n)$ for all $n\geq 1$. We treat these two cases separately.\\
\case{1: $v(t_n,x_n)=w(t_n,x_n)$ for all $n\geq 1$}\\
In this case, we must have $(t_n,x_n)\in B_{r}(t^*,x^*)$ for all $n\geq 1$. Indeed, for $(t,x)\notin B_{r/2}(t^*,x^*)$,
$$(w-u)(t,x)\leq (w-u_\star)(t,x)\leq (\phi-u_\star)(t,x)+\epsilon'-\frac{r^2}{4}+(u_\star-\phi)(t^*,x^*)<\epsilon'-\frac{r^2}{4}\leq 0 $$
where we have used \eqref{eqn: SBME Perron method bump lemma assumption 1} and the fact that $\epsilon'<r^2/4$. If $(t_n',x_n')$ denotes the maximum of $w-\beta$ on $B_r(t^*,x^*)\cap B_{r'}(t_0,x_0)$, arguing as in the proof of \Cref{eqn: SBME underline u and overline u solutions} shows that
\begin{align*}
\partial_t\beta(t_n',x_n')&=\partial_tw(t_n',x_n')=\partial_t\phi(t_n',x_n')-2(t_n'-t^*),\\
\nabla \beta(t_n',x_n')&\geq \nabla w(t_n',x_n')=\nabla \phi(t_n',x_n')-2(x_n'-x^*).
\end{align*}
It follows by \eqref{eqn: SBME app A2}, \eqref{eqn: SBME app A1} and \eqref{eqn: SBME Perron method bump lemma assumption 2} that
\begin{align*}
\big(\partial_t\beta-\H(\nabla \beta)\big)(t_n',x_n')&\leq \partial_t\phi(t_n',x_n')-\H\big(\nabla \phi(t_n',x_n')\big) +2\abs{t_n'-t^*}+2\Norm{\H}_{\mathrm{Lip},1,*}\dNorm{1}{x_n'-x^*}\\
&\leq -\epsilon +2\abs{t_n'-t^*}+2d\Norm{\H}_{\mathrm{Lip},1,*}\norm{x_n'-x^*}_2.
\end{align*}
Decreasing $r$ if necessary, it is therefore possible to ensure that $\smash{\big(\partial_t\beta-\H(\nabla \beta)\big)(t_n',x_n')\leq 0}$. To leverage this bound, observe that by \eqref{eqn: Perron method bump lemma (t_n,x_n)}, the continuity of $\beta$ and the fact that $(t_n,x_n)$ converges to $(t_0,x_0)$,
$$(v^\star-\beta)(t_n',x_n')\geq (w-\beta)(t_n',x_n')\geq (w-\beta)(t_n,x_n)=(v-\beta)(t_n,x_n)\geq (v^\star-\beta)(t_0,x_0)-\frac{2}{n}.$$
In particular, any subsequential limit $(t_\infty',x_\infty')$ of $(t_n',x_n')$ must satisfy
$$(v^\star-\beta)(t_\infty',x_\infty')\geq (v^\star-\beta)(t_0,x_0) \quad \text{and} \quad \big(\partial_t\beta-\H(\nabla \beta)\big)(t_\infty',x_\infty')\leq 0.$$
Since $(t_0,x_0)$ is a strict local maximum of $v^\star -\beta$ on $B_{r'}(t_0,x_0)$, the first of these inequalities shows that $(t_\infty',x_\infty')=(t_0,x_0)$ while the second implies the required subsolution criterion.\\
\case{2: $v(t_n,x_n)=u(t_n,x_n)$ for all $n\geq 1$}\\
In this case, let $(t_n',x_n')$ denote the maximum of $u^\star-\beta$ on $B_{r'}(t_0,x_0)$. Since $u^\star$ is a viscosity subsolution to the Hamilton-Jacobi equation \eqref{eqn: SBME app HJ eqn on Rp},
$$\big(\partial_t \beta-\H(\nabla \beta)\big)(t_n',x_n')\leq 0.$$
On the other hand, the inequality $v\geq u$ and \eqref{eqn: Perron method bump lemma (t_n,x_n)} reveal that
\begin{align*}
(v^\star-\beta)(t_n',x_n')&\geq (u^\star-\beta)(t_n',x_n')\geq (u^\star-\beta)(t_n,x_n)\geq (u-\beta)(t_n,x_n)=(v-\beta)(t_n,x_n)\\
&\geq  (v^\star-\beta)(t_0,x_0)-\frac{1}{n},
\end{align*}
so any subsequential limit $(t_\infty',x_\infty')$ of $(t_n',x_n')$ must satisfy
$$(v^\star-\beta)(t_\infty',x_\infty')\geq (v^\star-\beta)(t_0,x_0) \quad \text{and} \quad \big(\partial_t\beta-\H(\nabla \beta)\big)(t_\infty',x_\infty')\leq 0.$$
Since $(t_0,x_0)$ is a strict local maximum of $v^\star -\beta$ on $B_{r'}(t_0,x_0)$, the first of these inequalities shows that $(t_\infty',x_\infty')=(t_0,x_0)$ while the second implies the required subsolution criterion. This completes the proof.
\end{proof}

\begin{corollary}\label{SBME Perron supersolution}
If $\H:\R^d\to \R$ satisfies \eqref{eqn: SBME app A1} and \eqref{eqn: SBME app A2}, then the lower semi-continuous envelope $f_\star$ of the function \eqref{eqn: SBME Perron method solution} is a viscosity supersolution to the Hamilton-Jacobi equation \eqref{eqn: SBME app HJ eqn on Rp}.
\end{corollary}

\begin{proof}
Suppose for the sake of contradiction that $f_\star$ is not a supersolution to the Hamilton-Jacobi equation \eqref{eqn: SBME app HJ eqn on Rp}. Combining \Cref{SBME Perron subsolution} and \Cref{SBME Perron bump lemma} gives a function $v\in \S$ and a point $(t,x)\in (0,\infty)\times \Rp^d$ with $v(t,x)>f(t,x)$. The contradiction
$$f(t,x)=\sup_{u\in \S}u(t,x)\geq v(t,x)>f(t,x)$$
completes the proof.
\end{proof}

Together with \Cref{SBME Perron subsolution} and the comparison principle in \Cref{SBME comparison principle corollary on Rpd}, this result allows us to establish the well-posedness of the Hamilton-Jacobi equation \eqref{eqn: SBME app HJ eqn on Rp}.

\begin{proposition}\label{SBME WP of HJ eqn on Rp}
If $\H:\R^d\to \R$ and $\psi:\Rp^d\to \R$ satisfy \eqref{eqn: SBME app A1}-\eqref{eqn: SBME app A3}, then the Hamilton-Jacobi equation \eqref{eqn: SBME app HJ eqn on Rp} admits a unique viscosity solution $f\in \mathfrak{L}$ subject to the initial condition $\psi$. Moreover,
\begin{equation}
\sup_{t>0}\Norm{f(t,\cdot)}_{\mathrm{Lip},1}=\Norm{\psi}_{\mathrm{Lip},1}.
\end{equation}
\end{proposition}

\begin{proof}
Denote by $f\in \mathfrak{L}$ the function defined in \eqref{eqn: SBME Perron method solution}. Combining \Cref{SBME Perron subsolution} and \Cref{SBME Perron supersolution} shows that $f^\star$ is a viscosity subsolution to the Hamilton-Jacobi equation \eqref{eqn: SBME app HJ eqn on Rp} while $f_\star$ is a viscosity supersolution to this equation. By \Cref{SBME properties of semi-continuous envelope} and continuity of $\underline{u}$ and $\overline{u}$, it is clear that $\underline{u}\leq f_\star\leq f\leq f^\star\leq \overline{u}$. Moreover, any function $h:[0,\infty)\times \Rp^d\to \R$ with $\underline{u}\leq h\leq \overline{u}$ satisfies the bounds
\begin{align*}
\psi(x)&=\underline{u}(0,x)\leq h(0,x)\leq \overline{u}(0,x)=\psi(x),\\
h(t,x)-h(0,x)&\leq \overline{u}(t,x)-\psi(x)=Kt,\\
h(0,x)-h(t,x)&\leq \psi(x)-\underline{u}(t,x)\leq Kt.
\end{align*}
for all $t>0$ and every $x\in \Rp^d$. In particular, we have $h\in \mathfrak{L}$, and therefore $f_\star,f, f^\star\in \mathfrak{L}$. It follows by the comparison principle in \Cref{SBME comparison principle corollary on Rpd} that $f^\star\leq f_\star$. Since $f_\star\leq f\leq f^\star$ by \Cref{SBME properties of semi-continuous envelope}, we must have $f=f_\star=f^\star$. In particular, the function $f\in \mathfrak{L}$ is a continuous viscosity solution to the Hamilton-Jacobi equation \eqref{eqn: SBME app HJ eqn on Rp}. The uniqueness of such a viscosity solution is guaranteed by \Cref{SBME uniqueness to HJ eqn on Rp}. Recalling \Cref{SBME solution to HJ eqn on Rp is Lipschitz} gives the Lipschitz bound and completes the proof. 
\end{proof}

\subsection{Equivalence of solutions on \texorpdfstring{$\Rp^d$}{Rpd} and \texorpdfstring{$\Rpp^d$}{Rppd}}\label{SBME app equivalence}

In this section, we leverage the monotonicity assumption \eqref{eqn: SBME app A2} of the non-linearity to show that viscosity solutions to the Hamilton-Jacobi equations \eqref{eqn: SBME app HJ eqn on Rpp} and \eqref{eqn: SBME app HJ eqn on Rp} coincide. Combining this with \Cref{SBME WP of HJ eqn on Rp} gives a well-posedness theory for the Hamilton-Jacobi equation \eqref{eqn: SBME app HJ eqn on Rpp}.

To ignore the boundary of the upper half-plane, we proceed as in Proposition 2.1 of \cite{HB_cone} which is inspired by \cite{Crandall, Souganidis}. The main difference between \cite{HB_cone} and \cite{Crandall, Souganidis} is in the definition of a distance-like function to the boundary of the domain on which the Hamilton-Jacobi equation is defined. Since this distance-like function will reappear in the next section when we show that the solution to the Hamilton-Jacobi equation \eqref{eqn: SBME app HJ eqn on Rp} preserves the monotonicity of its initial condition, we will define it for a general closed convex cone $\K\subset \R^d$. Given a closed convex cone $\K\subset \R^d$, denote by
\begin{equation}
\K^*=\big\{x\in \R^d\mid x\cdot y\geq 0 \text{ for all } y\in \K\big\}
\end{equation}
its dual cone, and define the distance-like function 
$\udo: \K^*\to \Rp$ by
\begin{equation}\label{eqn: SBME distance-like function}
\udo(x)=\inf_{\substack{\dNorm{1}{y}=1\\y\in \K}}y\cdot x.
\end{equation}
The notion of a dual cone is reviewed in \Cref{SBME app background}. Before stating the main properties of this distance-like function, recall that the super-differential of a function $h:\K^*\to \R$ at a point $x\in \mathrm{int}(\K^*)$ is the set
\begin{equation}
\partial h(x)=\big\{p\in \R^d\mid h(x')\leq h(x)+p\cdot (x'-x) +o(x'-x) \text{ as } x'\to x \text{ in } \K^*\big\}.
\end{equation}

\begin{lemma}\label{SBME properties of distance-like function}
The function $\udo:\K^*\to \Rp$ defined in \eqref{eqn: SBME distance-like function} satisfies the following basic properties.
\begin{enumerate}
    \item The infimum defining $\udo(x)$ is achieved for every $x\in \K^*$.
    \item $\udo(x)=0$ if and only if $x\in \partial \K^*$.
    \item $\udo$ is Lipschitz continuous with respect to the normalized-$\ell^1$ norm. Moreover, it has Lipschitz constant at most one.
    \item $\udo$ is concave and $\K^*$-non-decreasing.
    \item If $x\in \mathrm{int}(\K^*)$, then $\partial \udo(x)\subset \K$ and $\dNorm{1}{p}\leq 1$ for any $p\in \partial \udo(x)$.
    \item If $h:\R^d\to \R$ is a differentiable function and $x\mapsto h(x)-\frac{1}{\udo(x)}$ achieves a local maximum at a point $x_0\in \mathrm{int}(\K^*)$, then $\smash{-\udo(x_0)^2\nabla h(x_0)\in \partial \udo (x_0)}$.
\end{enumerate}
\end{lemma}

\begin{proof}
We treat each property separately.
\begin{enumerate}
\item Consider a sequence $(y_n)\subset \K$ with $\dNorm{1}{y_n}=1$ and $y_n\cdot x\to \udo(x)$. Since $(y_n)$ is uniformly bounded, it admits a subsequential limit $\smash{y\in \K}$ with $\dNorm{1}{y}=1$ and $y\cdot x=\udo (x)$. We have used the equivalence and continuity of norms as well as the fact that $\K$ is closed. This shows that the infimum in the definition of $\udo(x)$ is attained.

\item If $\udo(x)=0$, then there exists $y\in \K$ with $\dNorm{1}{y}=1$ and $y\cdot x=0$. This shows that $x\in \partial \K^*$. On the other hand, if $\smash{x\in \partial \K^*}$, then there exists a non-zero $\smash{z\in \K}$ with $z\cdot x=0$. Taking $y=z/\dNorm{1}{z}$ gives $\smash{y\in \K}$ with $\dNorm{1}{y}=1$ and $\udo (x)\leq y\cdot x=0$. This shows that $\udo(x)=0$.

\item Fix $x,y\in \K^*$, and let $\smash{z\in \K}$ with $\dNorm{1}{z}=1$ be such that $\udo(y)=z\cdot y$. By the Cauchy-Schwarz inequality,
$$\udo(x)-\udo(y)\leq z\cdot x-z\cdot y=z\cdot(y-x)\leq \dNorm{1}{z}\Norm{y-x}_1=\Norm{y-x}_1.$$
Reversing the roles of $x$ and $y$ shows that $\udo$ is Lipschitz continuous with Lipschitz constant at most one.
\item Fix $x,y\in \K^*$ as well as $t\in [0,1]$, and let $z\in \K$ achieve the infimum for $\udo(tx+(1-t)y)$. It is clear that 
$$\udo(tx+(1-t)y)=z\cdot \big(tx+(1-t)y\big)=t (z\cdot x)+(1-t)(z\cdot y)\geq t\udo(x)+(1-t)\udo(y).$$
This shows that $\udo$ is concave. To see that $\udo$ is $\smash{\K^*}$-non-decreasing, fix $\smash{x,x'\in \K^*}$ with $\smash{x'-x\in \K^*}$, and let $y\in \K$ attain the infimum defining $\udo(x')$. Since $\smash{x'-x\in \K^*}$,
$$\udo(x')-\udo(x)\geq y\cdot x'-y\cdot x= (x'-x)\cdot y\geq 0$$
as required.

\item Fix $\smash{x\in \mathrm{int}(\K^*)}$, $\smash{z\in \K^*}$ and $p\in \partial \udo(x)$. Notice that $\partial\udo (x)\neq \emptyset$ as $\udo$ is concave. Since $\smash{\epsilon z\in \K^*}$ for every $\epsilon>0$, the $\K^*$-non-decreasingness of $\udo$ and the definition of the super-differential imply that
$$0\leq \udo(x+\epsilon z)-\udo(x)\leq p\cdot \epsilon z+o(\epsilon z).$$
Dividing by $\epsilon$ and letting $\epsilon$ tend to zero shows that $p\cdot z\geq 0$ for all $\smash{z\in \K^*}$. It follows by \Cref{SBME closed convex cone bidual} that $\smash{p\in \K^{**}=\K}$. To see that $\smash{\dNorm{1}{p}\leq 1}$ for $p\in \partial \udo (x)$, find $\epsilon>0$ small enough so that $\smash{x+\epsilon y\in \K^*}$ for all $y\in \R^d$ with $\Norm{y}_1=1$. Fix $p\in \partial \udo (x)$ and $y\in \R^d$ with $\Norm{y}_1=1$. Since $\smash{x-\epsilon y\in \K^*}$, the definition of the super-differential implies that
$$\udo (x-\epsilon y)\leq \udo(x)-\epsilon p\cdot y+o(\epsilon y).$$
Rearranging and using the $1$-Lipschitz continuity of $\udo$ reveals that
$$\epsilon p\cdot y\leq \epsilon \Norm{y}_1+o(\epsilon y)=\epsilon+o(\epsilon y).$$
Dividing by $\epsilon$ and letting $\epsilon$ tend to zero shows that $p\cdot y\leq 1$ for every $y\in \R^d$ with $\Norm{y}_1=1$. Choosing $y=d\sgn(p_k)e_k$ gives $\smash{\dNorm{1}{p}\leq 1}$.

\item Fix $z\in \K^*$. Since $x_0\in \mathrm{int}(\K^*)$ is a local maximum of the map $\smash{x\mapsto h(x)-\frac{1}{\udo(x)}}$, for every $\epsilon>0$ small enough,
$$h(x_0)-\frac{1}{\udo(x_0)}\geq h(x_0+\epsilon z)-\frac{1}{\udo(x_0+\epsilon z)}.$$
Rearranging and using the $1$-Lipschitz continuity of $\udo$ as well as the differentiability of $h$ reveals that
\begin{align*}
\udo(x_0+\epsilon z)&\leq \udo(x_0)-\udo(x_0)\udo(x_0+\epsilon z)\big(h(x_0+\epsilon z)-h(x_0)\big)\\
&= \udo(x_0)-\udo(x_0)^2\nabla h(x_0)\cdot \epsilon z+o(\epsilon z).
\end{align*}
This shows that $-\udo(x_0)^2\nabla h(x_0)\in \partial \udo (x_0)$ and completes the proof.\qedhere
\end{enumerate}
\end{proof}

\begin{proposition}\label{SBME monotonic non-linearities no boundary condition}
If $\H:\R^d\to \R$ is a continuous non-linearity satisfying \eqref{eqn: SBME app A2}, then a continuous function $\smash{u:[0,\infty)\times \Rp^d\to \R}$ is a viscosity subsolution to the Hamilton-Jacobi equation \eqref{eqn: SBME app HJ eqn on Rpp} if and only if it is a viscosity subsolution to the Hamilton-Jacobi equation  \eqref{eqn: SBME app HJ eqn on Rp}. An identical statement holds for viscosity supersolutions.
\end{proposition}

\begin{proof}
The argument for viscosity subsolutions and viscosity supersolutions being almost identical, we focus exclusively on the case of viscosity subsolutions. To begin with, suppose that $u$ is a viscosity subsolution to the Hamilton-Jacobi equation \eqref{eqn: SBME app HJ eqn on Rp}, and let $\smash{\phi\in C^\infty\big((0,\infty)\times \Rpp^d\big)}$ be a function with the property that $u-\phi$ has a local maximum at $\smash{(t^*,x^*)\in (0,\infty)\times \Rpp^d}$. After modifying $\phi$ outside of a neighborhood of $(t^*, x^*)$ so that it becomes a smooth function defined on the larger domain $(0,\infty) \times \Rp^d$, we can apply the subsolution criterion for \eqref{eqn: SBME app HJ eqn on Rp} and obtain the result.

Conversely, suppose that $u$ is a continuous viscosity subsolution to the Hamilton-Jacobi equation \eqref{eqn: SBME app HJ eqn on Rpp}, and consider a smooth function $\smash{\phi\in C^\infty\big((0,\infty)\times \Rp^d\big)}$ with the property that $u-\phi$ has a local maximum at $\smash{(t^*,x^*)\in (0,\infty)\times \Rp^d}$. If $\smash{x^*\in  \Rpp^d}$ there is nothing to prove, so assume that $\smash{x^*\in \partial \Rp^d}$. Perturbing the test function $\phi$ by a small quadratic function if necessary, suppose further that $(t^*,x^*)$ is a strict local maximum of $u-\phi$. To be more precise, assume that
\begin{equation}\label{eqn: SBME ignoring the boundary strict maximum}
u(t,x)-\phi(t,x)<u(t^*,x^*)-\phi(t^*,x^*)
\end{equation}
for any $(t,x)$ other than $(t^*,x^*)$ in the closure of the open neighborhood
$$\BigO_r=(t^*-r, t^*+r)\times \big(\mathrm{int}\big(B_r(x^*)\big)\cap \Rpp^d\big).$$
Decreasing $r>0$ if necessary, it is possible to ensure that $(t^*-r, t^*+r)\subset (0,\infty)$. To establish the subsolution criterion for \eqref{eqn: SBME app HJ eqn on Rp}, we proceed in two steps: first we show that there exists an almost maximizer of $u-\phi$ in $\BigO_r$, and then we use a variable doubling argument to conclude.\\
\step{1: almost maximizer in $\BigO_r$.}\\
Introduce the distance-like function $\udo:\Rp^d\to \R$ defined by
$$\udo(x)=\inf_{\substack{\dNorm{1}{y}=1\\y\in \Rp^d}}y\cdot x.$$
This corresponds to the function \eqref{eqn: SBME distance-like function} for the cone $\K=\Rp^d=(\Rp^d)^*$. For each $\epsilon>0$, define the function $\smash{\psi_\epsilon:(0,\infty)\times \Rp^d\to \R\cup \{-\infty\}}$ by
$$\psi_\epsilon(s,y)=u(s,y)-\phi(s,y)-\frac{\epsilon}{\udo(y)}.$$
Since $\psi_\epsilon$ is upper semi-continuous with values in $\R\cup \{-\infty\}$, there exists $\smash{(s_\epsilon,y_\epsilon)\in \overline{\BigO}_r}$ with
$$\psi_\epsilon(s_\epsilon,y_\epsilon)=\sup_{(s,y)\in \overline{\BigO}_r}\psi_\epsilon(s,y).$$
Decreasing $r$ if necessary and combining the continuity of $u-\phi$ with \eqref{eqn: SBME ignoring the boundary strict maximum} gives $(t,x)\in \BigO_r$ such that
$$u(s,y)-\phi(s,y)<u(t,x)-\phi(t,x)$$
for all $\smash{(s,y)\in \overline{\BigO}_r}$ with $\smash{s\in \{t^*-r,t^*+r\}}$ or $\smash{y\in \partial B_r(x^*)\cap \Rp^d}$. This means that for every $\epsilon>0$ small enough, we must have $\psi_\epsilon(s,y)<\psi_\epsilon(t,x)$ for all $\smash{(s,y)\in \overline{\BigO}_r}$ with $\smash{s\in \{t^*-r,t^*+r\}}$ or $\smash{y\in \partial B_r(x^*)\cap \Rp^d}$. Together with the term $\epsilon/\udo(y)$ in $\psi_\epsilon$, this ensures that 
$(s_\epsilon,y_\epsilon)\in \BigO_r$.

\noindent \step{2: doubling the variables.}\\
Fix $\epsilon,\delta>0$, and a smooth function $\zeta_\epsilon:\R\times \R^d\to [0,1]$ with
$$\supp \zeta_\epsilon \subset \BigO_r \quad \text{and} \quad \zeta_\epsilon(s_\epsilon,y_\epsilon)=1.$$
Given $0<\epsilon_0<1$ to be determined, consider the smoothed normalized-$\ell^1$ norm,
$$\Norm{x}_{1,\epsilon_0}=\frac{1}{d}\sum_{k=1}^d\big(x_k^2+\epsilon_0\big)^{\frac{1}{2}}.$$
For each $\theta>0$ introduce the modulus of continuity of $u$ on $\overline{\BigO}_r$,
$$\omega_u(\theta)=\sup\big\{\abs{u(t,x)-u(s,y)}\mid (s,x),(s,y)\in \overline{\BigO}_r \text{ and } \abs{t-s}^2+\Norm{x-y}_1\leq \theta^2 \big\}.$$
Since $u$ is continuous, and therefore uniformly continuous on $\overline{\BigO}_r$, it is possible to find $\theta>0$ sufficiently small that $\omega_u(\theta)<\delta$. With this $\theta>0$ at hand, define the function $\Psi_{\epsilon,\delta,\theta}: \overline{\BigO}_r\times \overline{\BigO}_r\to \R$ by
$$\Psi_{\epsilon,\delta,\theta}(t,x,s,y)=u(t,x)-\phi(s,y)-\frac{\epsilon}{\udo (y)}-\frac{2M_u}{\theta^2}\abs{t-s}^2-\frac{2M_u}{\theta^2}\Norm{x-y}_{1,\epsilon_0}+\delta \zeta_\epsilon(s,y),$$
where $\smash{M_u=\sup_{(t,x)\in \overline{\BigO}_r}\abs{u(t,x)}}$. Observe that $\smash{\Psi_{\epsilon,\delta,\theta}(s,y,s,y)=\psi_\epsilon(s,y)+\delta \zeta_\epsilon(s,y)}$. We now show that the maximizer $(t_0,x_0,s_0,y_0)$ of this function belongs to the open set $\BigO_r\times \BigO_r$. Given points $\smash{(t,x),(s,y)\in \overline{\BigO}_r\times\overline{\BigO}_r}$ with $\smash{\abs{t-s}^2+\Norm{x-y}_1\leq \theta^2}$, the triangle inequality and the definition of the modulus of continuity reveal that
$$\Psi_{\epsilon,\delta,\theta}(t,x,s,y)\leq \omega_u(\theta)+u(s,y)-\phi(s,y)-\frac{\epsilon}{\udo(y)}+\delta \zeta_\epsilon(s,y)=\omega_u(\theta)+\psi_\epsilon(s,y)+\delta \zeta_\epsilon(s,y).$$
On the other hand, for $\smash{(t,x),(s,y)\in \overline{\BigO}_r\times\overline{\BigO}_r}$ with $\smash{\abs{t-s}^2+\Norm{x-y}_1> \theta^2}$, the triangle inequality, the bound $\Norm{x-y}_{1,\epsilon_0}\geq \Norm{x-y}_1$ and the definition of $M_u$ imply that
$$\Psi_{\epsilon,\delta,\theta}(t,x,s,y)\leq u(s,y)-\phi(s,y)-\frac{\epsilon}{\udo(y)}+\delta \zeta_\epsilon(s,y)\leq \omega_u(\theta)+\psi_\epsilon(s,y)+\delta \zeta_\epsilon(s,y).$$
It follows that for any $(t,x)\in \overline{\BigO}_r$ and every $(s,y)\in \overline{\BigO}_r\setminus \supp \zeta_\epsilon$,
\begin{align*}
\Psi_{\epsilon,\delta,\theta}(t,x,s,y)&\leq \omega_u(\theta)+\psi_\epsilon(s_\epsilon,y_\epsilon)+\delta \zeta_\epsilon(s,y)=\omega_u(\theta)+\Psi_{\epsilon,\delta,\theta}(s_\epsilon,y_\epsilon,s_\epsilon,y_\epsilon)-\delta\\
&<\Psi_{\epsilon,\delta,\theta}(s_\epsilon,y_\epsilon,s_\epsilon,y_\epsilon),
\end{align*}
where we have used that $\zeta_\epsilon(s_\epsilon,y_\epsilon)=1$ and $\omega_u(\theta)<\delta$. This means that $(s_0,y_0)\in \supp \zeta_\epsilon\subset \BigO_r$. To show that $(t_0,x_0)$ also belongs to this open set, suppose that $\abs{t_0-s_0}^2+\Norm{x_0-y_0}_1> \theta^2$. The triangle inequality, the bound $\Norm{x_0-y_0}_{1,\epsilon_0}\geq \Norm{x_0-y_0}_1$ and the definition of $M_u$ imply that
$$\Psi_{\epsilon,\delta,\theta}(t_0,x_0,s_0,y_0)\leq u(t_0,x_0)-\phi(s_0,y_0)-\frac{\epsilon}{\udo(y_0)}-2M_u+\delta \zeta_\epsilon(s_0,y_0)\leq \Psi_{\epsilon,\delta,\theta}(s_0,y_0,s_0,y_0),$$
so, up to replacing $(t_0,x_0)$ with $(s_0,y_0)$, we may assume without loss of generality that
\begin{equation}\label{eqn: SBME ignoring the boundary limit of maximizer}
\abs{t_0-s_0}^2+\Norm{x_0-y_0}_{1}\leq \theta^2.
\end{equation}
Decreasing $\theta$ if necessary and recalling that $\BigO_r$ is open shows that $(t_0,x_0,s_0,y_0)\in \BigO_r\times \BigO_r$. Since the function $\smash{(t,x)\mapsto \Psi_{\epsilon,\delta,\theta}(t,x,s_0,y_0)}$ has a local maximum at $(t_0,x_0)$, the subsolution criterion for \eqref{eqn: SBME app HJ eqn on Rpp} implies that
\begin{equation}\label{eqn: SBME ignoring the boundary key inequality}
\frac{4M_u}{\theta^2}(t_0-s_0)-\H\bigg(\frac{2M_u}{\theta^2}z\bigg)\leq 0
\end{equation}
for the vector $z\in \R^d$ defined by
$$z_k=\frac{(x_0-y_0)_k}{d((x_0-y_0)_k^2+\epsilon_0)^{\frac{1}{2}}}.$$
On the other hand, since the function $\smash{(s,y)\mapsto \Psi_{\epsilon,\delta,\theta}(t_0,x_0,s,y)}$ achieves its maximum at an interior point $(s_0,y_0)\in \BigO_r$, a direct computation together with \Cref{SBME properties of distance-like function} shows that
\begin{align}\label{eqn: SBME ignoring the boundary time derivative}
\partial_t\phi(s_0,y_0)-\frac{4M_u}{\theta^2}(t_0-s_0)-\delta\partial_t \zeta_\epsilon(s_0,y_0)&=0,\\
\frac{\udo(y_0)^2}{\epsilon}\Big(\nabla \phi(s_0,y_0)+\frac{2M_u}{\theta^2}\widetilde{z}-\delta \nabla \zeta_\epsilon(s_0,y_0)\Big)&\in \partial\udo(y_0)\notag,
\end{align}
for the vector $\widetilde{z}\in \R^d$ defined by
$$\widetilde{z}_k=\frac{(y_0-x_0)_k}{d((y_0-x_0)_k^2+\epsilon_0)^{\frac{1}{2}}}=-z_k.$$
Remembering that $\partial\udo(y_0)\subset (\Rp^d)^*=\Rp^d$ by \Cref{SBME properties of distance-like function}, it is possible to find $p\geq 0$ with
$$\frac{2M_u}{\theta^2}z=\nabla \phi(s_0,y_0)-\delta \nabla \zeta_\epsilon(s_0,y_0)-p.$$
Substituting this and \eqref{eqn: SBME ignoring the boundary time derivative} into \eqref{eqn: SBME ignoring the boundary key inequality} and using the fact that the non-linearity $\H$ is non-decreasing reveals that
\begin{equation}\label{eqn: SBME ignoring the boundary key inequality 2}
\partial_t\phi(s_0,y_0)-\delta \partial_t\zeta_\epsilon(s_0,y_0)-\H\big(\nabla \phi(s_0,y_0)-\delta \nabla \zeta_\epsilon(s_0,y_0)\big)\leq 0.
\end{equation}
Recalling that $(s_0,y_0)\in \supp \zeta_\epsilon$ depends on $\epsilon, \delta$ and $\theta$, and that $\theta$ was chosen small enough in terms of $\delta$, we would now like to let $\theta\to 0$ and then $\delta \to 0$ in this inequality. Observe that for any $(t,x)\in \overline{\BigO}_r$,
\begin{align*}
u(t_0,x_0)-\phi(s_0,y_0)-\frac{\epsilon}{\udo(y_0)}+\delta \zeta_\epsilon(s_0,y_0)&\geq \Psi_{\epsilon,\delta,\theta}(t_0,x_0,s_0,y_0)\\
&\geq u(t,x)-\phi(t,x)-\frac{\epsilon}{\udo (x)}+\delta\zeta_\epsilon(t,x)\\
&=\psi_\epsilon(t,x)+\delta \zeta_\epsilon(t,x).
\end{align*}
If we denote by $(t_1,x_1)\in \supp \zeta_\epsilon\subset \BigO_r$ and $(t_1',x_1')\in \supp \zeta_\epsilon\subset \BigO_r$ subsequential limits of the sequences $(t_0,x_0)$ and $(s_0,y_0)$ as $\theta\to 0$ and then $\delta \to 0$, we must have $t_1=t_1'$ and $x_1=x_1'$ by \eqref{eqn: SBME ignoring the boundary limit of maximizer}. Moreover, the subsequential limit $(t_1,x_1)$ must satisfy the inequality
$$u(t_1,x_1)-\phi(t_1,x_1)\geq u(t_1,x_1)-\phi(t_1,x_1)-\frac{\epsilon}{\udo (x_1)}\geq u(t,x)-\phi(t,x)-\frac{\epsilon}{\udo(x)}$$
for all $(t,x)\in \BigO_r$. Writing $(t_2,x_2)\in \overline{\BigO}_r$ for a subsequential limit of the sequence $(t_1,x_1)$ as $\epsilon\to 0$ we find that
$$u(t_2,x_2)-\phi(t_2,x_2)\geq u(t,x)-\phi(t,x)$$
for all $(t,x)\in \BigO_r$. By continuity of $u-\phi$, this inequality extends to $\overline{\BigO}_r$. Since $(t^*,x^*)$ is a strict local maximum of $u-\phi$ on $\overline{\BigO}_r$, we must have $(t_2,x_2)=(t^*,x^*)$. It follows by letting $\theta\to 0$, then $\delta\to 0 $ and finally $\epsilon \to 0$ in \eqref{eqn: SBME ignoring the boundary key inequality 2} that
$$\partial_t \phi(t^*,x^*)-\H\big(\nabla \phi(t^*,x^*)\big)\leq 0.$$
This completes the proof.
\end{proof}

\begin{corollary}\label{SBME comparison principle on Rpp}
If $\smash{\H:\R^d\to \R}$ and $\smash{\psi:\Rp^d\to \R}$ satisfy \eqref{eqn: SBME app A1}-\eqref{eqn: SBME app A3}, then the Hamilton-Jacobi equation \eqref{eqn: SBME app HJ eqn on Rpp} admits a unique viscosity solution $f\in \mathfrak{L}$ subject to the initial condition $\psi$. Moreover,
\begin{equation}
\sup_{t>0}\Norm{f(t,\cdot)}_{\mathrm{Lip},1}=\Norm{\psi}_{\mathrm{Lip},1},
\end{equation}
and if $u,v\in \mathfrak{L}_{\mathrm{unif}}$ are respectively a continuous subsolution and a continuous supersolution to \eqref{eqn: SBME app HJ eqn on Rpp}, then
\begin{equation}
\sup_{\Rp\times \Rp^d}\big(u(t,x)-v(t,x)\big)=\sup_{\Rp^d}\big(u(0,x)-v(0,x)\big).
\end{equation}
To be more specific, if $\smash{L=\max\big(\sup_{t>0}\Norm{u(t,\cdot)}_{\mathrm{Lip},1},\sup_{t>0}\Norm{v(t,\cdot)}_{\mathrm{Lip},1}\big)}$ and $\smash{V=\Norm{\H}_{\mathrm{Lip},1,*}}$, then for every $Q>2L$ and $R\in \R$, the map
\begin{equation}
(t,x)\mapsto u(t,x)-v(t,x)-Q\big(\Norm{x}_1+Vt-R\big)_+
\end{equation}
achieves its supremum on $\smash{\{0\}\times \Rp^d}$.
\end{corollary}

\begin{proof}
This is an immediate consequence of \Cref{SBME WP of HJ eqn on Rp}, \Cref{SBME comparison principle on Rpd}, \Cref{SBME comparison principle corollary on Rpd} and \Cref{SBME monotonic non-linearities no boundary condition}.
\end{proof}

\subsection{Monotonicity of solutions on \texorpdfstring{$\Rp^d$}{Rpd}}

Recall that the notion of being $\CC^*$-non-decreasing is introduced at the beginning of Section~\ref{SBME section HJ on cone}. In this section, we follow the arguments in Section 4 of \cite{HB_cone} to show that the solution to the Hamilton-Jacobi equation \eqref{eqn: SBME app HJ eqn on Rp} preserves the monotonicity of its initial condition. To be more specific, we assume that
\begin{enumerate}[label = \textbf{A\arabic*}]
\setcounter{enumi}{3}
    \item the initial condition $\psi:\R^d\to \R$ is $\CC^*$-non-decreasing for some closed convex cone $\CC\subset \R^d$,\label{SBME app A4}
\end{enumerate}
and under a mild assumption on the dual cone $\CC^*$, we show that the solution to the Hamilton-Jacobi equation \eqref{eqn: SBME app HJ eqn on Rp} constructed in \Cref{SBME WP of HJ eqn on Rp} is also $\CC^*$-non-decreasing. This result will be used in \Cref{SBME section HJ on cone} when the well-posedness of the projected Hamilton-Jacobi equations \eqref{eqn: SBME projected HJ eqn} is established by means of \Cref{SBME WP of HJ eqn on Rp}. Indeed, it will allow us to verify the first condition in \eqref{eqn: SBME subsolution conditions} and \eqref{eqn: SBME supersolution conditions}.

\begin{proposition}\label{SBME monotone initial conition gives monotone solution}
Fix a non-linearity $\H:\R^d\to \R$ and an initial condition $\smash{\psi:\Rp^d\to \R}$ satisfying \eqref{eqn: SBME app A1}-\eqref{SBME app A4}. If $\smash{\Rp^d\cap \mathrm{int}(\CC^*)\neq \emptyset}$ and $f\in \mathfrak{L}$ is a viscosity solution to the Hamilton-Jacobi equation \eqref{eqn: SBME app HJ eqn on Rp} subject to the initial condition $\psi$, then $f$ is $\smash{\CC^*}$-non-decreasing.
\end{proposition}

\begin{proof}
Introduce the set
$\O=\big\{(x,x')\in \Rp^d\times \Rp^d\mid x'-x\in \CC^*\big\},$
and suppose for the sake of contradiction that there exists $T>0$ with
\begin{equation}\label{eqn: SBME HJ solution monotonic absurd 0}
\sup_{\substack{t\in [0,T]\\(x,x')\in \O}}\big(f(t,x)-f(t,x')\big)>0\geq \sup_{(x,x')\in \O}\big(f(0,x)-f(0,x')\big).
\end{equation}
The proof proceeds in three steps: first we perturb \eqref{eqn: SBME HJ solution monotonic absurd 0}, then we use a variable doubling argument to obtain a system of inequalities, and finally we contradict this system of inequalities.

\noindent\step{1: perturbing.}\\
Let $V=\Norm{\H}_{\mathrm{Lip},1,*}$ and fix a constant $L>0$ with
$$L>\Norm{\psi}_{\mathrm{Lip},1} \quad \text{and}\quad \lvert f(t,x)-\psi(x)\rvert \leq Lt$$
for all $(t,x)\in\Rpp\times \Rp^d$. The existence of such a constant follows from the assumption $f\in \mathfrak{L}$. Denote by $\smash{\udo:\CC^*\to \Rp}$ the distance-like function \eqref{eqn: SBME distance-like function} associated with the cone $\CC$,
$$\udo(y)=\inf_{\substack{\dNorm{1}{y'}=1\\y'\in \CC}}y'\cdot y.$$
Fix $\smash{y_0\in \Rp^d\cap  \mathrm{int}(\CC^*)}$ as well as $x_0\in \Rp^d$, and let $\theta\in C^\infty(\R)$ be an increasing function with $r_+\leq \theta(r)\leq (r+1)_+$ for all $r\in \R$. Given $0<\epsilon_0<1$ to be determined, consider the smoothed normalized-$\ell^1$ norm, 
\begin{equation*}
\Norm{x}_{1,\epsilon_0}=\frac{1}{d}\sum_{k=1}^d\big(x_k^2+\epsilon_0\big)^{\frac{1}{2}},
\end{equation*}
and introduce the function
$$\Phi(t,x)=\theta\big(\Norm{x}_{1,\epsilon_0}+Vt-R\big)$$
defined on $\Rp\times \R^d$, where $R>0$ is chosen large enough so that $\Phi(0,x_0)=0$. Increasing $R>0$ if necessary, it is possible to perturb the inequality \eqref{eqn: SBME HJ solution monotonic absurd 0} to ensure that
\begin{equation*}
\sup_{\substack{t\in [0,T]\\(x,x')\in \O}}\big(f(t,x)-f(t,x')-\Phi(t,x)\big)>0\geq \sup_{(x,x')\in \O}\big(f(0,x)-f(0,x')-\Phi(0,x)\big).
\end{equation*}
Picking $\delta>0$ small enough, it is also possible to guarantee that
\begin{align}\label{eqn: SBME HJ solution monotonic absurd 1}
\sup_{\substack{t\in [0,T]\\(x,x')\in \O}}&\Big(f(t,x)-f(t,x')-\delta t-\zeta(t,t)-\Phi(t,x)-\frac{\delta}{\udo(x'-x)}-2\delta \Norm{x-x'}_{1,\epsilon_0}^2\Big)\notag \\
&>\sup_{(x,x')\in \O}\Big(f(0,x)-f(0,x')-\zeta(0,0)-\Phi(0,x)-\frac{\delta}{\udo(x'-x)}-2\delta \Norm{x-x'}_{1,\epsilon_0}^2\Big)
\end{align}
for the perturbation function
$$\zeta(t,t')=\frac{\delta}{T-t}+\frac{\delta}{T-t'}.$$
This is a perturbed version of the absurd hypothesis \eqref{eqn: SBME HJ solution monotonic absurd 0}.\\
\step{2: system of inequalities.}\\
For each $\alpha\geq 1$, define the function $\Psi_\alpha:[0,T]\times [0,T]\times \O\times \CC^*\to \R\cup\{-\infty\}$ by
\begin{align}\label{eqn: SBME HJ solution monotonic Psi}
\Psi_\alpha(t,t',x,x',y)=f(t,x)-f(t',x')-&\Phi(t,x)-\psi_\alpha(x,x',y)\notag\\
&-\delta t-\zeta(t,t')-\alpha\abs{t-t'}^2-\delta \Norm{x-x'}_{1,\epsilon_0}^2,
\end{align}
where
$$\psi_\alpha(x,x',y)=\alpha \Norm{x'-x-y}_{1,\epsilon_0}^2+\frac{\delta}{\udo(y)}+\delta \Norm{y}_{1,\epsilon_0}^2.$$
Observe that $\Psi_\alpha(t,t,x,x',x'-x)$ coincides with the function being maximized in \eqref{eqn: SBME HJ solution monotonic absurd 1}. By doubling the variables in this way, we ensure that the function $\Psi_\alpha$ achieves its supremum at a point $(t_\alpha,t_\alpha',x_\alpha,x_\alpha', y_\alpha)$ which remains bounded as $\alpha$ tends to infinity. Indeed, if we temporarily fix $\alpha\geq 1$ and let $\smash{(t_{\alpha,n},t_{\alpha,n}',x_{\alpha,n},x_{\alpha,n}',y_{\alpha,n})}$ be a maximizing sequence for $\Psi_\alpha$, then the choice of $x_0$ implies that when $n$ is large enough,
\begin{equation}\label{eqn: SBME HJ solution monotonic Psi lower bounded}
\Psi_\alpha(t_{\alpha,n},t_{\alpha,n}',x_{\alpha,n}, x_{\alpha,n}',y_{\alpha,n})\geq \Psi_\alpha(0,0,x_0,x_0+y_0,y_0)=C_0
\end{equation}
for the constant $\smash{C_0=f(0,x_0)-f(0,x_0+y_0)-\zeta(0,0)-\frac{\delta}{\udo (y_0)}-2\delta\norm{y_0}_{1,\epsilon_0}^2}$. We have used the fact that $\smash{y_0\in \Rp^d}$ and that $\smash{\Rp^d}$ is a cone. Combining this with the Lipschitz bound \begin{equation}\label{eqn: SBME HJ solution monotonic f Lipschitz}
\abs{f(t,x)-f(t',x')}\leq L(t+t')+L\Norm{x-x'}_1\leq L(t+t')+L\Norm{x-x'}_{1,\epsilon_0}
\end{equation}
and the fact that $\Phi(t,x)\geq \Norm{x}_{1,\epsilon_0}-R$ reveals that
\begin{align}\label{eqn: SBME HJ solution monotonic Psi maximizer bounded}
Lt_{\alpha,n}+Lt_{\alpha,n}'+L\Norm{x_{\alpha,n}-x_{\alpha,n}'}_{1,\epsilon_0}+R-&\Norm{x_{\alpha,n}}_{1,\epsilon_0} \notag\\
&-\delta \Norm{y_{\alpha,n}}_{1,\epsilon_0}^2-\delta \Norm{x_{\alpha,n}-x_{\alpha,n}'}_{1,\epsilon_0}^2\geq C_0.
\end{align}
Noticing that $t_{\alpha,n},t_{\alpha,n}'<T$ due to the presence of $\zeta$ in the function $\Psi_\alpha$ and observing that the quadratic function $\smash{r\mapsto Lr-\delta r^2}$ is bounded by $\smash{\frac{L^2}{4\delta}}$ gives the uniform boundedness of $x_{\alpha,n}$ and $y_{\alpha,n}$ in both $n$ and $\alpha$ with respect to the normalized-$\smash{\ell^1}$ norm. Rearranging the lower bound \eqref{eqn: SBME HJ solution monotonic Psi maximizer bounded} also shows that
$$2LT+L\Norm{x_{\alpha,n}-x_{\alpha,n}'}_{1,\epsilon_0}+R-C_0\geq \delta \norm{x_{\alpha,n}-x_{\alpha,n}'}_{1,\epsilon_0}^2$$
which gives the uniform boundedness of $x_{\alpha,n}-x'_{\alpha,n}$ and hence $x_{\alpha,n}'$ in both $n$ and $\alpha$ with respect to the normalized-$\smash{\ell^1}$ norm. It is therefore possible to let $n$ tend to infinity along a subsequence to obtain a maximizer $(t_\alpha,t_\alpha',x_\alpha,x_\alpha',y_\alpha)$ of $\Psi_\alpha$ all of whose components are bounded by some constant $C_1>0$ that is independent of $\alpha$ with respect to the normalized-$\smash{\ell^1}$ norm. Choosing $\epsilon_0$ small enough, these components will also be assumed to be bounded by $C_1>0$ with respect to the smoothed version of the normalized-$\smash{\ell^1}$ norm. We now obtain some essential bounds on the components of this maximizer. Taking the limit as $n$ tends to infinity in the inequality \eqref{eqn: SBME HJ solution monotonic Psi lower bounded} reveals that
$$f(t_\alpha,x_\alpha)-f(t_\alpha',x_\alpha') -\frac{\delta}{\udo(y_\alpha)}-\delta\Norm{y_\alpha}_{1,\epsilon_0}^2-\alpha\abs{t_\alpha-t_{\alpha}'}^2-\delta \Norm{x_\alpha-x_\alpha'}_{1,\epsilon_0}^2\geq C_0.$$
Combining this with \eqref{eqn: SBME HJ solution monotonic f Lipschitz} gives
\begin{align*}
\alpha\abs{t_\alpha-t_\alpha'}^2+\frac{\delta}{\udo(y_\alpha)}+\delta \Norm{y_\alpha}_{1,\epsilon_0}^2&\leq L(t_\alpha+t_\alpha')+L\Norm{x_\alpha-x_\alpha'}_{1,\epsilon_0}-\delta \Norm{x_\alpha-x_\alpha'}_{1,\epsilon_0}^2-C_0\\
&\leq 2LT+\frac{L^2}{4\delta}-C_0,
\end{align*}
where we again used the fact that the quadratic function $\smash{r\mapsto Lr-\delta r^2}$ is bounded by $\smash{\frac{L^2}{4\delta}}$. If we introduce the constant $\smash{C_2=\max\big(2LT+\frac{L^2}{4\delta}-C_0,1\big)}$, this upper bound implies that
\begin{equation}\label{eqn: SBME HJ solution monotonic Psi maximizer bounds}
\abs{t_\alpha-t_\alpha'}\leq \sqrt{\frac{C_2}{\alpha}},\qquad \udo(y_\alpha)\geq \frac{\delta}{C_2}, \qquad \norm{y_\alpha}_{1,\epsilon_0}\leq \sqrt{\frac{C_2}{\delta}}.
\end{equation}
In particular $y_\alpha\in \mathrm{int}(\CC^*)$, so $\dNorm{1}{p}\leq 1$ for every $p\in \partial \udo(y_\alpha)$ by \Cref{SBME properties of distance-like function}. To leverage this observation, notice that $y\mapsto \Psi_\alpha(t_\alpha,t_\alpha',x_\alpha,x_\alpha',y)$ achieves a local maximum at $y_\alpha$, and therefore so does $y\mapsto -\psi_\alpha(x_\alpha,x_\alpha',y)$. It follows by a direct computation and \Cref{SBME properties of distance-like function} that the vector $p\in \R^d$ defined by
$$p_k=\frac{2\udo(y_\alpha)^2}{\delta}\bigg(\alpha \Norm{x_\alpha'-x_\alpha-y_\alpha}_{1,\epsilon_0}\frac{(x_\alpha'-x_\alpha-y_\alpha)_k}{((x_\alpha'-x_\alpha-y_\alpha)_k^2+\epsilon_0)^{\frac{1}{2}}}+\delta \Norm{y_\alpha}_{1,\epsilon_0}\frac{(y_\alpha)_k}{((y_\alpha)_k^2+\epsilon_0)^{\frac{1}{2}}}\bigg)$$
belongs to the super-differential $\partial\udo(y_\alpha)$. This means that
$$\alpha \Norm{x_\alpha'-x_\alpha-y_\alpha}_{1,\epsilon_0}\max_{1\leq k\leq d}\bigg\lvert \frac{(x_\alpha'-x_\alpha-y_\alpha)_k}{((x_\alpha'-x_\alpha-y_\alpha)_k^2+\epsilon_0)^{\frac{1}{2}}}\bigg\rvert\leq \frac{\delta}{2d \udo(y_\alpha)^2}+\delta \Norm{y_\alpha}_{1,\epsilon_0},$$
where we have used the bounds $\dNorm{1}{p}\leq 1$ and $\smash{\abs{(y_\alpha)_k}\leq ((y_\alpha)_k^2+\epsilon_0)^{\frac{1}{2}}}$. To bound this further, suppose that
\begin{equation}\label{eqn: SBME HJ solution monotonic Psi maximizer y_alpha bound}
\Norm{x_\alpha'-x_\alpha-y_\alpha}_{1,\epsilon_0}>C
\end{equation}
for some constant $C$ to be determined, and let $1\leq k^*\leq d$ be such that $\smash{((x_\alpha'-x_\alpha-y_\alpha)_{k^*}^2+\epsilon_0)^{\frac{1}{2}}>C}$. Observe that for any $z\in \R^d$ with $\smash{(z_{k^*}^2+\epsilon_0)^{\frac{1}{2}}>C}$,
$$\frac{\abs{z_{k^*}}}{(z_{k^*}^2+\epsilon_0)^{\frac{1}{2}}}=\frac{\abs{z_{k^*}}+\sqrt{\epsilon_0}}{(z_{k^*}^2+\epsilon_0)^{\frac{1}{2}}}-\frac{\sqrt{\epsilon_0}}{(z_{k^*}^2+\epsilon_0)^{\frac{1}{2}}}\geq 1-\frac{\sqrt{\epsilon_0}}{C}=\frac{C-\sqrt{\epsilon_0}}{C},$$
where we have used the fact that $\abs{z_k}+\sqrt{\epsilon_0}\geq (z_k^2+\epsilon_0)^{\frac{1}{2}}$. Together with \eqref{eqn: SBME HJ solution monotonic Psi maximizer bounds}, this implies that
$$\frac{\alpha(C-\sqrt{\epsilon_0})}{C}\Norm{x_\alpha'-x_\alpha-y_\alpha}_{1,\epsilon_0}\leq \frac{\delta}{2d \udo(y_\alpha)^2}+\delta \Norm{y_\alpha}_{1,\epsilon_0}\leq K$$
for the constant $\smash{K=\frac{C_2^2}{2d\delta}+\sqrt{\delta C_2}}$. Rearranging, remembering \eqref{eqn: SBME HJ solution monotonic Psi maximizer y_alpha bound} and choosing $\smash{C=\sqrt{\epsilon_0}+\frac{K}{\alpha}}$ reveals that
$$\Norm{x_\alpha'-x_\alpha-y_\alpha}_1\leq \Norm{x_\alpha'-x_\alpha-y_\alpha}_{1,\epsilon_0}\leq \max\Big(C,\frac{KC}{\alpha(C-\sqrt{\epsilon_0})}\Big)=\frac{K}{\alpha}+\sqrt{\epsilon_0}.$$
Letting $\epsilon_0$ tend to zero in this upper bound yields
\begin{equation}\label{eqn: SBME HJ solution monotonic y maximizer}
\Norm{x_\alpha'-x_\alpha-y_\alpha}_1\leq \frac{K}{\alpha}.
\end{equation}
Combining this with the first bound in \eqref{eqn: SBME HJ solution monotonic Psi maximizer bounds} and the fact that each component in the sequence of maximizers $(t_\alpha,t_\alpha,x_\alpha,x_\alpha',y_\alpha)$ is uniformly bounded by a constant independent of $\alpha$ gives the existence of a subsequential limit $(t_\infty,t_\infty,x_\infty,x_\infty',x_\infty'-x_\infty)$ with respect to the normalized-$\ell^1$ norm. Observe that for any $t\in [0,T)$ and every $(x,x')\in \O$,
\begin{align*}
f(t_\alpha,x_\alpha)-f(t_\alpha',x_\alpha')-\delta t_\alpha-\zeta(t_\alpha,t_\alpha')-\Phi(t_\alpha,x_\alpha)&-\frac{\delta}{\udo(y_\alpha)}-\delta\Norm{y_\alpha}_{1,\epsilon_0}^2-\delta\Norm{ x_\alpha-x_\alpha'}_{1,\epsilon_0}\\
&\geq \Psi_\alpha(t_\alpha,t_\alpha',x_\alpha,x_\alpha',y_\alpha)\geq \Psi_\alpha(t,t,x,x',x'-x).
\end{align*}
Taking the supremum over $(t,x,x')\in [0,T]\times \O$, recalling that $\Psi_\alpha(t,t,x,x',x'-x)$ coincides with the function being maximized in \eqref{eqn: SBME HJ solution monotonic absurd 1} and letting $\alpha\to \infty$ shows that $t_\infty>0$. At this point, we can use the fact that $f$ is a viscosity solution to the Hamilton-Jacobi equation \eqref{eqn: SBME app HJ eqn on Rp} to obtain a system of inequalities. Using the second inequality in \eqref{eqn: SBME HJ solution monotonic Psi maximizer bounds}, the bound \eqref{eqn: SBME HJ solution monotonic y maximizer} and the observation that $t_\infty>0$, fix $\alpha\geq 1$ large enough so that $\smash{x_\alpha'-x_\alpha\in \mathrm{int}(\CC^*)}$ and $\smash{t_\alpha,t_\alpha'>0}$. Introduce the smooth functions
$$\phi(t,x)=f(t,x)-\Psi_\alpha(t,t_\alpha',x,x_\alpha', y_\alpha) \quad \text{and} \quad \phi'(t',x')=f(t',x')+\Psi_\alpha(t_\alpha,t',x_\alpha,x',y_\alpha)$$
defined on $(0,\infty)\times \Rp^d$. Since $(t_\alpha,t_\alpha', x_\alpha, x_\alpha', y_\alpha)$ maximizes $\Psi_\alpha$, the function $f-\phi$ achieves a local maximum at $\smash{(t_\alpha,x_\alpha)\in (0,\infty)\times \Rp^d}$ while the function $f-\phi'$ achieves a local minimum at $\smash{(t_\alpha', x_\alpha')\in (0,\infty)\times \Rp^d}$. It follows by definition of a viscosity solution that
\begin{equation}\label{eqn: SBME HJ solution monotonic system of inequaltities}
\partial_t\phi(t_\alpha,x_\alpha)-\H\big(\nabla \phi(t_\alpha,x_\alpha)\big)\leq 0 \quad \text{and} \quad \partial_t\phi'(t_\alpha',x_\alpha')-\H\big(\nabla \phi'(t_\alpha',x_\alpha')\big)\geq 0.
\end{equation}
This is the system of inequalities that we now strive to contradict.\\
\noindent \step{3: reaching a contradiction.}\\
The choice $V=\Norm{\H}_{\mathrm{Lip},1,*}$ and a direct computation reveal that
\begin{align*}
\partial_t\phi'(t_\alpha',x_\alpha')-\H\big(\nabla \phi'(t_\alpha',x_\alpha')\big)
&< \delta +\frac{\delta}{(T-t_\alpha)^2}+V\dNorm{1}{\nabla\Phi(t_\alpha,x_\alpha)}+2\alpha(t_\alpha-t_\alpha')\\
&\qquad\qquad\qquad\qquad\qquad\qquad\qquad\qquad\qquad -\H\big(\nabla \phi(t_\alpha,x_\alpha)\big).
\end{align*}
Another direct computation shows that 
$$\partial_t\phi(t_\alpha,x_\alpha)= \delta+\frac{\delta}{(T-t_\alpha)^2}+\partial_t\Phi(t_\alpha,x_\alpha)+2\alpha(t_\alpha-t_\alpha')$$
and that $dV\abs{\partial_{x_k}\Phi(t_\alpha,x_\alpha)}\leq \partial_t\Phi(t_\alpha,x_\alpha)$. It follows by the first inequality in \eqref{eqn: SBME HJ solution monotonic system of inequaltities} that
$$\partial_t\phi'(t_\alpha',x_\alpha')-\H\big(\nabla \phi'(t_\alpha',x_\alpha')\big)< \partial_t\phi(t_\alpha,x_\alpha)-\H\big(\nabla \phi(t_\alpha,x_\alpha)\big)\leq 0$$
which contradicts the second inequality in \eqref{eqn: SBME HJ solution monotonic system of inequaltities} and completes the proof.
\end{proof}

\section{Background material}\label{SBME app background}

In this appendix, we establish three elementary results in analysis. The first is a classical result in convex analysis regarding the bidual of a closed convex cone. Recall that the dual of a convex cone $\K\subset \R^d$ is the closed convex cone
\begin{equation}
\label{e.def.dual.cone}
\K^*=\big\{x\in \R^d\mid x\cdot y\geq 0 \text{ for all } y\in \K\big\}.
\end{equation}
It is clear that any convex cone $\K$ is always a subset of its bidual $\K^{**}$. Since $\smash{\K^{**}}$ is closed, a necessary condition for this containment to be an equality is that $\K$ be closed; it turns out that this is also a sufficient condition. This is often deduced from the Hahn-Banach separation theorem \cite{Lemarechal} or the Fenchel-Moreau theorem \cite{Bertsekas}. For the reader's convenience we prove this duality result using the Hahn-Banach separation theorem as stated in Theorem 4.1.1 of \cite{Lemarechal}.

\begin{proposition}\label{SBME closed convex cone bidual}
If $\K\subset \R^d$ is a non-empty closed convex cone, then $\K=\K^{**}$.
\end{proposition}

\begin{proof}
It is clear that $\K\subset \K^{**}$. Suppose for the sake of contradiction that there exists $x\in \K^{**}$ with $x\notin \K$. Since $\K$ is a non-empty closed convex set, the Hahn-Banach separation theorem gives $\alpha\in \R^d$ with
\begin{equation}\label{eqn: SBME closed convex cone bidual key}
\alpha \cdot x> \sup\{\alpha \cdot y \mid y\in \K\}.
\end{equation}
Given $x_0\in \K$, the assumption that $\K$ is closed implies that $\smash{0=\lim_{n\to \infty}\frac{1}{n}x_0\in \K}$. Together with \eqref{eqn: SBME closed convex cone bidual key}, this means that $\alpha\cdot x>0$. If there were $y_0\in \K$ with $\alpha \cdot y> 0$, the fact that $\K$ is a cone would imply that $\alpha \cdot x\geq \lambda \alpha \cdot y_0$ for all $\lambda>0$, and letting $\lambda$ tend to infinity would give a contradiction. It follows by \eqref{eqn: SBME closed convex cone bidual key} that
$$\alpha \cdot x >0 = \sup\{\alpha \cdot y \mid y\in \K\},$$
where we have used that $0\in \K$. The lower bound implies that $-\alpha\in \K^*$ while the upper bound gives $x\cdot (-\alpha)< 0$. This contradicts the assumption that $\smash{x\in \K^{**}}$ and completes the proof.
\end{proof}

The second also belongs to the realm of convex analysis, and it gives a non-differential characterization of a Lipschitz function having its gradient in a closed convex set.
\begin{proposition}\label{SBME gradient in convex set}
If $\K\subset \R^d$ is a closed convex set and $\psi:\R^d\to \R$ is a Lipschitz function, then $\nabla \psi\in \K$ if and only if the following holds. For every $c\in \R$ and $x,x'\in \R^d$ with the property that for every $z\in \K$,  $(x'-x)\cdot z\geq c$, we have $\psi(x')-\psi(x)\geq c$.
\end{proposition}

\begin{proof}
Suppose that $\nabla \psi\in \K$, and fix $c\in \R$ and $x,x'\in \R^d$ with $(x'-x)\cdot z\geq c$ for every $z\in \K$. If we knew that $\psi$ was almost everywhere differentiable along the line joining $x$ and $x'$, we could apply the fundamental theorem of calculus to the one-dimensional Lipschitz function $t\mapsto \psi(x+t(x'-x))$ and conclude that
$$\psi(x')-\psi(x)=\int_0^1 \nabla \psi\big(x+t(x'-x)\big)\cdot (x'-x)\ud t\geq c.$$
Although $\psi$ could fail to be differentiable almost everywhere on the line joining $x$ and $x'$, we will now fix $\epsilon>0$ and show that it must be differentiable almost everywhere on some line joining some point $\smash{x_\epsilon\in B_\epsilon(x)}$ and some point $\smash{x'_\epsilon \in B_\epsilon(x')}$. Denote by 
$$\HH=\big\{y\in \R^d \mid y\cdot (x'-x)=0\big\}\cong \R^{d-1}$$
the hyperplane perpendicular to the line segment joining $x$ and $x'$, and write
$$\A_{\epsilon,x}=B_\epsilon(x)\cap(x+\HH)$$
for the cross-section of $B_\epsilon(x)$ through $x$ and perpendicular to the line segment joining $x$ and $x'$. Denote by $\mathcal{\L}$ the set of line segments between points in $\smash{\A_{\epsilon,x}}$ and points in $\smash{\A_{\epsilon,x'}}$ which are parallel to the line segment joining $x$ and $x'$. For each $\smash{y\in \A_{\epsilon,x}}$, write $\smash{\ell_y\in \L}$ for the unique line segment in $\smash{\A_{\epsilon,x}}$ through $y$, and introduce the set
$$\D_{y}=\big\{z\in \ell_y\mid \psi \text{ is not differentiable at } z\big\}$$
of points on $\ell_y$ at which $\psi$ is not differentiable. If $\D_y$ were a set of positive one-dimensional Lebesgue measure $\smash{m_1(\D_y)>0}$ for every $y\in \A_{\epsilon,x}$, then the $d$-dimensional Lebesgue measure of the set of points in $\cup_{y\in \A_{\epsilon,x}}\ell_y$ at which $\psi$ is not differentiable would have positive measure,
$$\int_{\A_{\epsilon,x}}m_1(\D_y)\ud y>0.$$
This would contradict Rademacher's theorem on the almost everywhere differentiability of Lipschitz functions (see Theorem 6 in Chapter 5.8 of \cite{Evans}). It is therefore possible to find $\smash{x_\epsilon\in \A_{\epsilon,x}}$ with $\smash{m_1(\D_{x_\epsilon})=0}$. If we write $\smash{x_\epsilon'\in \A_{\epsilon,x'}}$ for the right endpoint of $\smash{\ell_{x_\epsilon}}$, then the fundamental theorem of calculus implies that
$$\psi(x'_\epsilon)-\psi(x_\epsilon)=\int_0^1 \nabla \psi\big(x_\epsilon+t(x'_\epsilon-x_\epsilon)\big)\cdot (x'-x)\ud t\geq c.$$
Letting $\epsilon$ tend to zero shows that $\psi(x')-\psi(x)\geq c$ as required. Conversely, suppose that for every $c\in \R$ and $x,x'\in \R^d$ with the property that for every $z\in \K$,  $(x'-x)\cdot z\geq c$, we have $\psi(x')-\psi(x)\geq c$. Assume for the sake of contradiction that there exists $y\in \R^d$ with $\nabla \psi(y)\notin \K$. The Hahn-Banach separation theorem gives $v\in \R^d$ and $\delta>0$ with
$$v \cdot \nabla \psi(y)+\delta< \inf\{v \cdot z \mid z\in \K\}.$$
It follows that
$$\psi(y+\epsilon v)-\psi(y)\geq \epsilon \big(v\cdot \nabla \psi(y)+\delta\big).$$
Dividing by $\epsilon$ and letting $\epsilon$ tend to zero reveals that $\nabla \psi(y)\cdot v\geq v\cdot \nabla \psi(y)+\delta$.
This contradiction completes the proof.
\end{proof}

The third elementary result in analysis that we will prove regards the basic properties of semi-continuous envelopes. To strive for generality, fix a set $X\subset \R^d$ endowed with a norm $\norm{\cdot}$. Recall that a function $u:X\to \R$ is said to be \emph{upper semi-continuous} at a point $x\in X$ if
\begin{equation}
u(x)\geq \limsup_{y\to x}u(y):=\lim_{r\dec 0}\sup\big\{u(y)\mid y\in X \text{ with } \norm{y-x}\leq r\big\},
\end{equation}
and it is said to be \emph{lower semi-continuous} at a point $x\in X$ if
\begin{equation}
u(x)\leq \liminf_{y\to x}u(y):=\lim_{r\dec 0}\inf\big\{u(y)\mid y\in X\text{ with } \norm{y-x}\leq r\big\}.
\end{equation}
Moreover, the \emph{upper semi-continuous envelope} of $u$ is the function $u^\star:X\to \R$ defined by
\begin{equation}\label{eqn: SBME upper semi-continuous envelope}
u^\star(x)=\limsup_{y\to x}u(y)=\lim_{r\dec 0}\sup\big\{u(y)\mid y\in X \text{ with } \norm{y-x}\leq r\big\},
\end{equation}
while its \emph{lower semi-continuous envelope} is the function $u_\star:X\to \R$ defined by
\begin{equation}\label{eqn: SBME lower semi-continuous envelope}
u_\star(x)=\liminf_{y\to x}u(y)=\lim_{r\dec 0}\inf\big\{u(y)\mid y\in X\text{ with } \norm{y-x}\leq r\big\}.
\end{equation}
The following proposition collects the basic properties of semi-continuous envelopes. This result is used in \Cref{app SBME Perron} with $\smash{X=[0,\infty)\times \Rp^d}$ and $\smash{\norm{(t,x)}=\abs{t}+\Norm{x}_1}$.

\begin{proposition}\label{SBME properties of semi-continuous envelope}
The semi-continuous envelopes of a locally bounded function $u:X\to \R$ satisfy the following basic properties.
\begin{enumerate}
    \item $u_\star(x)\leq u(x)\leq u^\star(x)$ for all $x\in X$.
    \item $u^\star(x)=\min\{v(x)\mid u\leq v \text{ and } v \text{ is upper semi-continuous}\}$ for all $x\in X$. In particular, $u^\star$ is upper semi-continuous.
    \item $u_\star(x)=\max\{v(x)\mid v\leq u \text{ and } v \text{ is lower semi-continuous}\}$ for all $x\in X$. In particular, $u_\star$ is lower semi-continuous.
    \item $u$ is upper semi-continuous at $x\in X$ if and only if $u(x)=u^\star (x)$.
    \item $u$ is lower semi-continuous at $x\in X$ if and only if $u(x)=u_\star(x)$.
\end{enumerate}
\end{proposition}

\begin{proof}
To deduce properties of the lower semi-continuous envelope from the corresponding properties of the upper semi-continuous envelope we will leverage the observation that
\begin{align}\label{eqn: SBME lower and upper semi-continuous envelope relationship}
u_\star(x)&=\lim_{r\dec 0}\inf\big\{u(y)\mid y\in X\text{ with } \norm{y-x}\leq r\big\}\notag\\
&=-\lim_{r\dec 0}\sup\big\{-u(y)\mid y\in X\text{ with } \norm{y-x}\leq r\big\}=-(-u)^\star(x).
\end{align}
\begin{enumerate}
\item This is immediate from the definition of the semi-continuous envelopes in \eqref{eqn: SBME upper semi-continuous envelope} and \eqref{eqn: SBME lower semi-continuous envelope}.
\item If $v$ is an upper semi-continuous function with $u\leq v$, taking the limsup as $y$ tends to $x$ on both sides of the inequality $u(y)\leq v(y)$ and leveraging the upper semi-continuity of $v$ reveals that
$$u^\star(x)=\limsup_{y\to x}u(y)\leq \limsup_{y\to x}v(y)\leq v(x).$$
This implies that
$$u^\star(x)\leq \inf\{v(x)\mid u\leq v \text{ and } v \text{ is upper semi-continuous}\}.$$
To show that this infimum is achieved and that this inequality is in fact an equality, it suffices to prove that $u^\star$ is itself upper semi-continuous. Fix $x\in X$ as well as $\epsilon>0$, and find $r>0$ with
$$u^\star(x)+\epsilon>\sup\big\{u(y)\mid y\in X \text{ with } \norm{y-x}\leq r\big\}.$$
The triangle inequality reveals that for any $z\in X$ with $\norm{z-x}<r$,
$$u^\star(x)+\epsilon\geq \sup\big\{u(y)\mid y\in X \text{ with } \norm{y-z}\leq r-\norm{x-z}\big\}\geq u^\star(z).$$
It follows that $\smash{\limsup_{z\to x}u^\star(z)\leq u^\star(x)}$ so $u^\star$ is upper semi-continuous at $x$. Since $x$ is arbitrary, this establishes the claim.
\item Combining the previous part with \eqref{eqn: SBME lower and upper semi-continuous envelope relationship} shows that
$$u_\star(x)=-(-u)^\star(x)=\max\big\{-v(x)\mid -u\leq v \text{ and } v \text{ is upper semi-continuous}\big\}.$$
Observing that $v$ is upper semi-continuous if and only $-v$ is lower semi-continuous establishes the claim.
\item If $u$ is upper semi-continuous at $x$, then
$$u^\star(x)=\limsup_{y\to x}u(y)\leq u(x).$$
Together with the inequality $u(x)\leq u^\star(x)$, this shows that $u(x)=u^\star(x)$. On the other hand, if $u^\star(x)=u(x)$, then
$$\limsup_{y\to x}u(y)=u^\star(x)=u(x)\leq u(x)$$
so $u$ is upper semi-continuous at $x$.
\item Observe that $u$ is lower semi-continuous at $x\in X$ if and only if $-u$ is upper semi-continuous at $x\in X$. The previous part implies that this is the case if and only if $-u(x)=(-u)^\star(x)$. Invoking \eqref{eqn: SBME lower and upper semi-continuous envelope relationship} completes the proof. \qedhere
\end{enumerate}
\end{proof}

\end{appendix}

\bibliographystyle{abbrv}
\bibliography{sparse_PDE}

\end{document}